\documentclass{article}
\usepackage[utf8]{inputenc}
\usepackage{amsmath,amsthm,amssymb}
\usepackage{xcolor}
\usepackage[shortlabels]{enumitem}
\usepackage{hyperref}
\usepackage{bbm}

\newtheorem{theorem}{Theorem}[section]

\newtheorem{defn}[theorem]{Definition}

\newtheorem{remark}[theorem]{Remark}
\newtheorem{lemma}[theorem]{Lemma}
\newtheorem{prop}[theorem]{Proposition}
\newtheorem{cor}[theorem]{Corollary}

\title{Discretized Radial Projections in $\R^d$}
\author{Kevin Ren}
\date{June 2022}

\newcommand{\R}{\mathbb{R}}
\newcommand{\T}{\mathcal{T}}

\newcommand{\Z}{\mathbb{Z}}
\newcommand{\E}{\mathbb{E}}
\newcommand{\U}{\mathbb{U}}
\newcommand{\gesim}{\gtrsim}
\newcommand{\lesim}{\lesssim}
\newcommand{\leapprox}{\lessapprox}
\newcommand{\leapp}{\lessapprox}
\newcommand{\geapprox}{\gtrapprox}
\newcommand{\geapp}{\gtrapprox}
\newcommand{\app}{\approx}
\newcommand{\eps}{\varepsilon}

\newcommand{\cS}{\mathcal{S}}

\newcommand{\norm}[1]{\| #1 \|}
\newcommand{\va}{\vec{a}}
\newcommand{\vb}{\vec{b}}

\newcommand{\cP}{\mathcal{P}}
\newcommand{\cT}{\mathcal{T}}
\newcommand{\cD}{\mathcal{D}}
\newcommand{\cL}{\mathcal{L}}
\newcommand{\N}{\mathbb{N}}

\newcommand{\bT}{\mathbf{T}}
\newcommand{\oT}{\overline{\T}}
\newcommand{\by}{\mathbf{y}}
\newcommand{\bY}{\mathbf{Y}}
\newcommand{\bp}{\mathbf{p}}
\newcommand{\bx}{\mathbf{x}}
\newcommand{\bX}{\mathbf{X}}
\newcommand{\bz}{\mathbf{z}}
\newcommand{\bZ}{\mathbf{Z}}
\newcommand{\cE}{\mathcal{E}}
\newcommand{\cN}{\mathcal{N}}

\newcommand{\cB}{\mathcal{B}}

\newcommand{\A}{\mathbb{A}}
\newcommand{\bA}{\mathbf{A}}
\newcommand{\bB}{\mathbf{B}}
\newcommand{\bC}{\mathbf{C}}

\newcommand{\bcQ}{\overline{\cQ}}
\newcommand{\PP}{\mathbb{P}}
\newcommand{\cH}{\mathcal{H}}

\newcommand{\sep}{\mathrm{sep}}
\newcommand{\rarrow}{\rightarrow}

\newcommand{\tP}{\tilde{P}}
\newcommand{\obY}{\overline{\bY}}
\newcommand{\tmu}{\tilde{\mu}}
\newcommand{\tnu}{\tilde{\nu}}
\newcommand{\Gr}{\mathrm{Gr}}
\newcommand{\End}{\mathrm{End}}
\newcommand{\spt}{\mathrm{spt}}
\newcommand{\one}{\mathbbm{1}}
\newcommand{\cQ}{\mathcal{Q}}
\newcommand{\oM}{\overline{M}}
\newcommand{\bN}{\mathbf{N}}

\newcommand{\plusP}{\overset{P}{+}}
\newcommand{\Bad}{\mathrm{Bad}}

\newcommand{\odelta}{\overline{\delta}}
\newcommand{\tz}{\tilde{z}}

\begin{document}

\maketitle

\begin{abstract}
    We generalize a Furstenberg-type result of Orponen-Shmerkin to higher dimensions, leading to an $\eps$-improvement in Kaufman's projection theorem for hyperplanes and an unconditional discretized radial projection theorem in the spirit of Orponen-Shmerkin-Wang. Our proof relies on a new incidence estimate for $\delta$-tubes and a quasi-product set of $\delta$-balls in $\R^d$.
\end{abstract}

\tableofcontents







\section{Introduction}
Let $X$ be a set in $\R^n$, and define the radial projection $\pi_x (y) := \frac{y-x}{|y-x|} \in S^{n-1}$. We wish to study the size of radial projections $\pi_x (Y)$ of $Y$, where $x$ is taken in some set $X$. Recently, Orponen, Shmerkin, and Wang \cite{orponen2022kaufman} proved a strong radial projection theorem in two dimensions, but they prove a conditional result in higher dimensions. In this paper, we shall remove the condition $\dim_H (X) \ge k - \frac{1}{k} + \eta(k)$ in higher dimensions, which answers Conjecture 1.5 of \cite{shmerkin2021distance} and improves Theorem 1.9 of \cite{orponen2022kaufman}. We also improve upon the previously known result $\frac{d-1}{d} \min(\dim_H (X), \dim_H (Y)) + \eta(d, \dim_H (X), \dim_H (Y))$ of \cite[Theorem 6.15]{shmerkin2022non}.

\begin{theorem}\label{cor:shm_conj}
    Let $X, Y \subset \R^d$ be Borel sets with $\dim_H (X), \dim_H (Y) \le k$. If $X$ is not contained in a $k$-plane, then
    \begin{equation*}
        \sup_{x \in X} \dim_H (\pi_x (Y \setminus \{ x \})) \ge \min(\dim_H (X), \dim_H (Y)).
    \end{equation*}
\end{theorem}

In fact, we can prove the following slicing result, which improves Proposition 6.8 of \cite{orponen2022kaufman} and makes progress towards answering Conjecture 1.10 of \cite{orponen2022kaufman}.

\begin{cor}\label{cor:shm_conj_2}
Let $s \in (d-2, d]$, then there exists $\eps(s, d) > 0$ such that the following holds. Let $\mu, \nu$ be Borel probability measures on $\R^d$ with disjoint supports that satisfy $\cE_s (\mu), \cE_s (\nu) < \infty$ and $\dim_H (\spt(\nu)) < s + \eps(s, d)$. Further, assume that $\mu, \nu$ don't simultaneously give full measure to any affine $(d-1)$-plane $H \subset \R^d$. Then there exist restrictions of $\mu, \nu$ to subsets of positive measure (which we keep denoting $\mu, \nu$) such that the following holds. For almost every affine 2-plane $W \subset \R^d$ (with respect to the natural measure on the affine Grassmanian), if the sliced measures $\mu_W$, $\nu_W$ on $W$ is non-trivial, then they don't simultaneously give full measure to any line.
In other words,
    \begin{equation*}
        (\gamma_{d,2} \times \mu) \{ (V, x) : \mu_{V,x} (\ell) \nu_{V,x} (\ell) = |\mu_{V,x}| |\nu_{V,x}| > 0 \text{ for some } \ell \in \A(V + x, 1) \} = 0,
    \end{equation*}
where we parametrize affine 2-planes as $V + x$, for $x \in \R^d$ and $V$ in the Grassmannian $\Gr(d, 2)$ with the rotationally invariant Haar measure $\gamma_{d,2}$.
\end{cor}

We also deduce an $\eps$-improvement in Kaufman's projection theorem for hyperplanes. The proof is a standard higher-dimensional generalization of the details in \cite[Section 3.2]{orponen2021hausdorff} and we will omit it. For $\sigma \in S^{n-1}$, let $\pi_\sigma$ be projection in the direction orthogonal to $\sigma$.

\begin{theorem}\label{thm:kaufman}
    For every $k < s < t \le d$, there exists $\eps(s, t)$ such that the following holds.
    Let $E$ be an analytic set in $\R^d$ with $\dim_H (E) = t$. Then
    \begin{equation*}
        \dim_H \{ \sigma \in S^{d-1} : \dim_H (\pi_\sigma (E)) \le s \} \le s - \eps.
    \end{equation*}
\end{theorem}

\begin{remark}
    Kaufman's theorem is sharp when $s = k$ and $t \in (k, k+1]$ because $E$ can be contained within a $(k+1)$-plane.
\end{remark}

We also derive a higher-dimensional version of Beck's theorem (unlike in the discrete setting, the higher-dimensional version cannot proved by projection onto a generic 2D plane). The proof again follows similarly to the 2D version presented in \cite[Corollary 1.4]{orponen2022kaufman}.

\begin{cor}
    Let $X \subset \R^d$ be a Borel set such that $\dim_H (X \setminus H) = \dim_H X$ for all $k$-planes $H$. Then, the line set $\cL(X)$ spanned by pairs of distinct points in $X$ satisfies
    \begin{equation*}
        \dim_H (\cL(X)) \ge \min\{ 2 \dim_H X, 2k \}.
    \end{equation*}
\end{cor}

\subsection{Connections and related work}
Radial projections have also been used to study the Falconer distance set problem, which asks for lower bounds on the Hausdorff dimension of the distance set $\Delta(X) := \{ |x - y| : x, y \in X \}$ given the value of $\dim_H (X)$ for some $X \in \R^d$. In two dimensions, Wolff \cite{wolff1999decay} used Fourier analysis to show that if $\dim_H (X) \ge \frac{4}{3}$, then $\Delta(X)$ has positive Lebesgue measure. Using Orponen's radial projection theorem \cite{orponen2018radial}, Guth-Iosevich-Ou-Wang \cite{guth2020falconer} used a good-bad tube decomposition and decoupling to improve the threshold to $\dim_H (X) \ge \frac{5}{4}$. See also works of Keleti-Shmerkin \cite{keleti2019new} \cite{keleti2019new}, Shmerkin \cite{shmerkin2021improved}, Liu \cite{liu2020hausdorff}, and Stull \cite{stull2022pinned} which provide better lower bounds for $\dim_H (\Delta(X))$ given that $\dim_H (X) \in (1, \frac{5}{4})$. In higher dimensions, the works of Du-Iosevich-Ou-Wang-Zhang \cite{du2021improved} and Wang-Zheng \cite{wang2022improvement} used a good-bad tube decomposition using Orponen's radial projection theorem and decoupling techniques \cite{orponen2018radial} to provide state-of-the-art results when the dimension $d$ is even; when $d$ is odd, a more classical approach purely based on decoupling gave the best estimates \cite{du2019sharp}, \cite{harris2021low}. More recently, Shmerkin and Wang \cite{shmerkin2022dimensions} prove a radial projection theorem in the spirit of this paper to provide an improved lower bound when $\dim_H (X) = \frac{d}{2}$, $d = 2, 3$; using their framework combined with updated results of \cite{orponen2022kaufman}, one can show for example that $\dim_H (\Delta(X)) \ge \frac{5}{8}$ when $X \subset \R^3$ satisfies $\dim_H (X) = \frac{3}{2}$. In fact, all of these works prove lower bounds on the size of the pinned distance set, $\Delta_x (X) := \{ |x - y| : y \in X \}$. In the forthcoming companion papers \cite{DOKZFalconer}, \cite{DOKZFalconerDec}, we use Theorem \ref{cor:shm_conj} to improve the lower bounds for the Falconer distance set problem in all dimensions $d$.

Very recently, radial projections in dimension $2$ have been used to prove the ABC sum-product conjecture and Furstenberg set conjecture, and yield progress on the discretized sum-product problem \cite{orponen2023projections}, \cite{ren2023furstenberg}. It is natural to wonder whether the exciting progress in 2 dimensions will generalize to higher dimensions. The starting point of the breakthrough work of \cite{orponen2023projections} (which was also used in \cite{ren2023furstenberg}) is a sharp radial projection theorem in 2 dimensions, \cite[Theorem 1.1]{orponen2022kaufman}. We hope to use our higher dimensional radial projection theorem to prove analogous results to \cite{orponen2023projections}, \cite{ren2023furstenberg} in all dimensions.


\subsection{Discretized results}
We deduce Theorem \ref{cor:shm_conj} from $\delta$-discretized versions. The following notation will be used throughout this paper.

\begin{defn}
    Let $P \subset \R^d$ be a bounded nonempty set, $d \ge 1$. Let $\delta > 0$ be a dyadic number, and let $0 \le s \le d$ and $C > 0$. We say that $P$ is a $(\delta, s, C, k)$-set if for every $(r, k)$-plate $H$ with $r \in [\delta, 1]$, we have
    \begin{equation*}
        |P \cap H|_\delta \le C \cdot |P|_\delta \cdot r^s.
    \end{equation*}
    If $k$ is not specified, we default to $k = 0$ (which becomes a standard definition from \cite{orponen2021hausdorff} because $(r, 0)$-plates are $r$-balls).
\end{defn}

\begin{defn}
    Let $\T \subset \R^d$ be a bounded nonempty set of dyadic $\delta$-tubes, $d \ge 2$. Let $\delta > 0$ be a dyadic number, and let $0 \le s \le d$, $0 \le k \le d-2$, and $C > 0$. We say that $\T$ is a $(\delta, s, C, k)$-set of tubes if for every $(r, k+1)$-plate $H$ and $\delta \le r \le 1$, we have
    \begin{equation}\label{eqn:tube non-concentration}
        |\T \cap H| \le C \cdot |\T| \cdot r^s.
    \end{equation}
    If $k$ is not specified, we default to $k = 0$. We also say $\T$ is a $(\delta, s, C, k)$-set of tubes from scale $r_1$ to $r_2$ if the non-concentration condition \eqref{eqn:tube non-concentration} holds for $r_2 \le r \le r_1$.
\end{defn}

A $(\delta, s, C, k)$-set of balls cannot be concentrated in a small neighborhood of a $k$-plane, while a $(\delta, s, C, k)$-set of tubes cannot be concentrated in a small neighborhood of a $(k+1)$-plane.

The main ingredient in the proof of Theorem \ref{cor:shm_conj} is an $\eps$-improvement to the (dual) Furstenberg set problem that generalizes Theorem 1.3 in \cite{orponen2021hausdorff} to higher dimensions.
\begin{theorem}\label{thm:main}
    For any $0 \le k < d-1$, $0 \le s < k+1$, $s < t \le d$, $\kappa > 0$, there exists $\eps(s, t, \kappa, k, d) > 0$ such that the following holds for all small enough $\delta \in 2^{-\N}$, depending only on $s, t, \kappa, k, d$. Let $\cP \subset \cD_\delta$ be a $(\delta, t, \delta^{-\eps})$-set with $\cup \cP \subset [0, 1)^d$, and let $\cT \subset \cT^\delta$ be a family of $\delta$-tubes. Assume that for every $p \in \cP$, there exists a $(\delta, s, \delta^{-\eps}, 0)$ and $(\delta, \kappa, \delta^{-\eps}, k)$-set $\cT(p) \subset \cT$ such that $T \cap p \neq \emptyset$ for all $T \in \cT(p)$. Then $|\cT| \ge \delta^{-2s-\eps}$.
\end{theorem}

\begin{remark}\label{rmk:s > k}
    The condition of $\T(p)$ being a $(\delta, \kappa, \delta^{-\eps}, k+1)$-set is to prevent the counterexample in (say) $\R^3$ when $s = 1, t \in (1, 2]$, and $\T$ is a maximal set of $\delta^{-2}$ many essentially distinct tubes in $[0, 1]^2$. This condition is automatically taken care of when $s > k$: any $(\delta, s, \delta^{-\eps}, 1)$-set is a $(\delta, \kappa, \delta^{-\eps}, k+1)$-set with $\kappa = s-k$.
\end{remark}

\begin{remark}
    We can make this decay around $k$-plane assumption assuming that (1) $P$ is a $(\delta, \kappa, \delta^{-\eps}, k+1)$-set and (2) for $p \in P$, $|\T(p) \cap P| \ge \delta^\eps |P|$. This will be useful for radial projection estimates, since we can guarantee (1) by Theorem B.1 of \cite{shmerkin2022non} and (2) because we can get rid of $\lesim \delta^{\eps} |P|$ many pairs $(p, q)$ for a fixed $p$.
\end{remark}

In fact, we can prove the following refined version of Theorem \ref{thm:main}.

\begin{theorem}\label{thm:main_refined}
    For any $0 \le k < d-1$, $0 \le s < k+1$, $\max(s, k) < t \le d$, $\kappa > 0$, $r_0 \le 1$, there exists $\eps(s, t, \kappa, k, d) > 0$ such that the following holds for all small enough $\delta/r_0 \in 2^{-\N} \cap (0, \delta_0)$, with $\delta_0$ depending only on $s, t, \kappa, k, d$. Let $H$ be a $(r_0, k+1)$-plate, $\cP \subset \cD_\delta \cap H$ be a $(\delta, t, (\delta/r_0)^{-\eps})$-set with $\cup \cP \subset [0, 1)^d$, and let $\cT \subset \cT^\delta \cap H$ be a family of $\delta$-tubes. Assume that for every $p \in \cP$, there exists a set $\T(p) \subset \T$ such that:
    \begin{itemize}
        \item $T \cap p \neq \emptyset$ for all $T \in \cT(p)$;

        \item $\cT(p)$ is a $(\delta, s, (\delta/r_0)^{-\eps} r_0^{k-s}, 0)$-set down from scale $r$;

        \item $\cT(p)$ is a $(\delta, \kappa, (\delta/r_0)^{-\eps} r_0^{-\kappa}, k)$-set.
    \end{itemize}
    Then $|\cT| \ge (\frac{\delta}{r_0})^{-\eps} \delta^{-2s} r_0^{2(s-k)}$.
\end{theorem}

\begin{remark}\label{rmk:uniform eps}
    (a) Given fixed $k, \kappa$, the value of $\eps$ can be chosen uniformly in a compact subset of $\{ (s, t) : 0 \le s < k+1, \max(s, k) < t \le d \}$. Indeed, if $\eps > 0$ works for $(s, t)$, then $\frac{\eps}{2}$ works in the $\frac{\eps}{2}$-neighborhood of $(s, t)$.

    (b) Conjecture: can we replace the condition of being in $H$ by $\T(p)$ being a $(\delta, k, (\delta/r_0)^{-\eps}, 0)$-set from scales $1$ to $r_0$?
\end{remark}

Using Theorem \ref{thm:main_refined}, a bootstrap argument based on \cite{orponen2022kaufman} gives the following.




\begin{theorem}\label{thm:intermediate}
Let $k \in \{ 1, 2, \cdots, d-1 \}$, $k-1 < \sigma < s \le k$, and $\eps > 0$. There exist $N, K_0$ depending on $\sigma, s, k$, and $\eta(\eps) > 0$ (with $\eta(1) = 1$) such that the following holds. Fix $r_0 \le 1$, and $K \ge K_0$. Let $\mu, \nu$ be $\sim 1$-separated $s$-dimensional measures with constant $C_\mu, C_\nu$ supported on $E_1, E_2$, which lie in $B(0, 1)$. Assume that $|\mu|, |\nu| \le 1$. Let $A$ be the pairs of $(x, y) \in E_1 \times E_2$ that lie in some $K^{-1}$-concentrated $(r_0, k)$-plate. Then there exists a set $B \subset E_1 \times E_2$ with $\mu \times \nu (B) \lesim K^{-\eta}$ such that for every $x \in E_1$ and $r$-tube $T$ through $x$, we have
\begin{equation*}
    \nu(T \setminus (A|_x \cup B|_x)) \lesim \frac{r^\sigma}{r_0^{\sigma-(k-1)+N\eps}} K^N.
\end{equation*}
The implicit constant may depend on $C_\mu, C_\nu, \sigma, s, k$.
\end{theorem}




\begin{remark}
    (a) It is not assumed that $\mu, \nu$ are a probability measures, just that $\mu(B(0,1)), \nu(B(0,1)) \le 1$.

    (b) If $\alpha>d-1$, then the numerology of Theorem \ref{thm:intermediate} doesn't apply. Instead, Orponen's radial projection theorem \cite{orponen2018radial} in dimension $d$ applies. The result (stated in \cite[Lemma 3.6]{guth2020falconer} for $d = 2$, but can be generalized to all dimensions $d$) is that for $\gamma = \eps/C$, there exists a set $B \subset E_1 \times E_2$ with $\mu_1 \times \mu_2 (B) \le r^{\gamma}$ such that for every $x \in E_1$ and $\delta$-tube $T$ through $x$, we have
\begin{equation*}
    \mu_2 (T \setminus B|_x) \lesim r^{d-1-\eps}.
\end{equation*}
Note that the set $A$ of ``concentrated pairs'' is not needed here.
    
    (c) If $r \sim r_0$, we can obtain a slightly better result by projecting to a generic $k$-dimensional subspace and following the argument in \cite[Section 3.2]{du2021improved}. The result is that for $\gamma = \eps/C$, there exists a set $B \subset E_1 \times E_2$ with $\mu_1 \times \mu_2 (B) \le \delta^{\gamma}$ such that for every $x \in E_1$ and $r$-tube $T$ through $x$, we have
\begin{equation*}
    \mu_2 (T \setminus B|_x) \lesim r^{k-1-\eps}.
\end{equation*}
    The set $A$ is again not needed in this case. The main novelty of Theorem \ref{thm:intermediate} comes when $r < r_0$.
\end{remark}




\subsection{Proof ideas}
The main proof ideas for Theorem \ref{thm:main} are as follows:
\begin{enumerate}
    \item Perform a standard multiscale decomposition argument due to \cite{orponen2021hausdorff} to reduce the original problem to two building blocks: the case when $\cP$ is a $(\delta, s)$-set and when $\cP$ is a $t$-regular set. The first case doesn't happen all the time and has no loss by an elementary incidence argument, so we focus on gaining an $\eps$-improvement in the second case. A $t$-regular set $\cP$ has the special property that $|\cP \cap Q|$ is still a $(\Delta, t)$-set for $Q \in \cD_{\Delta} (\cP)$, $\Delta = \delta^{1/2}$.

    \item If $\cP$ is $t$-regular with $\Delta = \delta^{1/2}$, we may find a $\Delta$-tube $\bT$ such that upon dilation of $\bT$ to $[0, 1]^d$, we obtain a new Furstenberg problem with the ball set having a quasi-product structure. See Appendix A of \cite{orponen2021hausdorff}.
    
    \item Finally, we will use discretized sum-product type arguments to conclude an $\eps$-improvement to the dual Furstenberg problem assuming $\cP = X \times Y \subset \R^{d-1} \times \R$ has a quasi-product structure. In very rough terms, we shall lift $Y$ to have dimension close to $1$, and apply multi-linear Kakeya. This idea of lifting the dimension was found in He's work on a higher-rank discretized sum-product theorem \cite{he2016discretized} in a slightly different context.
\end{enumerate}

To prove Theorem \ref{thm:main_refined}, we use a similar multiscale decomposition argument as in (1) to reduce to two building blocks: a smaller version of the setting of Theorem \ref{thm:main_refined} and a smaller version of Theorem \ref{thm:main}. The smaller version of Theorem \ref{thm:main_refined} has no loss by an elementary incidence argument, and the smaller version of Theorem \ref{thm:main} admits a gain.

For Theorem \ref{thm:intermediate}, we first prove the case when $\mu, \nu$ are supported in a $r_0 K$ plate (where $K$ is a small power of $r_0^{-1}$). This uses a similar argument as in \cite[Lemma 2.8]{orponen2022kaufman}. The general case follows from applying this special case many times.





\subsection{Structure of the paper}
In Section \ref{sec:prelims}, we introduce some key concepts that will be used throughout the paper. In Sections \ref{sec:quasi-product sets} through \ref{sec:main thm general}, we prove Theorem \ref{thm:main} first for quasi-product sets following ideas of \cite{he2020orthogonal}, and then for regular sets and finally for general sets following \cite{orponen2021hausdorff}. In Section \ref{sec:refined from main}, we prove Theorem \ref{thm:main_refined} from Theorem \ref{thm:main}. In Section \ref{sec:power decay}, we generalize a radial projection theorem of Shmerkin \cite[Theorem 6.3]{shmerkin2022non} that enables us to assume our sets have power decay around $k$-planes. In Section \ref{sec:threshold}, we prove Theorem \ref{thm:threshold_refined} following ideas from \cite{orponen2022kaufman}. Finally, in Section \ref{sec:corollaries}, we prove Theorem \ref{cor:shm_conj} and \ref{cor:shm_conj_2} from the discretized results.

\textbf{Acknowledgments.} The author is supported by a NSF GRFP fellowship. The author would like to thank Xiumin Du, Tuomas Orponen, Yumeng Ou, Pablo Shmerkin, Hong Wang, and Ruixiang Zhang for helpful discussions. We thank Paige Bright and Yuqiu Fu for suggesting to include a higher-dimensional version of Beck's theorem in this paper.

\section{Preliminaries}\label{sec:prelims}
This section will summarize the argument of \cite{orponen2021hausdorff}, and in lieu of proofs (with the exception of Proposition \ref{prop:easy_est}), we either refer the reader to \cite{orponen2021hausdorff} or defer the proof to a later section.
\subsection{Definitions}
We use $A \lesim B$ to denote $A \le CB$ for some constant $C$. We use $A \lesim_N B$ to indicate the constant $C$ can depend on $N$. We will also use $A \leapp B$ in future proofs; its exact meaning will always be clarified when used.

For a finite set $A$, let $|A|$ denote the cardinality of $A$. If $A$ is infinite, let $|A|$ denote the Lebesgue measure of $A$.

For a set $A$, let $A^c = \R^d \setminus A$.

For a tube $T$, let $\ell(T)$ denote the central line segment of $T$.

For a set $E$, let $E^{(\delta)}$ be the $\delta$-neighborhood of $E$.

For $A \subset X \times Y$ and $x \in X$, define the slice $A|_x = \{ y \in Y : (x, y) \in A \}$ and $A|^y = \{ x \in X : (x, y) \in A \}$.

For a measure $\mu$ and a set $G$, define the restricted measure $\mu|_G$ by $\mu|_G (A) = \mu(G \cap A)$. The renormalized restricted measure is $\mu_G = \frac{1}{\mu(G)} \mu|_G$.

For vectors $v_1, \cdots, v_i \in \R^d$, $1 \le i \le d$, the quantity $|v_1 \wedge \cdots \wedge v_i|$ is the non-negative volume of the parallelepiped spanned by $v_1$ through $v_i$.


$B(x, r)$ is the ball in $\R^d$ of radius $r$ centered at $x$. We also use the notation $B_r$ for an arbitrary $r$-ball in $\R^d$.

For sets $A, B$ and $P \subset A \times B$, let $A \plusP B := \{ a + b : (a, b) \in P \}$.

\begin{defn}
    We say $\mu$ supported in $\R^d$ is an $\alpha$-dimensional measure with constant $C_\mu$ if $\mu(B_r) \le C_\mu r^\alpha$ for all $r \le 1$ and balls $B_r$ of radius $r$.
\end{defn}

\subsection{Plates}\label{subsec:r-net}
We work in $\R^d$. An $(r, k)$-plate is the $r$-neighborhood of a $k$-dimensional hyperplane in $\R^d$. We construct a set $\cE_{r,k}$ of $(r, k)$-plates with the following properties:
\begin{itemize}
    \item Each $(\frac{r}{2}, k)$-plate intersecting $B(0, 1)$ lies in at least one plate of $\cE_{r,k}$;

    \item For $s \ge r$, every $(s, k)$-plate contains $\lesim \left( \frac{s}{r} \right)^{(k+1)(d-k)}$ many $(r, k)$-plates of $\cE_{r,k}$.
\end{itemize}

For example, when $k = 1$ and $d = 2$, we can simply pick $\sim r^{-1}$ many $r$-tubes in each of an $r$-net of directions. This generalizes to higher $k$ and $d$ via a standard $r$-net argument, but we haven't seen it in the literature, so we provide a precise construction.

An $r$-net of a metric space is a subset $S$ such that $B(x, r) \cap B(y, r) = \emptyset$ for $x \neq y, x, y \in S$. The affine Grassmanian manifold $\A(k, d)$ is the set of all $k$-planes in $\R^d$. By counting degrees of freedom, we see that $\dim \A(k, d) = (k+1)(d-k)$. Any such plane is uniquely $V = V_0 + a$ for some $k$-dimensional subspace $V_0$ and $a \in V_0^\perp$. For $V=V_0 + a$ and $W=W_0 + b$, define their distance $d_\A$ to be (following Section 3.16 of \cite{mattila1999geometry}):
\begin{equation*}
    d_\A (V, W) = \norm{\pi_{V_0} - \pi_{W_0}}_{op} + |a - b|,
\end{equation*}
where $\pi_{V_0}:\R^d\to V_0$ and $\pi_{W_0}:\R^d\to W_0$ are orthogonal projections, and $\|\cdot\|_{op}$ is the usual operator norm for linear maps. Let $\A_0 (k, d)$ be the submanifold of $k$-planes $V_0 + a$ with $a \in B(0, 10)$. Since the manifold $(\A_0 (k, d), d_\A)$ is compact and smooth, it can be covered by finitely many charts that are $\sim 1$-bilipschitz to a subset of $\R^{(k+1)(d-k)}$.

From a maximal $cr$-net $\cN$ of the set of affine planes of $\A_0 (k, d)$ with $c > 0$ a sufficiently small constant, we can construct a set $\cE_{r,k}$ of $(r, k)$-plates whose central planes are the elements of $\cN$. We now check the two properties for $\cE_{r,k}$.

To prove the first property, let $H$ be a $(\frac{r}{2}, k)$-plate intersecting $B(0, 1)$. Then the central plane $P = P_H$ must lie at distance $\le 2cr$ from some element $Q$ of $\cN$ (otherwise, we can add it to the net). Let $P = P_0 + a$ and $Q = Q_0 + b$. Hence, $\norm{\pi_{P_0} - \pi_{Q_0}}_{op} \le 2cr$ and $|a - b| \le 2cr$, so for $x \in P \cap B(0, 10)$ (so $x-a \in P_0$),
\begin{equation*}
    |\pi_{Q_0} (x-a) - (x-a)| \le 2cr |x-a| \le 2cr(|x| + |a|) \le 40cr.
\end{equation*}
Now, note that $\pi_{Q_0} (x-a) + b \in Q$. It is close to $x$ if $c < \frac{1}{100}$:
\begin{equation*}
    |\pi_{Q_0} (x-a) + b - x| \le 40cr + |a-b| \le 50cr < \frac{r}{2}.
\end{equation*}
We have proved $P \cap B(0, 10) \subset Q^{(r/2)}$ and thus $P^{(r/2)} \cap B(0, 10) \subset Q^{(r)}$. Hence, $H$ is contained in the $(r, k)$-plate with central plane $Q$.

To prove the second property, we note that the set of $k$-planes in $\A(k, d)$ whose intersection with $B(0, 10)$ is contained in a given $(s, k)$-plate is contained in an $O(s)$-ball $B$ of $\A(k, d)$. First suppose $B$ is contained within some coordinate chart; we would like to prove that $|\cN \cap B| \lesim \left( \frac{s}{r} \right)^{(k+1)(d-k)}$. To show this, note that $\{ B(x, r) : x \in \cN \cap B \}$ is a packing of $B^{(r)}$ with finitely overlapping $r$-balls. Now map the chart to $\R^{(k+1)(d-k)}$. Since the map only distorts distances by a constant factor, we can pack $|\cN \cap B|$ many finitely overlapping $c_1 r$-balls into a ball of radius $O(s)$. Thus by a volume argument, we have $|\cN \cap B| \lesim \left( \frac{s}{r} \right)^{(k+1)(d-k)}$. Since there are finitely many charts, we can apply the argument to $B$ intersecting each chart, which proves the second property.


We specialize our discussion to tubes. For each scale $\delta$, let $\T^\delta$ be a cover of $[0, 1]^d$ with $\delta$-tubes such that every $\frac{\delta}{2}$-tube (and in particular every $r$-tube with $r < \frac{\delta}{2}$) is contained in at least $1$ and at most $C_d$ many tubes of $\T^\delta$. Slightly abusing notation (\'a la \cite{orponen2021hausdorff}), we will also use $\T, \T_\delta, \T_\Delta$ to represent sets of tubes, where the subscript $\delta$ helpfully indicates a set of $\delta$-tubes.

In Theorem \ref{thm:intermediate}, we pay attention to certain plates with disproportionately much mass.

\begin{defn}
We say that a $(r, k)$-plate $H$ is $c$-concentrated on $\mu$ if $\mu(H) \ge c$.
\end{defn}

Other notation is following \cite{orponen2021hausdorff}. Unlike \cite{orponen2021hausdorff}, we work with ordinary rather than dyadic tubes. The advantage of dyadic tubes is that every $2^{-n}$-tube is in a unique $2^{-m}$-tube if $n > m$; thus, dyadic tubes will avoid the $C_d$ loss incurred by the finitely overlapping cover $\T^\delta$. However, dyadic tubes have the disadvantage that they don't behave well under rotations or dilations, and it would be more cumbersome to define $(\delta, s, C, k)$-sets of dyadic tubes (whereas the definition for ordinary tubes is more geometric). Thus, in principle it is possible to work with dyadic tubes and save on the $C_d$ loss, but it doesn't affect our numerology in the end (since our losses will depend badly on $d$ anyway), so we chose to work with ordinary tubes throughout. 

\begin{defn}\cite{orponen2021hausdorff}
    Let $P \subset \R^d$ be a bounded nonempty set, $d \ge 1$. Let $\delta > 0$ be a dyadic number, and let $0 \le s \le d$ and $C > 0$. We say that $P$ is a $(\delta, s, C)$-set if
    \begin{equation*}
        |P \cap Q|_\delta \le C \cdot |P|_\delta \cdot r^s, \qquad Q \in \cD_r (\R^d), r \in [\delta, 1].
    \end{equation*}
\end{defn}

\begin{defn}
    Let $\T \subset \R^d$ be a bounded nonempty set of dyadic $\delta$-tubes, $d \ge 2$. Let $\delta > 0$ be a dyadic number, and let $0 \le s \le d$, $0 \le k \le d-2$, and $C > 0$. We say that $\T$ is a $(\delta, s, C, k)$-set of tubes if for every $(r, k+1)$-plate $H$ and $\delta \le r \le 1$, we have
    \begin{equation*}
        |\T \cap H| \le C \cdot |\T| \cdot r^s.
    \end{equation*}
    If $k$ is not specified, we default to $k = 0$.
\end{defn}

The following is a simpler interpretation of $(\delta, s, C, k)$-set if the tubes all pass through the same point.

\begin{defn}
    Let $\sigma(t) \in S^{d-1}$ be the slope of the central axis of $t$.
\end{defn}

\begin{lemma}\label{lem:slope set dim}
    Let $\T$ be a set of $\delta$-tubes intersecting $p$. Then if $\T$ is a $(\delta, s, C, k)$-set, then $\sigma(\T)$ is a $(\delta, s, O(C), k)$-set. Conversely, if $\sigma(\T)$ is a $(\delta, s, C, k)$-set, then $\T$ is a $(\delta, s, O(C), k)$-set.
\end{lemma}

\begin{proof}
    Let $\pi_p : \R^d \to S^{d-1}$ denote spherical projection through $p$. Then $\pi_p (t \setminus B(p, 1/2))$ is well-defined and equals $\sigma(t)$, up to an additive loss of $C\delta$. Fix a $(r, k)$-plate $H \in S^{d-1}$. Then the set of tubes with slope in $H$ and passing through $p$ must lie in a $(r + C\delta, k+1)$-plate $p^{(C\delta)} + \pi_p^{-1} (H)$. Conversely, for any $(r, k+1)$-plate $W$ containing $p$, the set of possible slopes of tubes through $p$ contained in $W$ is contained in a $(r + C\delta, k)$-plate $(\pi_p (W - p))^{C\delta}$.
\end{proof}

We will need the following lemma from \cite{orponen2021hausdorff}.
\begin{lemma}[\cite{orponen2021hausdorff}, Lemma 2.7]\label{lem:small sets}
    Let $P \subset [-2, 2]^d$ be a $(\delta, s, C)$-set. Then $P$ contains a $\delta$-separated $(\delta, s, O_d (C))$-subset $P'$ with $|P'| \le \delta^{-s}$.
\end{lemma}

First, since $(\delta, \kappa, \delta^{-\eps}, k)$-sets are $(\delta, \kappa, \delta^{-\eps}, k')$-sets for any $k' < k$, we can assume that $k \le s < k+1$. Next, since $(\delta, t, \delta^{-\eps})$-sets are $(\delta, t', \delta^{-\eps})$-sets for $t' < t$, we may assume $t \le k+1$. In particular, we get $t - s \le 1$, a useful assumption.

We record a useful geometric fact about $(r, k)$-plates.

\begin{lemma}\label{lem:concentrated_points}
    Fix $C_\sep \ge 1$, then there exists $r_0$ depending on $C_\sep$ such that the following is true for $r < r_0$. If $(x, y)$ lie in an $(r, k)$-plate $H$ and $|x-y| = C_\sep^{-1}$, then any $r$-tube $T$ through $x, y$ will lie in $H^{(CC_\sep r)}$, which is a $(CC_\sep r, k)$-plate.
\end{lemma}

\begin{proof}
    For $C$ sufficiently large: If $T$ does not lie in $H^{(CC_\sep r)}$, then $H \cap T$ will be contained in a $(2C_\sep)^{-1}$-tube segment of $T$.
\end{proof}

\subsection{An Elementary Estimate}
We prove a classical estimate which can be viewed as Theorem \ref{thm:main} with $\eps = 0$. We won't need the fact that $\T(p)$ is a $(\delta, \kappa, \delta^{-\eps}, k)$-set. The $d = 2$ case is proven as Proposition 2.13 and Corollary 2.14 of \cite{orponen2021hausdorff}. For higher dimensions, the proof is similar and we sketch the details. Let $A \leapp_\delta B$ denote the inequality
\begin{equation*}
    A \le C \cdot \log (\frac{1}{\delta})^C B.
\end{equation*}

\begin{prop}\label{prop:easy_est}
    Let $0 \le s \le t \le d-1$, and let $C_P, C_T \ge 1$. Let $\cP \subset \cD_\delta$ be a $(\delta, t, C_P)$-set. Assume that for every $p \in \cP$ there exists a $(\delta, s, C_T)$-family $\T(p) \subset \T^\delta$ of dyadic $\delta$-tubes with the property that $T \cap p \neq \emptyset$ for all $T \in \T(p)$, and $|\T(p)| = M$ for some $M \ge 1$.

    Let $\T \subset \T^\delta$ be arbitrary, and define $I(\cP, \T) = \{ (p, T) \in \cP \times \T : T \in \T(p) \}$. Then
    \begin{equation*}
        |I(\cP, \T)| \leapprox_\delta \sqrt{C_P C_T} \cdot (M\delta^s)^{\theta/2} \cdot |\T|^{1/2} |\cP|,
    \end{equation*}
    where $\theta = \theta(s, t) = \frac{d-1-t}{d-1-s} \in [0,1]$. (If $s=t=d-1$, then $\theta(s,t) = 0$.)
\end{prop}

The following corollary of Proposition \ref{prop:easy_est} is the form we will use.

\begin{cor}\label{cor:easy_est}
    Let $0 \le s \le t \le d-1$, and let $C_P, C_T \ge 1$. Let $\cP \subset \cD_\delta$ be a $(\delta, t, C_P)$-set. Assume that for every $p \in \cP$ there exists a $(\delta, s, C_T)$-family $\T(p) \subset \T^\delta$ of dyadic $\delta$-tubes with the property that $T \cap p \neq \emptyset$ for all $T \in \T(p)$, and $|\T(p)| = M$ for some $M \ge 1$. If $\T = \cup_{p \in \cP} \T(p)$, then
    \begin{equation*}
        |\T| \geapprox_\delta (C_P C_T)^{-1} \cdot M\delta^{-s} \cdot (M\delta^s)^{\frac{t-s}{d-1-s}}.
    \end{equation*}
    (If $s = t = d-1$, then $\frac{t-s}{d-1-s} = 0$.)
\end{cor}

\begin{remark}
    To use Corollary \ref{cor:easy_est}, we need $t \le d-1$. Fortunately, this is a harmless assumption because $s < d-1$, and changing $t$ to $\min(t, d-1)$ makes the hypothesis of Theorem \ref{thm:main} weaker.
\end{remark}

\begin{proof}
    We begin with an application of Cauchy-Schwarz.
\begin{align*} |I(\mathcal{P},\mathcal{T})| &= \sum_{T \in \mathcal{T}} |\{p \in \mathcal{P} : T \in \mathcal{T}(p)\}| \\
&\leq |\mathcal{T}|^{1/2} \left|  \{(T,P,P'): T\in\mathcal{T}(p)\cap \mathcal{T}(p')  \}\right|^{1/2}.
\end{align*}
Note that we have the following bounds:
\begin{equation}\label{form94} |\mathcal{T}(p) \cap \mathcal{T}(p')| \lesssim \min \left\{C_{T} \cdot M \cdot \left(\tfrac{\delta}{d(p,p')+\delta} \right)^{s}, \left( \tfrac{1}{d(p,p')+\delta} \right)^{d-1} \right\}, \end{equation}
where $d(p,p')$ stands for the distance of the midpoints of $p$ and $p'$. To prove \eqref{form94}, observe that if $T \in \mathcal{T}(p) \cap \mathcal{T}(p')$, then $T$ lies in a $\frac{\delta}{d(p,p')+\delta}$-tube with central line being the line between $p$ and $p'$. Thus, the first bound in \eqref{form94} follows from $\T(p)$ being a $(\delta, s, C_T)$-set with $|\T(p)| = M$, and the second bound is the maximum number of essentially distinct $\delta$-tubes that can fit inside a $\frac{\delta}{d(p,p')+\delta}$-tube.

Write $\theta := \theta(s,t) := \frac{(d-1) - t}{(d-1) - s} \in [0,1]$. (If $s = t = d-1$, we set $\theta := 0$.) The parameter $\theta$ is chosen so that $t = s\theta+(d-1)(1-\theta)$. Then \eqref{form94} and the inequality $\min\{a,b\} \leq a^{\theta}b^{1 - \theta}$ imply that
\begin{displaymath} |\mathcal{T}(p) \cap \mathcal{T}(p')| \lesssim (C_{T}M\delta^{s})^{\theta} \cdot d(p,p')^{-t}. \end{displaymath}
Since $\cP$ is a $(\delta,t,C_{P})$-set, for fixed $p\in\mathcal{P}$ we have
\begin{displaymath}
\sum_{p'} (d(p,p')+\delta)^{-t} \lesssim \sum_{ \sqrt{2}\cdot\delta \le 2^{-j} \le \sqrt{2} } 2^{t j}|\{p'\in \mathcal{P}: d(p,p') \le 2^{-j}\}| \lessapprox_{\delta} C_{P}\cdot |\mathcal{P}|.
\end{displaymath}
We deduce that
\begin{displaymath} \sum_{p, p'} |\mathcal{T}(p) \cap \mathcal{T}(p')| \lesssim (C_{T}M\delta^{s})^{\theta} \sum_{p, p'} (d(p,p')+\delta)^{-t} \lessapprox_{\delta} C_{P}(C_{T}M\delta^{s})^{\theta} \cdot |\mathcal{P}|^{2},\end{displaymath}
so
\begin{displaymath} |I(\mathcal{P},\mathcal{T})| \lessapprox_{\delta} C_{P}^{1/2}(C_{T}M\delta^{s})^{\theta/2} \cdot  |\mathcal{T}|^{1/2}|\mathcal{P}| \leq \sqrt{C_{P}C_{T}} \cdot (M\delta^{s})^{\theta/2} \cdot  |\mathcal{T}|^{1/2}|\mathcal{P}|. \end{displaymath}
This proves Proposition \ref{prop:easy_est}, and Corollary \ref{cor:easy_est} follows by observing $|I(\cP, \T)| \ge M |\cP|$.
\end{proof}


\subsection{Multiscale analysis}
Following Section 4 of \cite{orponen2021hausdorff}, we would like to change scale from $\delta$ to $\Delta > \delta$, while preserving the properties of $\T(p)$. We say $A \leapprox_\delta B$ if there exists an absolute constant $C \ge 1$ such that $A \le C \cdot [\log (1/\delta)]^C$. We start by naming the objects in Theorem \ref{thm:main}.

\begin{defn}
    Fix $\delta \in 2^{-\N}, s \in [0, d-1], C > 0, M \in \N$. We say that a pair $(\cP_0, \T_0) \subset \cD_\delta \times \T^\delta$ is a $(\delta, s, C_1, \kappa, C_2, M)$-nice configuration if for every $p \in \cP_0$, there exists a $(\delta, s, C_1, 0)$ and $(\delta, \kappa, C_2, k)$-set $\T(p) \subset \T_0$ with $|\T(p)| = M$ and such that $T \cap p \neq \emptyset$ for all $T \in \T(p)$.
\end{defn}

Using the method of induction on scales, we would like to relate nice configurations at scale $\delta$ to nice configurations at scales $\Delta, \frac{\delta}{\Delta}$, where $\delta < \Delta \le 1$. The following proposition, which combines Propositions 4.1 and 5.2 of \cite{orponen2021hausdorff}, gives a way of doing so with only polylog losses. Our proof relies on the same ideas as \cite{orponen2021hausdorff}, with some technical simplifications. We defer the proof to Section \ref{sec:refined from main}, where we prove a slightly more general version.

\begin{prop}\label{prop:nice_tubes}
    Fix dyadic numbers $0 < \delta < \Delta \le 1$. Let $(\cP_0, \T_0)$ be a $(\delta, s, C_1, \kappa, C_2, M)$-nice configuration. Then there exist sets $\cP \subset \cP_0$, $\T(p) \subset \T_0 (p), p \in \cP$, and $\T_\Delta \subset \T^\Delta$ such that denoting $\T = \cup_{p \in \cP} \T(p)$ the following hold:
    \begin{enumerate}[(i)]
        \item\label{item1} $|\cD_\Delta (\cP)| \approx_\delta |\cD_\Delta (\cP_0)|$ and $|\cP \cap Q| \approx_\delta |\cP_0 \cap Q|$ for all $Q \in \cD_\Delta (\cP)$.


        \item\label{item21} There exists $\bN$ such that $|\T \cap \bT| \sim \bN$ for all $\bT \in \T_\Delta$.

        \item\label{item3} $(\cD_\Delta (\cP), \T_\Delta)$ is $(\Delta, s, C^1_\Delta, \kappa, C^2_\Delta, M_\Delta)$-nice for some $C^1_\Delta \approx_\delta C_1$, $C^2_\Delta \approx_\delta C_2$, and $M_\Delta \ge 1$.

        \item\label{item4} For each $Q \in \cD_\Delta (\cP)$, let $\T_\Delta (Q)$ be the tubes in $\T_\Delta$ through $Q$. Then for all $\bT \in \T_\Delta (Q)$, we have
        \begin{equation*}
            |\{ (p, T) \in (\cP \cap Q) \times \T : T \in \T(p) \text{ and } T \subset \bT \} | \geapp_\delta \frac{M \cdot |\cP \cap Q|}{|\T_\Delta (Q)|}.
        \end{equation*}

        \item\label{item5} For each $Q \in \cD_\Delta (\cP)$, there exist $C^1_Q \approx_\delta C_1$, $C^2_Q \approx_\delta C_2$, $M_Q \ge 1$, a subset $\cP_Q \subset \cP \cap Q$ with $|\cP_Q| \geapp_\Delta |\cP \cap Q|$, and a family of tubes $\T_Q \subset \T^{\delta/\Delta}$ such that $(S_Q (\cP_Q), \T_Q)$ is $(\delta/\Delta, s, C^1_Q, \kappa, C^2_Q, M_Q)$-nice.
    \end{enumerate}

    Furthermore, the families $\T_Q$ can be chosen so that
    \begin{equation}\label{eqn:item6}
        \frac{|\T_0|}{M} \geapp_\delta \frac{|\T_\Delta|}{M_\Delta} \cdot \left( \max_{Q \in \cD_\Delta (\cP)} \frac{|\T_Q|}{M_Q} \right).
    \end{equation}
\end{prop}

Iterate this proposition to get (for details, see \cite[Corollary 4.1]{shmerkin2022dimensions})

\begin{cor}\label{cor:multiscale}
    Fix $N \ge 2$ and a sequence $\{ \Delta_j \}_{j=0}^n \subset 2^{-\N}$ with
    \begin{equation*}
        0 < \delta = \Delta_N < \Delta_{N-1} < \cdots < \Delta_1 < \Delta_0 = 1.
    \end{equation*}
    Let $(\cP_0, \cT_0) \subset \cD_\delta \times \T^\delta$ be a $(\delta, s, C_1, \kappa, C_2, M)$-nice configuration. Then there exists a set $\cP \subset \cP_0$ such that:
    \begin{enumerate}
        \item $|\cD_{\Delta_j} (\cP)| \approx_\delta |\cD_{\Delta_j} (\cP_0)|$ and $|\cP \cap \textbf{p}| \approx_\delta |\cP_0 \cap \textbf{p}|$, $1 \leq j \leq N$, $\textbf{p} \in \cD_{\Delta_j} (\cP)$.
        \item For every $0 \leq j \leq N-1$ and $\textbf{p} \in \cD_{\Delta_j}$, there exist numbers $C_{\textbf{p}}^1 \approx_\delta C^1$, $C_{\textbf{p}}^2 \approx_\delta C^2$, and $M_{\textbf{p}} \geq 1$, and a family of tubes $\cT_{\textbf{p}} \subset \T^{\Delta_{j+1}/\Delta_j}$ with the property that $(S_{\textbf{p}} (\cP \cap \textbf{p}), \cT_{\textbf{p}})$ is a $(\Delta_{j+1}/\Delta_j, s, C_{\textbf{p}}^1, \kappa, C_{\textbf{p}}^2, M_{\textbf{p}})$-nice configuration.
    \end{enumerate}

    Furthermore, the families $\cT_{\textbf{p}}$ can be chosen such that if $\textbf{p}_j \in \cD_{\Delta_j} (\cP)$ for $0 \le j \le N-1$, then
    \begin{equation*}
        \frac{|\cT_0|}{M} \geapp_\delta \prod_{j=0}^{N-1} \frac{|\cT_{\textbf{p}_j}|}{M_{\textbf{p}_j}}.
    \end{equation*}
    Here, $\geapp_\delta$ means $\gesim_N \log(1/\delta)^C$, and likewise for $\leapp_\delta, \app_\delta$.
\end{cor}

\subsection{Uniform sets and branching numbers}
The following exposition borrows heavily from \cite[Section 2.3]{orponen2023projections}.
\begin{defn}
    Let $n \ge 1$ and
    \begin{equation*}
        \delta = \Delta_n < \Delta_{n-1} < \cdots < \Delta_1 \le \Delta_0 = 1
    \end{equation*}
    be a sequence of dyadic scales. We say that a set $P \subset [0, 1)^d$ is $\{ \Delta_j \}_{j=1}^n$-uniform if there is a sequence $\{ N_j \}_{j=1}^n$ such that $N_j \in 2^{\N}$ and $|P \cap Q|_{\Delta_j} = N_j$ for all $j \in \{ 1, 2, \cdots, n \}$ and $Q \in \cD_{\Delta_{j-1}} (P)$.
\end{defn}

\begin{remark}
    By uniformity, we have $|P|_{\Delta_m} = |P \cap Q|_{\Delta_m} |P|_{\Delta_\ell}$ for $0 \le \ell < m \le n$ and $Q \in \cD_{\Delta_\ell} (P)$.
\end{remark}

As a result, we can always refine a set $P$ to be uniform:

\begin{lemma}\label{lem:uniform}
    Let $P \subset [0, 1)^d$, $m, T \in \N$, and $\delta = 2^{-mT}$. Let $\Delta_j := 2^{-jT}$ for $0 \le j \le m$, so in particular $\delta = \Delta_m$. Then there is a $\{ \Delta_j \}_{j=1}^m$-uniform set $P' \subset P$ such that
    \begin{equation*}
        |P'|_\delta \ge (2T)^{-m} |P|_\delta.
    \end{equation*}
    In particular, if $\eps > 0$ and $T^{-1} \log (2T) \le \eps$, then $|P'| \ge \delta^\eps |P|$.
\end{lemma}

Uniform sets can be encoded by a branching function.

\begin{defn}
    Let $T\in \mathbb{N}$, and let $\cP\subset [0,1)^d$ be a $\{\Delta_j\}_{j=1}^n$-uniform set, with $\Delta_j: = 2^{-jT}$, and with associated sequence $\{N_j\}_{j=1}^n\subset \{ 1, \dots, 2^{dT}\}^n$. 
    We define the \emph{branching function} $f: [0, n]\rightarrow [0, dn]$ by setting $f(0)=0$, and 
    \[
    f(j):=\frac{\log |\cP|_{2^{-jT}}}{T} =\frac{1}{T}\sum_{i=1}^j \log N_i, \quad i \in \{1, \dots n\},
    \]
    and then interpolating linearly between integers.
\end{defn}

\begin{defn}
    Let $s_f (a, b) = \frac{f(b) - f(a)}{b - a}$ denote the slope of a line segment between $(a, b)$ and $(f(a), f(b))$. We say that a function $f: [0,n]\rightarrow \mathbb{R}$ is $\eps$-superlinear on $[a,b]\subset [0,n]$, or that $(f, a, b)$ is $\eps$-superlinear, if 
 \[ f(x)\geq f(a) + s_f (a, b) (x-a) - \eps (b - a), x\in [a, b]. 
 \]
 We say that $(f, a, b)$ is $\eps$-linear if
  \[ |f(x) - f(a) - s_f (a, b) (x-a)| \le \eps (b - a), x\in [a, b]. 
 \]
\end{defn}

The following lemma converts between branching functions and the uniform structure of $P$. It is \cite[Lemma 8.3]{orponen2021hausdorff} (or an immediate consequence of the definitions)
\begin{lemma}
    Let $P$ be a $(\Delta^i)_{i=1}^m$-uniform set in $[0, 1)^d$ with associated branching function $f$, and let $\delta = \Delta^m$.
    \begin{enumerate}[(i)]
        \item If $f$ is $\eps$-superlinear on $[0, m]$, then $P$ is a $(\delta, s_f (0, m), O_\Delta (1) \delta^{-\eps})$-set.

        \item If $f$ is $\eps$-linear on $[0, m]$, then $P$ is a $(\delta, s_f (0, m), O_\Delta (1) \delta^{-\eps}, O_\Delta (1) \delta^{-\eps})$-regular set between scales $\delta$ and $1$.
    \end{enumerate}
\end{lemma}

The crucial branching lemma is \cite[Lemma 8.5]{orponen2021hausdorff} applied to the function $\frac{2}{d} \cdot f$:

\begin{lemma}\label{lem:good branching}
    Fix $s \in (0, 1)$ and $t \in (s, d]$. For every $\eps > 0$ there is $\tau = \tau(\eps, s, t) > 0$ such that the following holds: for every piecewise affine $d$-Lipschitz function $f : [0, m] \to \R$ such that
    \begin{equation*}
        f(x) \ge tx - \eps m \text{ for all } x \in [0, m],
    \end{equation*}
    there exists a family of non-overlapping intervals $\{ [c_j, d_j] \}_{j=1}^n$ contained in $[0, m]$ such that:
    \begin{enumerate}
        \item For each $j$, at least one of the following alternatives holds:
        \begin{enumerate}
            \item $(f, c_j, d_j)$ is $\eps$-linear with $s_f (c_j, d_j) \ge s$;

            \item $(f, c_j, d_j)$ is $\eps$-superlinear with $s_f (c_j, d_j) = s$.
        \end{enumerate}

        \item $d_j - c_j \ge \tau m$ for all $j$;

        \item $|[0, m] \setminus \cup_j [c_j, d_j]| \lesim_{s,t} \eps m$.
    \end{enumerate}
\end{lemma}

\subsection{Combinatorial and probabilistic preliminaries}
In this section, we collect a few of the results from additive combinatorics and probability that will be used in the following sections.

First, we make the following observation (Lemma 19 of \cite{he2020orthogonal}) about intersections of high-probability events. (That lemma was stated for Lebesgue measure but the same proof works for general measures $\nu$.)

\begin{lemma}\label{lem:intersections_of_events}
    Let $A \subset \R^d$ equipped with a measure $\nu$ and $\Theta$ be an index set equipped with a probability measure $\mu$. Suppose there is $K \ge 1$ and for each $\theta \in \Theta$, a Borel subset $A_\theta$ with $\nu(A_\theta) \ge \nu(A)/K$. Then
    \begin{equation*}
        \mu^{\otimes q} (\{ (\theta_1, \theta_2, \cdots, \theta_q) : \nu(A_{\theta_1} \cap A_{\theta_2} \cap \cdots \cap A_{\theta_q}) \ge \frac{\nu(A)}{2K^q} \})\ge \frac{1}{2K^q}.
    \end{equation*}
\end{lemma}

Next, we state Rusza's triangular inequality \cite[Lemma 21]{he2020orthogonal} (see also \cite{ruzsa1978cardinality}):

    \begin{lemma}\label{lem:rusza triangle}
        For any sets $A, B, C \subset \R^d$, we have
        \begin{equation*}
            |B|_\delta |A-C|_\delta \lesim_d |A-B|_\delta |B-C|_\delta.
        \end{equation*}
    \end{lemma}

We also would like the Pl\"unnecke-Rusza inequality, in the form stated by \cite[Lemma 22]{he2020orthogonal}:

\begin{lemma}\label{lem:plunnecke-rusza}
    Let $A, B$ be bounded subsets of $\R^d$. For all $K \ge 1$, $\delta > 0$, if $|A + B|_\delta \le K |B|_\delta$, then for all $k, \ell \ge 1$, we have
    \begin{equation*}
        |kA - \ell A|_\delta \lesim_d K^{k+\ell} |B|_\delta.
    \end{equation*}
    Here, $kA = \underbrace{A + \cdots + A}_{k \text{ times}}$.
\end{lemma}

In a similar spirit, the set of $w$ such that $X + wX$ is small compared to $|X|$ forms a ring. The following is a restatement of \cite[Lemma 30(i,ii)]{he2016discretized} for $\R$. Note that $\End(\R) \simeq \R$ with identity $1$.

    \begin{lemma}\label{lem:ring_structure}
        Define $S_\delta (X; K) = \{ w \in [-K, K] : |X + wX|_\delta \le K |X|_\delta \}$.
        \begin{enumerate}[(i)]
            \item If $a \in S_\delta (X; \delta^{-\eps})$ and $b \in \R$ such that $|a-b| \le \delta^{1-\eps}$, then $b \in S_\delta (X; \delta^{-O(\eps)})$.

            \item If $1, a, b \in S_\delta (X; \delta^{-\eps})$, then $a+b, a-b, ab$ all belong to $S_\delta (X; \delta^{-O(\eps)})$.
        \end{enumerate}
    \end{lemma}

The following theorem (a special case of Theorem 5 of \cite{he2016discretized}) is a quantitative statement that $\frac{1}{2}$-dimensional subrings of $\R$ don't exist. In fact, by repeated sum-product operations, we can get all of $\R$.
\begin{theorem}\label{thm:iterated sum-product}
    We work in $\R^1$. Given $\kappa, \eps_0 > 0$, there exist $\eps > 0$ and an integer $s \ge 1$ such that for $\delta < \delta_0 (\kappa, \eps_0)$, the following holds. For every $(\kappa, \delta^{-\eps})$-set $A \subset B(0, \delta^{-\eps})$, we have
    \begin{equation*}
        B(0, \delta^{\eps_0}) \subset \langle A \rangle_s + B(0, \delta),
    \end{equation*}
    where $\langle A \rangle_1 := A \cup (-A)$ and for any integer $s \ge 1$, define $\langle A \rangle_{s+1} := \langle A \rangle_s \cup (\langle A \rangle_s + \langle A \rangle_1) \cup (\langle A \rangle_s \cdot \langle A \rangle_1)$.
\end{theorem}

Finally, we shall need a discretized variant of the Balog-Szemer\'{e}di-Gowers theorem. Our version is closest to \cite[Theorem 4.38]{orponen2020improved}, which is taken from \cite[p. 196]{bourgain2010discretized}, which in turn refers to Exercise 6.4.10 in \cite{tao2006additive}. But the exercise is only sketched in \cite{tao2006additive}, so for completeness, we provide a proof in Appendix \ref{appendix:proof of bsg}. 

\begin{theorem}\label{thm:bsg}
    Let $K \ge 1$ and $\delta > 0$ be parameters. Let $A, B$ be bounded subsets of $\R^d$, and let $P \subset A \times B$ satisfy
    \begin{equation*}
        |P|_\delta \ge K^{-1} |A|_\delta |B|_\delta \text{ and } |\{ a + b : (a, b) \in P \}|_\delta \le K (|A|_\delta |B|_\delta)^{1/2}.
    \end{equation*}
    Then one can find subsets $A' \subset A, B' \subset B$ satisfying
    \begin{itemize}
        \item $|A'|_\delta \gesim_d K^{-2} |A|_\delta, |B'|_\delta \gesim_d K^{-2} |B|_\delta$,

        \item $|A' + B'|_\delta \lesim_d K^8 (|A|_\delta |B|_\delta)^{1/2}$,

        \item $|P \cap (A' \times B')| \gesim_d \frac{|A|_\delta |B|_\delta}{K^2}$.
    \end{itemize}
    (Implicit constants depend on $d$ but not on $\delta, K$.)
\end{theorem}

We also need the following version of multi-linear Kakeya.
\begin{theorem}[Theorem 1 in \cite{carbery2013endpoint}]\label{thm:multlinear kakeya}
    Let $2 \le k \le d$ and $\T_1, \T_2, \cdots, \T_k$ be families of $1$-tubes in $\R^d$. Then
    \begin{equation*}
        \int_{\R^d} \left( \sum_{T_1 \in \T_1} \cdots \sum_{T_k \in \T_k} |e(T_1) \wedge \cdots \wedge e(T_k)| \chi_{T_1 \cap \cdots \cap T_k} (x) \right)^{1/(k-1)} \, dx \lesim_{k,d} \left( \prod_{i=1}^k |\T_i| \right)^{1/(k-1)}.
    \end{equation*}
    Here, $e(T_i)$ is the unit vector in the direction of tube $T_i$.
\end{theorem}

\subsection{Energy}
\begin{defn}
    The $(s, k)$-Riesz energy of a finite Borel measure $\mu$ on $\R^d$ is
    \begin{equation*}
        I_{s,k}^\delta (\mu) = \int (|(x_0 - x_1) \wedge \cdots \wedge (x_0 - x_k)| + \delta)^{-s} \, d\mu(x_0) \cdots d\mu(x_k).
    \end{equation*}
    If $k = 1$ and $\delta = 0$, we recover the usual $s$-dimensional Riesz energy.
\end{defn}

\begin{lemma}\label{lem:energy}
    \begin{enumerate}[(a)]
        \item Fix $0 < s < t$ and a measure $\mu$ with total mass $C$. If $\mu(H_r) \le Cr^t$ for every $(r, k-1)$-plate $H_r$ and $r > 0$, then $I_{s,k}^0 (\mu) \lesim_{s,t} C^k$.

        \item Fix $0 < \delta < \frac{1}{2}$. If $I_{s_i,k_i}^\delta (\mu) \le C$ for $1 \le i \le m$, then $\spt(\mu)$ contains a set which is simultaneously a $(\delta, \frac{s_i}{k_i}, O(1) \cdot (Cm)^{1/k_i} \log \delta^{-1}, k_i-1)$-set for each $i$.
    \end{enumerate}
\end{lemma}

\begin{remark}
    If $k_1 = m = 1$ in part (b), then we can drop the log factor (c.f. proof of Lemma A.6 in \cite{orponen2021hausdorff}). We don't know if we can drop the log factor for $k > 1$ or $m > 1$.
\end{remark}

\begin{proof}
    (a) Let $\rho_i$ be the distance between $x_i$ and the plane spanned by $x_0, \cdots, x_{i-1}$; notice that $|(x_0 - x_1) \wedge \cdots \wedge (x_0 - x_k)| = \prod_{i=1}^k \rho_i$. Thus, we can rewrite $I_{s,k} (\mu)$ as an iterated integral
    \begin{equation*}
        \int d\mu(x_0) \int \rho_1^{-s} d\mu(x_1) \int \rho_2^{-s} \, d\mu(x_2) \cdots \int \rho_k^{-s} \, d\mu(x_k).
    \end{equation*}
    We will be done if we show for all $1 \le i \le k$ and choices of $x_0, \cdots, x_{i-1}$, that $\int \rho_i^{-s} \, d\mu(x_i) \lesim C$. Let $H$ be the span of $x_0$ through $x_{i-1}$, and observe that by definition, $\{ x_i : \rho_i \ge r \} \subset H^{(r)}$, which is contained in a $(r, k-1)$-plate. Thus,
    \begin{equation*}
        \int \rho_i^{-s} \, d\mu(x_i) \lesim C + \sum_{\rho = 2^{-n}, n \ge 1} C\rho^{t-s} \lesim_{s,t} C.
    \end{equation*}

    (b) Let $P_i = \{ x_0 \in \spt(\mu) : \int (|(x_0 - x_1) \wedge \cdots \wedge (x_0 - x_{k_i})| + \delta)^{-s_i} \, d\mu(x_1) \cdots d\mu(x_{k_i}) < 2mC \}$. By Markov's inequality, $\mu(P_i) > 1 - \frac{1}{2m}$, so by the union bound, $P = \cap_{i=1}^m P_i$ satisfies $\mu(P) > \frac{1}{2}$.
    
    We claim that $\mu(P \cap H_r) \le C r^{s_i/k_i}$ for all $(r, k_i-1)$-plates $H_r$ and $\delta \le r \le 1$, $1 \le i \le m$. Indeed, if $P \cap H_r = \emptyset$, then we are done. Otherwise, pick $x_0 \in P \cap H_r$ and observe that if $x_1, x_2, \cdots, x_{k_i} \in H_r$, then $|(x_0 - x_1) \wedge \cdots \wedge (x_0 - x_{k_i})| + \delta \lesim r$. Thus, we get $\mu(H_r)^{k_i} \cdot r^{-s} \le 2C$, so $\mu(P \cap H_r) \le (2C r^s)^{1/k_i}$.

    Finally, let $P'_c \subset \cD_\delta (P)$ be those dyadic $\delta$-cubes $p$ such that $\mu(p) \sim c$. We know $\sum_{c = 2^{-n} \in [\delta^d, 1]} \mu(P'_c) \ge \frac{1}{4}$, so by dyadic pigeonholing, some $\mu(P'_c) \gesim (\log \delta^{-1})^{-1}$. Then $P'_c$ will be a $(\delta, \frac{s_i}{k_i}, O(1) \cdot (Cm)^{1/k_i} \log \delta^{-1}, k_i-1)$-set for all $1 \le i \le m$.
\end{proof}

\section{Improved incidence estimates for quasi-product sets}\label{sec:quasi-product sets}
The main novelty of this paper is the following Proposition, which is a higher-dimensional refinement of \cite[Proposition 4.36]{orponen2020improved} (see also \cite[Proposition A.7]{orponen2021hausdorff}). It can be viewed as a variant of Theorem \ref{thm:main} for quasi-product sets.

\begin{prop}\label{prop:improved_incidence_weaker}
    Given $0 \le k < d-1$, $0 \le s < k+1$, $\tau, \kappa > 0$, there exist $\eta(s, k, \kappa, \tau, d) > 0$ and $\delta_0 (s, k, \kappa, \tau, d) > 0$ such that the following holds for all $\delta \in (0, \delta_0]$.
    
    Let $\bY \subset (\delta \cdot \Z) \cap [0,1)$ be a $(\delta, \tau, \delta^{-\eta})$-set, and for each $\by \in \bY$, assume that $\bX_{\by} \subset (\delta \cdot \Z)^{d-1} \cap [0,1)^{d-1}$ is a $(\delta, \kappa, \delta^{-\eta}, k)$-set with cardinality $\ge \delta^{-s+\eta}$. Let
    \begin{equation*}
        \bZ = \bigcup_{\by \in \bY} \bX_{\by} \times \{ \by \}.
    \end{equation*}
    For every $\bz \in \bZ$, assume that $\T(\bz)$ is a set of $\delta$-tubes each making an angle $\ge \frac{1}{100}$ with the plane $y = 0$ with $|\T(\bz)| \ge \delta^{-s+\eta}$ such that $\bz \in T$ for all $T \in \T(\bz)$. Then $|\T| \ge \delta^{-2s-\eta}$, where $\T = \cup_{\bz \in \bZ} \T(\bz)$.
\end{prop}

\begin{remark}
    In contrast to Theorem \ref{thm:main} and \cite[Proposition 4.36]{orponen2020improved}, we (perhaps surprisingly) don't need any non-concentration assumptions on the tube sets $\T(\bz)$ (even when $d = 2$). Instead, it suffices to have weak non-concentration assumptions on $\bY$ and $\bX_\by$ for each $\by \in \bY$. The non-concentration assumption on $\bX_\by$ is necessary: otherwise, we can take $s = k$, and let $\bZ$ to be the $\delta$-balls contained in some $(\delta, k+1)$-plate $H$, and $\T$ to be the $\delta$-tubes contained in $H$.
\end{remark}

\subsection{An improved slicing estimate}\label{subsec:slicing}
We will eventually deduce Proposition \ref{prop:improved_incidence_weaker} from the following slicing estimate.

\begin{theorem}\label{thm:projection}
    For $0 \le k \le d-2$, $0 \le s < k+1$, and $0 < \kappa \le 1$, there exists $\eps > 0$ such that the following holds for sufficiently small $\delta < \delta_0 (s, k, d, \eps)$. Let $\T$ be a $(\delta, \kappa, \delta^{-\eps}, k)$-set of $\delta$-tubes each making angle $\ge \frac{1}{100}$ with the plane $y = 0$ with $|\T| \ge \delta^{-2s+\eps}$. Let $\mu$ be a probability measure on $\R$ such that for all $\delta \le r \le 1$, we have $\mu(B_r) \le \delta^{-\eps} r^\kappa$. Then there is a set $\cD \subset \R$ with $\mu(\cD) \ge 1 - \delta^\eps$ such that the slice of $\cup \T'$ at $z = z_0$ has $\delta$-covering number $\ge \delta^{-s-\eps}$, for every subset $\T' \subset \T$ with $|\T'| \ge \delta^\eps |\T|$ and $x \in \cD$.
\end{theorem}

\begin{remark}
    One should compare Theorem \ref{thm:projection} to \cite[Theorem 1]{he2020orthogonal}. Indeed, if $k = 0$ and $d = 2$, Theorem \ref{thm:projection} is a direct corollary of \cite[Theorem 1]{he2020orthogonal}. We can see this by using ball-tube duality, which turns $\T$ into a subset of $\R^2$. Under this duality, the slice of $\cup \T'$ at $z = z_0$ becomes the orthogonal projection $\pi_{\tz_0}$ to a line in the dual space, for some $\tz_0 \in S^1$. The map $z_0 \to \tz_0$ induces a pushforward measure $\tmu$ of $\mu$ which still satisfies the non-concentration condition $\tmu(B_r) \lesim \delta^{-\eps} r^\kappa$, so we can apply \cite[Theorem 1]{he2020orthogonal}. (For more details, see the proof of Proposition A.7 in \cite{orponen2021hausdorff}.)

    In higher dimensions, we can still use duality to turn $\T$ into a subset of $\A(d, 1) \sim \R^{2(d-1)}$, and then slices of $\cup \T'$ become orthogonal projections to $(d-1)$-planes. Unfortunately, \cite[Theorem 1]{he2020orthogonal} does not apply because the pushforward measure $\tmu$ is still supported on a line in $S^{d-1}$. This approach is bound to fail because \cite[Theorem 1]{he2020orthogonal} does not use the strong assumption that $\T$ is non-concentrated around $(k+1)$-planes. Using this assumption is the key novelty of this proof. 
\end{remark}

Nonetheless, Theorem \ref{thm:projection} will borrow many ideas from the proof of \cite[Theorem 1]{he2020orthogonal} and He's previous work \cite{he2016discretized}. Roughly, the strategy is as follows.
\begin{itemize}
    \item As in \cite{he2020orthogonal}, reduce to the following slightly weaker statement: given $\T$ and $\mu$, we can find a subset $\T' \subset \T$ such that the conclusion of Theorem \ref{thm:projection} holds for $\T'$ in place of $\T$. This relies on a formal exhaustion argument.

    \item Then, as in \cite{he2020orthogonal}, reduce this slightly weaker to the following even weaker statement: there exists $z_0 \in E := \spt \mu$ such that the slice of $\cup \T$ at $z = z_0$ has $\delta$-covering number $\ge \delta^{-s-\eps}$. This relies on additive combinatorics (e.g. the Balog-Szemer\'{e}di-Gowers theorem) and some probability.

    \item Assume this is false: that for all $z_0 \in E$, the slice of $\cup \T$ at $z = z_0$ has $\delta$-covering number $\leapp \delta^{-s}$. Using additive combinatorics as in \cite{he2016discretized}, the same conclusion is true for all $z_0 \in E'$, which is the set of sums or differences of $m$ many terms, each of which is a product of $m$ elements of $E$. (Here, $m$ will be a fixed large integer.)

    \item Finally, if $m$ is sufficiently large in terms of $\kappa, \eps$, then $E'$ contains a large interval $[0, \delta^{\eps}]$ (c.f. \cite[Theorem 5]{he2016discretized}). Essentially, we have a set of $\geapp \delta^{-2s}$ many tubes $\T$, each containing $\geapp \delta^{-1}$ many $\delta$-balls, such that the union of the $\delta$-balls has cardinality $\leapp \delta^{-(s+1)}$. Without further restrictions, this Furstenberg-type problem doesn't lead to a contradiction: take $s = k$ and $\T$ to be the set of $\delta$-tubes in a $(\delta, k+1)$-plate. Luckily, our set of tubes $\T$ is still a $(\delta, \kappa, \delta^{-O(\eps)}, k)$-set, which rules out this counterexample. Indeed, we may finish using multi-linear Kakeya.
\end{itemize}

The reader be warned: we shall execute this strategy in reverse order. This is mainly because the main innovation of the paper is the fourth bullet point.




\subsection{An improved Furstenberg estimate}
The following estimate complements work of Zahl \cite{zahl2022unions}: we prove an $\eps$-improvement on the union of tubes under a mild $(k+1)$-plane non-concentration for the set of tubes. As in Zahl \cite{zahl2022unions}, the key technique is multilinear Kakeya.

\begin{theorem}\label{thm:mult_kakeya}
    For any $0 \le k < d-1$, $0 \le s < k+1$, $0 < \kappa \le 1$, there exists $\eps > 0$ such that the following holds for sufficiently small $\delta > 0$. Let $\T$ be a $(\delta, \kappa, \delta^{-\eps}, k)$-set of $\delta$-tubes with $|\T| \ge \delta^{-2s+\eps}$, and for each $t \in \T$, let $P_t$ be a set of $\delta$-balls intersecting $t$ such that $|P_t| \ge \delta^{-1+\eps}$. Then $|\cup P_t| \gesim \delta^{-(s+1)-\eps}$.
\end{theorem}

\begin{proof}
The proof below is lossy and can possibly be improved (say by induction on scale). Also, the $\eps$ can be determined explicitly in terms of the parameters but we choose not to do so here.

We use $\geapprox$ notation to hide $\delta^{-C \eps}$ terms, where $C$ can depend on the other parameters. Let $P = \cup P_t$, and suppose $|P| \leapp \delta^{-(s+1)}$. Let $\T(p)$ be the set of tubes in $\T$ through $p$. Use a bush argument to upper bound $|\T(p)|$:
\begin{equation*}
    \delta^{-(s+1)} \geapp |P| \ge |\cup_{t \ni p} (P_t \setminus B(p, \delta^{2\eps}))| \ge \delta^{2d\eps} \sum_{t \ni p} |P_t \setminus B(p, \delta^{2\eps})| \ge \delta^{-1+(2d+1) \eps} |\T(p)|.
\end{equation*}
Thus, $|\T(p)| \leapp \delta^{-s}$ for all $p \in P$. We get the following inequality chain
\begin{equation*}
    \delta^{-2s-1} \leapp \delta^{-1} |\T| \leapp I(P, \T) \leapp \delta^{-s} |P| \leapp \delta^{-2s-1.}
\end{equation*}
This means $I(P, \T) \app \delta^{-2s-1}$, $|P| \app \delta^{-s-1}$, and $|\T| \app \delta^{-2s}$. Now perform a dyadic pigeonholing to extract a subset $P' \subset P$ such that $|\T(p)| \in [M, 2M]$ for all $p \in P'$ and $I(P', \T) \app \delta^{-2s-1}$. We know from before that $M \le \delta^{-s}$, and $\delta^{-2s-1} \leapp I(P', \T) \app M |P'| \leapp M |P| \leapp \delta^{-2s-1}$, so $M \app \delta^{-s}$ and $|P'| \app \delta^{-s}$. (This type of dyadic pigeonholing will also be used later. We also remark that dyadic pigeonholing was not necessary to achieve this step; simply let $P'$ be the set of $p \in P$ satisfying $|\T(p)| \ge \delta^{-s+C\eta}$ for some large $C$, and use the bound on $I(P, \T)$ to get a lower bound for $|P'|$.)

Now, we claim that $P'$ is a $(\delta, \kappa, \delta^{-O(\eps)}, k+1)$-set. Fix $\delta < r < 1$ and let $H_r$ be a $(r, k+1)$-plate. We first bound $I(P' \cap H_r, \T)$. Letting $H_{r'}$ be the $(r', k+1)$-plate that is a dilate of $H_r$ with the same center, we have
\begin{align*}
    I(P' \cap H_r, \T) &\le \sum_{r' = 2^{-\N} \cap [r, 1]} I(P' \cap H_r, H_{r'} \setminus H_{r'/2}) \\
            &\le \sum_{r'} \frac{r}{r' \delta} \cdot |\T \cap H_{r'}| \\
            &\le \sum_{r'} \frac{r}{r' \delta} \cdot |\T| \delta^{-\eps} (r')^\kappa \\
            &\lesim \delta^{-1} |\T| \delta^{-\eps} r^\kappa.
\end{align*}
Thus, since $I(P' \cap H_r, \T) \geapp \delta^{-s} |P' \cap H_r|$, we have $|P' \cap H_r| \leapp \delta^{s-1} |\T| r^\kappa \app \delta^{-(s+1)} r^\kappa \app |P'| r^\kappa$.

Finally, since $I(P', \T) \geapp |P'| \delta^{-s} \geapp \delta^{-2s-1}$ and $|\T| \leapp \delta^{-2s}$, by dyadic pigeonholing there exists a subset $\T' \subset \T$ with $|\T'| \app |\T|$ such that each $t \in \T'$ contains $\app \delta^{-1}$ many $\delta$-balls in $P'$. Now since $I(P', \T') \geapp \delta^{-1} |\T'| \geapp \delta^{-2s-1}$, $|P'| \leapp \delta^{-s-1}$, and $|\T(p)| \leapp \delta^{-s}$ for all $p \in P'$, by dyadic pigeonholing we can find $\tP \subset P'$ with $|\tP| \geapp |P'|$ such that each $p \in \tP$ lies in $\app \delta^{-s}$ many tubes in $\T'$.

Now, we are in good shape to apply multilinear Kakeya. For $p \in \tP$, let $\T(p)$ be the tubes in $\T'$ through $p$. By a bush argument, $\cup \T(p)$ contains $\geapp \delta^{-(s+1)}$ many $\delta$-balls in $P$. Since $P'$ is a $(\delta, \kappa, \delta^{-O(\eps)}, k+1)$-set, there are $\geapp \delta^{-(s+1)(k+3)}$ many $(k+3)$-tuples of points $(p_0, p_1, \cdots, p_{k+2})$ such that $p_0$ and $p_i$ lie on some tube $t_i \in \T_i$ and $|e(t_1) \wedge \cdots \wedge e(t_{k+2})| \geapp 1$ (where $e(t)$ is the unit vector in the direction of tube $t$). Thus, there is a choice of $p_1, \cdots, p_{k+2}$ such that there are $\geapp \delta^{-(s+1)}$ many valid choices for $p_0$. But this leads to a contradiction by the following argument. Let $\T_i$ be the tubes of $\T$ through $p_i$, $1 \le i \le k+2$; then by a rescaled version of Multilinear Kakeya (Theorem \ref{thm:multlinear kakeya}), the number of valid choices for $p_0$ is $\leapp \left( \prod_{i=1}^{k+2} |\T_i| \right)^{1/(k+1)} \leapp \delta^{-s(k+2)/(k+1)}$, which (using $s < k+1$) is much smaller than $\delta^{-(s+1)}$ provided that $\eps, \delta$ are sufficiently small in terms of the parameters. This contradiction completes the proof.
\end{proof}

\subsection{From Furstenberg to weak slicing}
\textit{This subsection contains ideas from \cite{he2016discretized} and \cite{he2020orthogonal}.}
\begin{theorem}\label{thm:projection_weaker}
    For $0 \le k \le d-2$, $0 \le s < k+1$, and $0 < \kappa \le 1$, there exists $\eps > 0$ such that the following holds for sufficiently small $\delta < \delta_0 (s, k, d, \eps)$. Let $\T$ be a $(\delta, \kappa, \delta^{-\eps}, k)$-set of $\delta$-tubes each making angle $\ge \frac{1}{100}$ with the plane $y = 0$ with $|\T| \ge \delta^{-2s+\eps}$. Let $\mu$ be a probability measure on $\R$ such that for all $\delta \le r \le 1$, we have $\mu(B_r) \le \delta^{-\eps} r^\kappa$. Then there exists $z_0 \in \spt \mu$ such that the slice of $\cup \T$ at $z = z_0$ has $\delta$-covering number $\ge \delta^{-s-\eps}$.
\end{theorem}

\begin{proof}
    We use $\leapp$ to denote $\le C \delta^{-C\eps}$, where $C$ may depend on $\kappa, s$.
    
    Let $E := \spt \mu$; without loss of generality, assume $E$ is closed. Let $z_1 = \inf E$ and $z_2 = \sup E$; then $d(z_1, z_2) \ge \delta^{2\eps/\kappa}$ since $\mu(B(z_1, \delta^{2\eps/\kappa})) \le \delta^{-\eps} \cdot \delta^{2\eps} \le \delta^{\eps} < 1$.

    Let $X = \cup \T(z = z_1)$ and $Y = \cup \T(z = z_2)$; we are given that $|X|_\delta, |Y|_\delta \leapp \delta^{-s}$. On the other hand, since each $T \in \T$ passes through $O(1)$ many elements in $X$ and $O(1)$ many elements in $Y$, we get that
    \begin{equation*}
        \delta^{-2s} \geapp |X|_\delta |Y|_\delta \gesim |\T| \geapp \delta^{-2s},
    \end{equation*}
    so in fact, $|X|, |Y| \app \delta^{-s}$ and $|\T| \app \delta^{-2s}$.

    Let $E' := [z_1, z_2] \setminus (B(z_1, \delta^{2\eps/\kappa}) \cup B(z_2, \delta^{2\eps/\kappa}))$; then $\mu(E') \ge 1 - 2 \delta^\eps \ge \frac{1}{2}$ for $\delta$ small enough.

    Let $f(z) = \frac{z-z_1}{z_2-z}$; note that on $E'$, we have that $f$ is $\app 1$-bilipschitz, and $f(z) \app 1$ for all $z \in E'$.
    
    The problem condition literally states $|(z_2 - z) X + (z - z_1) Y|_\delta \leapp \delta^{-s}$ for $z \in E'$; since $(z_2 - z) \app 1$, we can divide through by $(z_2 - z)$ to get
    \begin{equation*}
        |X + f(z) Y|_\delta \leapp \delta^{-s} \text{ for } z \in E'.
    \end{equation*}
    Now pick an arbitrary $z' \in E'$. In particular, we get $|X + f(z') Y|_\delta \leapp \delta^{-s}$, so by Lemma \ref{lem:rusza triangle}, we have for all $z \in E'$,
    \begin{equation*}
        |X - \frac{f(z)}{f(z')} X|_\delta \le \frac{|X + f(z) Y|_\delta |f(z) Y + \frac{f(z)}{f(z')} X|_\delta}{|f(z) Y|_\delta} \leapp \delta^{-s}.
    \end{equation*}
    In addition, since $|X + f(z') Y|_\delta \leapp \delta^{-s} \leapp |X|_\delta$, the Pl\"unnecke-Rusza inequality (Lemma \ref{lem:plunnecke-rusza}) gives $|X + X|_\delta \leapp |Y|_\delta \leapp \delta^{-s}$.
    
    Define $\tmu = g_* (\mu)$, the pushforward of $g(z) = \frac{f(z)}{f(z')}$; then $g$ (like $f$) is $\geapp 1$-bilipschitz on $E'$, so $\tmu$ also satisfies a non-concentration condition $\tmu(B_r) \leapp r^\kappa$ for $\delta \le r \le 1$. Now pick $\eps_0 > 0$, and assume $\eps$ is chosen sufficiently small in terms of $\eps_0$. By the iterated sum-product Theorem \ref{thm:iterated sum-product}, we can find an integer $m \ge 1$ such that for $\delta < \delta_0 (\kappa, \eps, \eps_0)$,
    \begin{equation*}
        B(0, \delta^{\eps_0}) \subset \langle A \rangle_m + B(0, \delta),
    \end{equation*}
    where $\langle A \rangle_1 := A \cup (-A)$ and for any integer $m \ge 1$, define $\langle A \rangle_{m+1} := \langle A \rangle_m \cup (\langle A \rangle_m + \langle A \rangle_1) \cup (\langle A \rangle_m \cdot \langle A \rangle_1)$.

    By applying the ring structure Lemma \ref{lem:ring_structure} many times, we see that $B(0, \delta^{\eps_0}) \subset S_\delta (X; \delta^{-O_m (\eps)})$ and since $1 \in S_\delta (X; \delta^{-O(\eps)})$, that $B(1, \delta^{\eps_0}) \subset S_\delta (X; \delta^{-O_m (\eps)})$. By definition of $S_\delta$ and Lemma \ref{lem:rusza triangle}, we get for $w \in B(1, \delta^{\eps_0})$,
    \begin{equation*}
        |X + w f(z') Y|_\delta \le \frac{|X - w X|_\delta |w X + w f(z') Y|_\delta}{|w X|_\delta} \leapp \delta^{-s}.
    \end{equation*}
    In other words, for all $z_0 \in I := f^{-1} ([(1-\delta^{\eps_0}) f(z'), (1+\delta^{\eps_0}) f(z')])$, the slice $\cup \T(z = z_0)$ has $\delta$-covering number $\leapp \delta^{-s}$. Since $f$ is $\app 1$-bilipschitz and $f(z') \app 1$, we have $|I| \app \delta^{\eps_0}$.

    Now, we seek a contradiction to Theorem \ref{thm:mult_kakeya}. For every $t \in \T$, we let $P_t$ be the $\delta$-balls on $t$ with $z$-coordinate in $I$. We observe the following:
    \begin{itemize}
        \item Recall our assumption that $\T$ is a $(\delta, \kappa, \delta^{-\eps}, k)$-set of $\delta$-tubes with $|\T| \ge \delta^{-2s+\eps}$.

        \item $|P_t| \gesim \delta^{-1} |I| \geapp \delta^{-1+\eps_0}$.

        \item $|\cup P_t| \le |\cup_{z \in I} (X_* + f(z) Y_*)| \leapp \delta^{-(s+1)}$.
    \end{itemize}
    Thus, if $\eps, \eps_0$ are sufficiently small in terms of $s, k, \kappa$, then we contradict Theorem \ref{thm:mult_kakeya}.
    
    
\end{proof}

\subsection{An intermediate slicing result}
\textit{This subsection contains ideas from \cite{he2020orthogonal}.}

Let $\cE(\T, \eps)$ be the set of exceptional slices,
\begin{equation*}
    \cE(\T, \eps) = \{ z_0 \in \R : \exists \T' \subset \T, |\T'| \ge \delta^\eps |\T|, |\cup \T'(z=z_0)| < \delta^{-s-\eps} \}.
\end{equation*}

Just like in \cite[Proposition 25]{he2020orthogonal}, we will prove a weaker version of Theorem \ref{thm:projection}; the stronger version follows from a formal exhaustion argument which we present in the next subsection.

\begin{theorem}\label{thm:projection_reduced}
    With assumptions of Theorem \ref{thm:projection}, there exists $\T' \subset \T$ such that $\mu(\cE(\T')) \le \delta^\eps$.
\end{theorem}

\begin{proof}
    We use $\leapp$ to denote $\le C \delta^{-C\eps}$, where $C$ may depend on $\kappa, s$. Let $\pi : \R^d \to \R^{d-1}$ be the projection onto the plane orthogonal to the $z$-axis. For a tube $t$, let $t(z') = \pi(t \cap \{ z = z' \})$, and for a set of tubes $\T$, let $\T(z')$ denote the slice $\pi(\T \cap \{ z = z' \})$.
    
    We follow the argument in \cite[Proof of Proposition 7]{he2020orthogonal}. Suppose Theorem \ref{thm:projection_reduced} is false. We can find $z_1$ and a subset $\T''' \subset \T$ with $|\T'''| \ge \delta^\eps |\T|$ such that $|\cup \T'''(z_1)| < \delta^{-s-\eps}$. For this $\T'''$ we have $\mu(\cE(\T''')) \ge \delta^\eps$, hence $\mu(\cE(\T''') \setminus B(z_1, \delta^{3\eps/\kappa})) \ge \delta^\eps - \delta^{2\eps} > 0$ by the non-concentration property of $\mu$. Thus, we can find $z_2$ with $|z_1 - z_2| \geapp 1$ and $\T'' \subset \T'''$ with $|\T''| \ge \delta^{2\eps} |\T|$ such that $X := |\cup \T''(z_1)| < \delta^{-s-\eps}$ and $Y := |\cup \T''(z_2)| < \delta^{-s-\eps}$. Since every $t \in \T''$ passes through a point in $X$ and a point in $Y$, and since there are $\leapp 1$ many tubes through given points $x \in X$ and $y \in Y$, we can find $\T' \subset \T''$ with $|\T'| \geapp |\T| \geapp \delta^{-2s}$ such that for every $x \in X, y \in Y$, there is at most one tube in $\T'$ through $x, y$. In particular,
    \begin{equation*}
        \delta^{-2s} \leapp |\T'| \le |X|_\delta |Y|_\delta \leapp \delta^{-2s},
    \end{equation*}
    and so $|X|_\delta, |Y|_\delta \geapp \delta^{-s}$, $|\T| \leapp |\T'| \leapp \delta^{-2s}$.
    
    For this $\T'$ we have $\mu(\cE(\T')) \ge \delta^\eps$, so defining $\cD = \cE(\T') \setminus (B(z_1, \delta^{3\eps/\kappa}) \cup B(z_2, \delta^{3\eps/\kappa}))$, we have $\mu(\cD) \ge \delta^\eps - 2\delta^{2\eps} > \delta^{2\eps}$.
    
    \textbf{Claim 1.} For $z = az_1 + (1-a)z_2 \in \cD$, we have $a, 1-a \geapp 1$. Furthermore, there exists $X_z \subset X$, $Y_z \subset X$, and $\T_z \subset \T'$ with $|X_z|_\delta, |Y_z|_\delta \geapp \delta^{-s}, |\T_z| \geapp \delta^{-2s}$ such that $|X_z + \frac{1-a}{a} Y_z| \leapp \delta^{-s}$ and for each $t \in \T_z$, we have $t(z_1) \in X_z^{(\delta)}$ and $t(z_2) \in Y_z^{(\delta)}$.

    \textit{Proof.} The first claim is evident by definition of $\cD$. For the second claim, since $z \in \cE(\T')$, there exists $\T'_z \subset \T'$ such that $|\T'_z| \geapp \delta^{-2s}$ and $|\T'_z (z)|_\delta \leapp \delta^{-s}$. Now notice that for each $x \in X, y \in Y$ there is at most one tube $t \in \T'$ passing through $x, y$. Let $P$ be the set of $(x, y) \in X \times Y$ with exactly one tube $t_{x,y} \in \T'_z$ passing through $x, y$. So $|P| \ge |\T'_z| \geapp \delta^{-2s}$. We also observe that $|ax + (1-a)y - t_{x,y} (z)| \le \delta$, and so $|aX \plusP (1-a)Y|_\delta \le |\T'_z (z)|_{2\delta} \leapp \delta^{-2s}$. Thus, by the Balog-Szemer\'{e}di-Gowers theorem \ref{thm:bsg}, we can find $X_z \subset X$, $Y_z \subset Y$, and $\T_z \subset \T'_z$ such that $|aX_z|_\delta, |(1-a)Y_z|_\delta \geapp \delta^{-s}, |\T_z| \geapp \delta^{-2s}$, $|a X_z + (1-a) Y_z| \leapp \delta^{-s}$, and for each $t \in \T_z$, we have $t(z_1) \in X_z^{(\delta)}$ and $t(z_2) \in Y_z^{(\delta)}$. Then $|X_z|_\delta, |Y_z|_\delta \geapp \delta^{-s}$ and $|X_z + \frac{(1-a)}{a} Y_z| \leapp \delta^{-s}$, proving the Claim. \qed

    Now, we apply Lemma \ref{lem:intersections_of_events} to the sets $X_z^{(\delta)} \times Y_z^{(\delta)}$, the measure $\frac{1}{\mu(\cD)} \mu|_\cD$, and $K = \delta^{-C\eps}$ for a sufficiently large $C$. The result, after applying Fubini's theorem, is that we can find $z_*$, $X_* := X_{z_*} \subset X$, $Y_* := Y_{z_*} \subset Y$, and a subset $\cD' \subset \cD$ with $\mu(\cD') \geapp \mu(\cD) \geapp 1$ and $z_* \in \cD'$ such that for all $z \in \cD'$, we have
    \begin{equation*}
        |X_*^{(\delta)} \cap X_z^{(\delta)}| |Y_*^{(\delta)} \cap Y_z^{(\delta)}| \geapp \delta^{2(n-1)} \delta^{-2s}.
    \end{equation*}
    Since $|X_*^{(\delta)} \cap X_z^{(\delta)}| \lesim |X|_\delta \leapp \delta^{-s}$ and $|Y_*^{(\delta)} \cap Y_z^{(\delta)}| \lesim |Y|_\delta \leapp \delta^{-s}$, we have in fact $|X_*^{(\delta)} \cap X_z^{(\delta)}|, |Y_*^{(\delta)} \cap Y_z^{(\delta)}| \app \delta^{-s}$. In particular, $|X_z^{(\delta)}|, |Y_z^{(\delta)}| \app \delta^{-s}$ for all $z \in \cD'$.

    The next leg of the proof is to show:
    
    \textbf{Claim 2.} For all $z \in \cD'$, if we write $z = az_1 + (1-a)z_2$, then $|X_* + \frac{1-a}{a} Y_*|_\delta \leapp \delta^{-s}$.

    \textit{Proof.} Note that Claim 1 tells us $|X_z + \frac{1-a}{a} Y_z|_\delta \leapp \delta^{-s}$. Combining this with the Rusza triangle inequality (Lemma \ref{lem:rusza triangle}), $X_*^{(\delta)} \cap X_z^{(\delta)} \subset X_z^{(\delta)}$, and $|A^{(\delta)}|_\delta \sim_d |A|_\delta$ for any subset $A$ of the doubling metric space $\R^d$, we have
    \begin{equation*}
        |X_z - X_*^{(\delta)} \cap X_z^{(\delta)}|_\delta \lesim |X_z^{(\delta)} - X_z^{(\delta)}|_\delta \lesim |X_z - X_z|_\delta \lesim \frac{|X_z + \frac{1-a}{a} Y_z|_\delta^2}{|\frac{1-a}{a} Y_z|} \leapp \delta^{-s}.
    \end{equation*}
    The same argument shows (where $z_* = a_* z_1 + (1-a_*) z_2$):
    \begin{equation*}
        |X_* - X_*^{(\delta)} \cap X_z^{(\delta)}|_\delta \lesim |X_* - X_*|_\delta \lesim \frac{|X_* + \frac{1-a_*}{a_*} Y_*|_\delta^2}{|\frac{1-a_*}{a_*} Y_*|} \leapp \delta^{-s}.
    \end{equation*}
    Thus, by Lemma \ref{lem:rusza triangle} again, we have
    \begin{equation*}
        |X_* - X_z|_\delta \lesim \frac{|X_z - X_*^{(\delta)} \cap X_z^{(\delta)}|_\delta |X_* - X_*^{(\delta)} \cap X_z^{(\delta)}|_\delta}{|X_*^{(\delta} \cap X_z^{(\delta)}|_\delta} \leapp \delta^{-s}.
    \end{equation*}
    Similarly, we have $|Y_* - Y_z|_\delta \leapp \delta^{-s}$. A final application of Lemma \ref{lem:rusza triangle} gives
    \begin{align*}
        |X_* + \frac{1-a}{a} Y_*|_\delta &\lesim \frac{|X_z + \frac{1-a}{a} Y_*|_\delta |X_z - X_*|_\delta}{|X_z|_\delta} \\
            &\lesim \frac{|X_z + \frac{1-a}{a} Y_z|_\delta |X_z - X_*|_\delta |\frac{1-a}{a} (Y_z - Y_*)|_\delta}{|X_z|_\delta |\frac{1-a}{a} Y_z|_\delta} \\
            &\leapp \frac{\delta^{-s} \delta^{-s} \delta^{-s}}{\delta^{-s} \delta^{-s}} \le \delta^{-s}.
    \end{align*}
    This proves Claim 2. \qed

    Finally, we seek a contradiction by applying Theorem \ref{thm:projection_weaker} to $\T_{z_*}$ and $\mu|_{\cD'}$. We satisfy the condition (if $\eps$ is sufficiently small) because Claim 1 and $|\T| \leapp \delta^{-2s}$ tell us that $\T_{z_*}$ is a $(\delta, \kappa, \delta^{-O(\eps)}, k)$-set with $|\T_{z_*}| \geapp \delta^{-2s}$. But we violate the conclusion (if $\eps$ is sufficiently small) because Claim 2 tells us that $|aX_* + (1-a) Y_*|_\delta \leapp \delta^{-s}$. This contradiction finishes the proof of Theorem \ref{thm:projection_reduced}.

\end{proof}

\subsection{Formal exhaustion argument}
Using Theorem \ref{thm:projection_reduced}, we prove the following proposition, which implies Theorem \ref{thm:projection} with a different value for $\eps$.

\begin{prop}
    For $0 \le k < d-1$, $0 \le s < k+1$, and $0 < \kappa \le 1$, there exists $\eps > 0$ such that the following holds for sufficiently small $\delta < \delta_0 (s, k, d, \eps)$. Let $\T$ be a $(\delta, \kappa, \delta^{-\eps/2}, k)$-set of $\delta$-tubes each making angle $\ge \frac{1}{100}$ with the plane $y = 0$ with $|\T| \ge \delta^{-2s+\eps/2}$. Let $\mu$ be a probability measure on $\R$ such that for all $\delta \le r \le 1$, we have $\mu(B_r) \le \delta^{-\eps} r^\kappa$. Then $\mu(\cE(\T, \frac{\eps}{3})) \le \delta^{\eps/2}$.
\end{prop}

The idea is the following. A first application of Theorem~\ref{thm:projection_reduced} gives a subset $\T' \subset \T$ with $\mu(\cE(\T',\epsilon)) \leq \delta^\epsilon$. Either $\T'$ is large enough in which case we are done or we can cut $\T'$ out of $\T$ and apply Theorem~\ref{thm:main} again. This will give us another subset $\T'$. Then we iterate until the union of these sets $\T'$ is large enough.
\begin{proof}
Let $N\geq 0$ be an integer. Suppose we have already constructed pairwise disjoint sets $\T_1,\cdots,\T_N$ such that $\mu(\cE(\T_i,\epsilon)) \leq \delta^\epsilon$ for every $i = 1,\cdots,N$. Either we have
\begin{equation}\label{eq:AprimeBIG}
\left| \T \setminus \bigcup_{i=1}^N \T_i \right| \leq \delta^{\frac{\epsilon}{2}} |\T|,
\end{equation}
in which case we stop, or the set $\T \setminus \bigcup_{i=1}^N \T_i$ satisfies the conditions of Theorem \ref{thm:projection}. In the latter case Theorem \ref{thm:projection} gives us $\T_{N+1} \subset \T \setminus \bigcup_{i=1}^N \T_i$ with $\mu(\cE(A_{N+1},\epsilon)) \leq \delta^\epsilon$. By construction, $\T_{N+1}$ is disjoint with any of the $\T_i$, $i = 1,\cdots,N$. 

When this procedure ends write $\T_0 = \bigcup_{i=1}^N \T_i$. Then \eqref{eq:AprimeBIG} says $|\T \setminus \T_0| \leq \delta^\frac{\eps}{2} |\T|$. Moreover, since the $\T_i$'s are disjoint, $|\T_0| = \sum_{i=1}^N |\T_i|$.

Set $a_i = \frac{|\T_i|}{|\T_0|}$. We claim that
\[\cE(\T,\frac{\eps}{3}) \subset \bigcup_I \bigcap_{i \in I} \cE(\T_i,\eps),\]
where the index set $I$ runs over subsets of $\{1, 2, \cdots, n \}$ with $\sum_{i \in I}a_i \geq \delta^\frac{\eps}{2}$. Since $\mu(\cE(\T_i))) \le \delta^\eps$ for all $i$, the desired upper bound $\mu(\cE(\T, \frac{\eps}{3})) \le \delta^{\eps/2}$ then follows immediately from Markov's inequality applied to the event $\sum a_i \one_{\T_i}$ (or \cite[Lemma 20]{he2020orthogonal}).

We will now show the claim. Let $z_0 \in \cE(A,\frac{\eps}{3})$, so there exists $\T' \subset \T$ with $|\T'| \geq \delta^\frac{\eps}{3} |\T|$ and $|\pi_{z_0} (\T')|_\delta \leq \delta^{-s - \frac{\eps}{3}}$. Consider the index set $I$ defined as
\[I = \{1 \le i \le n \mid |\T' \cap \T_i| \geq \delta^\eps |\T_i|\}.\]
We have
\begin{align*}
\delta^{\eps/2} |\T| &\le |\T'| - |\T \setminus \T_0| \\
&\leq \sum_{i=1}^n |\T' \cap \T_i|\\
&\lesim \sum_{i \in I} |\T_i| + \sum_{i \notin I} \delta^\eps |\T_i| \\
&\lesim \sum_{i\in I} a_i |\T| + \delta^\eps |\T|
\end{align*}
Hence $\sum_{i \in I}a_i \geq \delta^\frac{\eps}{2}$. On the other hand, for all $i \in I$, since
\begin{equation*}
    |\pi_{z_0}(\T' \cap \T_i))|_\delta \le |\pi_{z_0} (\T')|_\delta \le \delta^{-s-\frac{\eps}{3}},
\end{equation*}
we have $z_0 \in \cE(\T_i,\eps)$ for all $i \in I$. This finishes the proof of the claim.
\end{proof}

\subsection{Proof of Proposition \ref{prop:improved_incidence_weaker}}
\textit{This subsection is based on Section A.7 of \cite{orponen2021hausdorff}.}

We restate Proposition \ref{prop:improved_incidence_weaker}.

\begin{prop}\label{prop:improved_incidence_weaker'}
    Given $0 \le k < d-1$, $0 \le s < k+1$, $\tau, \kappa > 0$, there exist $\eta(s, k, \kappa, \tau, d) > 0$ and $\delta_0 (s, k, \kappa, \tau, d) > 0$ such that the following holds for all $\delta \in (0, \delta_0]$.
    
    Let $\bY \subset (\delta \cdot \Z) \cap [0,1)$ be a $(\delta, \tau, \delta^{-\eta})$-set, and for each $\by \in \bY$, assume that $\bX_{\by} \subset (\delta \cdot \Z)^{d-1} \cap [0,1)^{d-1}$ is a $(\delta, \kappa, \delta^{-\eta}, k)$-set with cardinality $\ge \delta^{-s+\eta}$. Let
    \begin{equation*}
        \bZ = \bigcup_{\by \in \bY} \bX_{\by} \times \{ \by \}.
    \end{equation*}
    For every $\bz \in \bZ$, assume that $\T(\bz)$ is a set of $\delta$-tubes each making an angle $\ge \frac{1}{100}$ with the plane $y = 0$ with $|\T(\bz)| \ge \delta^{-s+\eta}$ such that $\bz \in T$ for all $T \in \T(\bz)$. Then $|\T| \ge \delta^{-2s-\eta}$, where $\T = \cup_{\bz \in \bZ} \T(\bz)$.
\end{prop}

\begin{proof}
    Let $A \leapp B$ denote $A \le C \delta^{-C \eta} B$ for some absolute constant $C \ge 1$. A $(\delta, u, m)$-set stands for a $(\delta, u, C\delta^{-C\eta}, m)$-set.

    First, without loss of generality, assume $|\T(\bz)| = \delta^{-s+\eta}$ for each $\bz \in \bZ$.
    
    Suppose $|\T| \le \delta^{-2s-\eta}$. Let 
    \begin{equation*}
        \T(\by) = \bigcup_{\bx \in \bX_{\by}} \T(\bx,\by).
    \end{equation*}
    Since each tube in $\T(\by)$ has angle $\ge \frac{1}{100}$ with the plane $y = 0$, it only intersects $O(1)$ many $\delta$-balls $(\bx, \by)$ for a given $\by$. Since $|\T(\bx, \by)| \geapp \delta^{-s}$ for each $\bx \in \bX_\by$, we get $|\T(\by)| \geapp \delta^{-s} |\bX_{\by}|$. With the counter-assumption $|\T| \leapp \delta^{-2s}$, this forces $|\bX_{\by}| \leapp \delta^{-s}$ for each $\by \in \bY$. On the other hand, $|\bX_{\by}| \geapp \delta^{-s}$ and so $|\T| \app \delta^{-2s}$.

    Now, we check that $\T(\by)$ is a $(\delta, \kappa, \delta^{-O(\eta)}, k)$-set. Pick a $(r, k+1)$-plane $H$. We claim that either $\T(\by) \cap H = \emptyset$ or $H(y = \by)$ is contained in a $(O(r), k)$-plate. Indeed, if $H(y = \by)$ is not contained within a $(Cr, k)$-plate, then $H$ is contained within the $O(C^{-1})$-neighborhood of the plane $y = \by$, which means that $H$ cannot contain any tubes of $\T(\by)$ if $C$ is large enough (since the tubes of $\T(\by)$ have angle $\ge \frac{1}{100}$ with that plane). Thus, we may assume $H(y = \by)$ is contained within a $(Cr, k)$-plate, which means
    \begin{multline*}
        |\T(\by) \cap H| = |\bigcup_{\bx \in \bX_\by \cap H} \T(\bx, \by) \cap H| \\
        \le |\bX_\by \cap H| \cdot \delta^{-s+\eta} \leapp |\bX_\by| r^\kappa \cdot \delta^{-s+\eta} \leapp r^\kappa |\T(\by)|.
    \end{multline*}

    Since $|\T(\by)| \app |\T|$ for each $\by \in \bY$, there is a subset $\oT \subset \T$ such that $|\T| \app |\oT|$ and each $T \in \oT$ belongs to $\app |\bY|$ of the sets $\T(\by)$. We show $\oT$ is a $(\delta, \kappa, \delta^{-O(\eta)}, k)$-set. Indeed, given a $(r, k+1)$-plate $H$, we have
    \begin{multline*}
        |\oT \cap H| \app \sum_{T \in \oT \cap H} \frac{1}{|\bY|} \sum_{\by \in \bY} \one_{\T(y)} (T)\\
        \leapp \frac{1}{|\bY|} \sum_{\by \in \bY} |\oT(\by) \cap H| \leapp \frac{1}{|\bY|} \sum_{\by \in \bY} 
 r^\kappa |\T(y)| \le r^\kappa |\oT|.
    \end{multline*}
    Finally, we refine $\bY$ further: since
    \begin{equation*}
        \sum_{\by \in \bY} |\oT \cap \T(y)| = \sum_{T \in \oT} |\{ \by \in \bY : T \in \T(\bY) \}| \app |\oT| |\bY|,
    \end{equation*}
    we can find a subset $\obY \subset \bY$ with the property that $|\oT(y)| := |\oT \cap \T(y)| \app |\oT|$ for each $y \in \obY$. Also, $\obY$ is still a $(\delta, \tau, \delta^{-O(\eta)})$-set.

    Now for each $\by \in \obY$, the large subset $\oT(y) \subset \oT$ has small covering number $|\bX_{\by}| \app \delta^{-s}$. On the other hand, $|\oT| \app \delta^{-2s}$. This contradicts Theorem \ref{thm:projection} if $\eta$ is chosen sufficiently small in terms of the $\eps$ of the theorem.
\end{proof}


\section{Improved incidence estimates for regular sets}\label{sec:regular sets}
In this section, we prove a version of Theorem \ref{thm:main} for regular sets.

\begin{defn}
    Let $\delta \in 2^{-2\N}$ be a dyadic number. Let $C, K > 0$, and let $0 \le s \le d$. A non-empty set $\cP \subset \cD_\delta$ is called $(\delta, s, C, K)$-regular if $\cP$ is a $(\delta, s, C, 0)$-set, and
    \begin{equation*}
        |\cP|_{\delta^{1/2}} \le K \cdot \delta^{-s/2}.
    \end{equation*}
\end{defn}

\begin{theorem}\label{thm:improved_incidence}
    For any $0 \le s, k < d-1$, $\max(s, k) < t \le d$, $\kappa > 0$, there exists $\eps(s, t, \kappa, k, d) > 0$ such that the following holds for all small enough $\delta \in 2^{-\N}$, depending only on $s, t, \kappa, k, d$. Let $\cP \subset \cD_\delta$ be a $(\delta, t, \delta^{-\eps}, \delta^{-\eps})$-regular set. Assume that for every $p \in \cP$, there exists a $(\delta, s, \delta^{-\eps}, 0)$ and $(\delta, \kappa, \delta^{-\eps}, k)$-set $\cT(p) \subset \cT$ with $|\T(p)| = M$ such that $T \cap p \neq \emptyset$ for all $T \in \cT(p)$. Then $|\cT| \ge M \delta^{-s-\eps}$.
\end{theorem}

\subsection{Initial reductions}
\textit{This subsection is based on Sections 6 and A.1-A.3 of \cite{orponen2021hausdorff}.}

In this section, let $A \leapp B$ denote $A \le C \delta^{-C \eps} B$ for some constant $C \ge 1$ depending only on $s, t, \kappa, k, d$. Also, let $\cP \cap Q := \{ p \in \cP : p \subset Q \}$.

The proof will be based on contradiction, so assume $|\T| \le M\delta^{-s-\eps}$. Let's rename $\cP$ to $\cP_0$ and $\cT$ to $\cT_0$, reserving $\cP, \cT$ for the use of Proposition \ref{prop:nice_tubes}.

By Corollary \ref{cor:easy_est}, we have $1 \geapp (M\delta^s)^{\frac{t-s}{d-1-s}}$, so $M \leapp \delta^{-s}$ and $|\T| \leapp \delta^{-2s}$. But $\T(p)$ is a $(\delta, s, \delta^{-\eps})$-set, so $M \app \delta^{-s}$. Finally, by Lemma \ref{lem:small sets}, we may assume $|\cP_0| \app \delta^{-t}$ (passing to subsets will preserve the $(\delta, t, \delta^{-\eps}, \delta^{-\eps})$-regularity of $\cP_0$).

The next reduction will make the value $|\cP_0 \cap Q|$ uniform for different $Q \in \cD_\Delta (\cP_0)$. Let $\cQ_0 = \cD_\Delta (\cP_0)$. By $(\delta, t, \delta^{-\eps}, \delta^{-\eps})$-regularity of $\cP_0$, we have $|\cQ_0| \leapp \Delta^{-t}$. On the other hand, since $\cP_0$ is a $(\delta, t)$-set, we have that for all $Q \in \cQ_0$,
\begin{equation}\label{eqn:light squares}
    |\cP_0 \cap Q| \leapp \Delta^{-t}
\end{equation}
This means $|\cQ_0| \geapp \Delta^{-t}$. Hence, $|\cQ_0| \app \Delta^{-t}$. Now using \eqref{eqn:light squares} again and $|\cP_0| \app \Delta^{-2t}$, there exists $\cQ_0' \subset \cQ_0$ with $|\cQ_0'| \geapp |\cQ_0|$ such that for each $Q \in \cQ_0'$,
\begin{equation}\label{eqn:good squares}
    |\cP_0 \cap Q| \app \Delta^{-t}
\end{equation}
Using \eqref{eqn:good squares}, we quickly check that $\cQ_0'$ is a $(\Delta, t)$-set. Indeed, for $r \in (\Delta, 1)$ and $Q_r \in \cD_r$, we have
\begin{equation}\label{eqn:cQ0 dim check}
    |\cQ_0' \cap Q_r| \overset{\eqref{eqn:good squares}}{\app} \Delta^t \cdot |\cP_0 \cap Q_r| \leapp \Delta^t \cdot |\cP_0| \cdot r^t \app |\cQ_0'| \cdot r^t.
\end{equation}
(The second inequality uses that $\cP_0$ is a $(\Delta, t)$-set.)
Let $\cP_0' = \bigcup_{Q \in \cQ_0'} \cP \cap Q$; then $|\cP_0'| \app |\cP|$ and $|\cP_0' \cap Q| = |\cP_0 \cap Q| \app \Delta^{-t}$ for $Q \in \cP_0'$. Apply Proposition \ref{prop:nice_tubes} to find $\cP \subset \cP_0', \T(p) \subset \T_0 (p)$, $\cT_\Delta$, $\T$, and $\T_Q$. Let $\cQ = \cD_\Delta (\cP)$.

\textbf{Claim.} $M_\Delta \app \Delta^{-s}$ and $|\T_\Delta| \leapp \delta^{-s}$.

\textit{Proof.} By Proposition \ref{prop:nice_tubes}\ref{item3}, we know that $(\cD_\Delta (\cP), \T_\Delta)$ is $(\Delta, s, C_\Delta^1, \kappa, C_\Delta^2, M_\Delta)$-nice, so $M_\Delta \geapp \Delta^{-s}$. Also, by Corollary \ref{cor:easy_est}, we have that
\begin{equation}\label{eqn:lower T_Delta}
    |\T_\Delta| \geapp M_\Delta \delta^{-s/2} \cdot (M_\Delta \delta^{s/2})^{\frac{t-s}{d-1-s}}.
\end{equation}
Next, for any $Q \in \cQ$, we know that $(S_Q (\cP \cap Q), \T_Q)$ is $(\Delta, s, C_Q^1, \kappa, C_Q^2, M_Q)$-nice. Recall that
\begin{equation*}
    S_Q (\cP \cap Q) = \{ S_Q (p) : p \in \cP, p \subset Q \} \subset \cD_\Delta.
\end{equation*}
We also know $|\cP \cap Q| \app |\cP_0' \cap Q| \app \delta^{-t/2}$ and $\cP$ is a $(\delta, t)$-set, so by a similar check to \eqref{eqn:cQ0 dim check}, we get that $S_Q (\cP \cap Q)$ is a $(\Delta, t)$-set. Thus by Corollary \ref{cor:easy_est}, we have
\begin{equation*}
    |\T_Q| \geapp M_Q \cdot \delta^{-s/2}
\end{equation*}
But by our counterassumption $|\T_0| \leapp \delta^{-2s}$, we get from \eqref{eqn:item6} in Proposition \ref{prop:nice_tubes} and $M \geapp \delta^{-s}$,
\begin{equation*}
    \delta^{-2s} \geapp \frac{|\T_\Delta|}{M_\Delta} \cdot \frac{|\T_Q|}{M_Q} \cdot M \geapp \frac{|\T_\Delta|}{M_\Delta} \cdot \delta^{-3s/2}.
\end{equation*}
Thus, $|\T_\Delta| \leapp M_\Delta \delta^{-s/2}$. Substitute into \eqref{eqn:lower T_Delta} to get
\begin{equation*}
    \delta^{-s/2} \geapp \frac{|T_\Delta|}{M_\Delta} \geapp \delta^{-s/2} \cdot (M_\Delta \delta^{s/2})^{\frac{t-s}{d-1-s}}.
\end{equation*}
Thus, $M_\Delta \delta^{s/2} \leapp 1$, so $M_\Delta \leapp \Delta^{-s}$ and $|\T_\Delta| \leapp \delta^{-s}$, proving the Claim. \qed

Thus, we get the higher-dimensional analogues of properties (H1-2), (G1-4) of \cite{orponen2021hausdorff} except we only know $|\T| \leapp \delta^{-2s}$ and not $|\T| \geapp \delta^{-2s}$. But this is not a limitation. We repeat and relabel these properties here:
\begin{enumerate}[(G1)]
    \item\label{g1} $|\cQ| \app \Delta^{-t}$ and $|\cP \cap Q| \app \Delta^{-t}$ for all $Q \in \cQ$.

    \item\label{g2} Every tube $\bT \in \T_\Delta$ satisfies $|\T \cap \bT| \leapp \delta^{-s}$.

    \item\label{g3} For every square $Q \in \cQ$, there corresponds a $(\Delta, s, 0)$-set and $(\Delta, \kappa, k)$-set $\T_\Delta (Q) \subset \T_\Delta$ of cardinality $\app M_\Delta \app \Delta^{-s}$ such that $\bT \cap Q \neq \emptyset$ for all $\bT \in \T_\Delta (Q)$.

    \item\label{g4} $|\T| \leapp \delta^{-2s}$  and $|\T_\Delta| \app \Delta^{-2s}$.

    \item\label{g5} For $\bT \in \T_\Delta (Q)$, we have
    \begin{equation*}
        |\{ (p, T) \in (\cP \cap Q) \times \T : T \in \T(p) \cap \bT \}| \geapp \Delta^{-s-t}.
    \end{equation*}
\end{enumerate}

Item \ref{g1} follows from Proposition \ref{prop:nice_tubes}\ref{item1}.

Item \ref{g3} follows from Proposition \ref{prop:nice_tubes}\ref{item3} and Claim.

Item \ref{g4} follows from $|\T_0| \leapp \delta^{-2s}$ and Claim.

Item \ref{g5} follows from Proposition \ref{prop:nice_tubes}\ref{item4} and the estimation $M \cdot |\cP \cap Q| / |\T_\Delta (Q)| \app \Delta^{-s-t}$, which uses item \ref{g1}, item \ref{g3}, and the fact $M \app \delta^{-s}$ we proved at the beginning of the argument.

Item \ref{g2} follows from Proposition \ref{prop:nice_tubes}\ref{item21}, the fact that a given $\delta$-tube lies in $\lesim 1$ many of the $\bT$'s in $\T_\Delta$, and item \ref{g4}:
\begin{equation*}
    \delta^{-2s} \geapp |\T| \gesim \sum_{\bT \in \T_\Delta} |\T \cap \bT| \sim |\T_\Delta| \cdot \bN \app \Delta^{-2s} \cdot \bN.
\end{equation*}

\subsection{Transferring angular non-concentration to ball non-concentration}
\textit{This subsection is based on Section A.4 of \cite{orponen2021hausdorff}.}

We first recall some notation. For a unit vector $\sigma \in \R^d$, define $\pi_\sigma (\vec{v}) := \vec{v} - (\vec{v} \cdot \sigma) \sigma$ to be the orthogonal projection to the orthogonal complement of $\sigma$. For a $\delta$-tube $T$, let $\sigma(T) \in S^{d-1}$ denote the direction of $T$.

In this subsection, we fix a $Q \in \cD_\Delta (\cP)$. Our goal is to show that for many $\bT \in \T_\Delta (Q)$, the $\Delta^{-1}$-rescaled version of $\pi_{\sigma(\bT)} (\cup (\cP \cap Q))$ contains a $(\Delta, s, 0)$ and $(\Delta, \kappa', k)$-set for some $\kappa' > 0$. This is the content of the next Proposition \ref{prop:projections}, which is a higher-dimensional extension of Lemma A.6 of \cite{orponen2021hausdorff}. The proposition encodes the following principle: If we have a set of orthogonal projections in $\Gr(d, d-1)$ (which we view as $S^{d-1}$) that don't concentrate around $k$-planes, and we have a $t$-dimensional set $X$ with $t > k$, then many projections of $X$ will not concentrate around $k$-planes.

\begin{prop}\label{prop:projections}
    Let $0 \le \max(s, k) < t \le d$, $\kappa > 0$, and $\bA, \bB > 0$. Let $\cP$ be a $(\Delta, t, \Delta^{-\bA \eps})$-set in $[0,1)^d$, and let $\Gamma \subset S^{d-1}$ be a $(\Delta, s, \Delta^{-\bA \eps}, 0)$-set and $(\Delta, \kappa, \Delta^{-\bA \eps}, k)$-set. There exists a subset $\Sigma \subset \Gamma$ with $|\Sigma| \ge \frac{1}{2} |\Gamma|$ such that the following holds for all $\sigma \in \Sigma$: if $\cP' \subset \cP$ is an arbitrary subset of cardinality $|\cP'| \ge \Delta^{\bB \eps} |\cP|$, then $\pi_\sigma(\cP')$ contains a $(\Delta, \frac{1}{k+1} \min(\frac{t-k}{2}, \kappa), \Delta^{-\bC(\bA + \bB) \eps}, k)$ and $(\Delta, s, \Delta^{-\bC(\bA + \bB)}, 0)$-set, where $\bC \ge 1$ is absolute depending on $k$.
\end{prop}



\begin{proof}
    We will use a variation of the energy argument due to Kaufman \cite{kaufman1968hausdorff} in the form used to prove \cite[Lemma A.6]{orponen2021hausdorff}. An alternate proof can follow \cite[Lemma 27]{he2020orthogonal}, but this approach would give weaker bounds.

    Let $\mu$ be the $\Delta$-discretized probability measure corresponding to $\cP$,
    \begin{equation*}
        \mu := \frac{1}{|\cP|} \sum_{q \in \cP} \frac{\cL^d|_q}{\Delta^d},
    \end{equation*}
    where $\cL^d$ is $d$-dimensional Lebesgue measure.
    Since $\cP$ is a $(\Delta, t, \Delta^{-\bA \eps})$-set, we have $\mu(B(x, r)) \leapp r^t$ for all $r > \delta$, and it's also true for $r < \delta$ since $\mu$ behaves like Lebesgue measure at small scales.
    We will choose a uniformly random $\sigma \in \Gamma$ and consider what happens to the energy of $\mu$ under projection by $\sigma$. By linearity of expectation and definition of energy,
    \begin{equation*}
        E_{s,1} := \E_\sigma [I_{s,1}^\Delta (\pi_\sigma \mu)] = \int \E_\sigma [(|\pi_\sigma (x_0 - x_1)| + \delta)^{-s}] \, d\mu(x_0) d\mu(x_1).
    \end{equation*}
    Since $\Gamma$ is a $(\Delta, s)$-set, we have $\E_\sigma [(|\pi_\sigma (x_0 - x_1)| + \Delta)^{-s}] \lesim (\log \Delta^{-1}) \cdot \Delta^{-\bA \eps} |x_0 - x_1|^{-s}$ (c.f. \cite{kaufman1968hausdorff}), and so $E_{s,1} \leapp I_{s,1}^0 (\mu) \leapp 1$ by Lemma \ref{lem:energy}(a) and $s < t$.

    Analogously, we have (let $\beta = \min(\kappa, \frac{t-k}{2})$):
    \begin{equation*}
        E_{\beta,k+1} := \E_\sigma [I_{\beta,k+1}^\Delta (\pi_\sigma \mu)] = \int \E_\sigma \left[\left(\left|\bigwedge_{i=1}^{k+1} \pi_\sigma (x_0 - x_i) \right| + \Delta \right)^{-\beta} \right] \, d\mu(x_0) \cdots d\mu(x_{k+1}).
    \end{equation*}
    Observe that
    \begin{equation*}
        \left|\bigwedge_{i=1}^{k+1} \pi_\sigma (x_0 - x_i)\right| = \left|\sigma \wedge \bigwedge_{i=1}^{k+1} \pi_\sigma (x_0 - x_i)\right| = \left|\sigma \wedge \bigwedge_{i=1}^{k+1} (x_0 - x_i)\right| = \left|\bigwedge_{i=1}^{k+1} (x_0 - x_i)\right| \cdot \rho,
    \end{equation*}
    where $\rho$ is the distance from $\sigma$ to the plane spanned by $x_1 - x_0$ through $x_{k+1} - x_0$.
    (The first equality follows since $\sigma$ is orthogonal to each $\pi_\sigma (x_0 - x_i)$. The second equality follows since $\wedge$ is multilinear and $\sigma \wedge \sigma = 0$. The third equality follows by the geometric definition of wedge product as a volume of a parallelepiped.)
    Thus, since $\Gamma$ is a $(\Delta, \kappa, k)$-set and $\beta \le \kappa$, we have
    \begin{align*}
        \E_\sigma \left[\left(\left|\bigwedge_{i=1}^{k+1} \pi_\sigma (x_0 - x_i) \right| + \Delta \right)^{-\beta} \right] &\lesim \sum_{\rho = 2^{-n} \in (\Delta, 1)} \Delta^{-\bA \eps} \cdot \rho^{\kappa-\beta} \left|\bigwedge_{i=1}^{k+1} (x_0 - x_i)\right|^{-\beta} \\
        &\lesim (\log \Delta^{-1}) \cdot \Delta^{-\bA \eps} \left|\bigwedge_{i=1}^{k+1} (x_0 - x_i)\right|^{-\beta},
    \end{align*}
    and so $E_{\beta, k+1} \leapp I_{\beta, k+1}^0 (\mu) \leapp 1$ by Lemma \ref{lem:energy}(a) and $\beta < t-k$.
    
    Consequently, by Markov's inequality we can find $\Sigma \subset \Gamma$ with $|\Sigma| \ge \frac{1}{2} |\Gamma|$ such that for each $\sigma \in \Sigma$, we have $I_{s,1}^\Delta (\pi_\sigma \mu) \le \Delta^{-2C_1 \bA \eps}$ and $I_{\beta,k+1}^\Delta (\pi_\sigma \mu) \le \Delta^{-2C_1 \bA \eps}$. For any $\cP' \subset \cP$ with $|\cP'| \ge \Delta^{\bB \eps} |\cP|$, we have $I_{s,1}^\Delta (\pi_\sigma \mu_{\cP'}) \le \Delta^{-(2C_1 \bA + 2\bB) \eps}$ and $I_{\beta,k+1}^\Delta (\pi_\sigma \mu_{\cP'}) \le \Delta^{-(2C_1 \bA + (k+2) \bB) \eps}$, where $\mu_{\cP'} = \frac{1}{\mu(\cP')} \mu|_{\cP'}$ is the renormalized restriction of $\mu$ to $\cP'$. Then Lemma \ref{lem:energy}(b) gives the desired conclusion.
\end{proof}

\subsection{Finding a special $\Delta$-tube}
\textit{This subsection is based on Section A.4 of \cite{orponen2021hausdorff}.}

Apply Proposition \ref{prop:projections} to $S_Q (\cP \cap Q)$, which is a $(\Delta, t)$-set using \ref{g1} and the fact that $\cP$ is a $(\Delta, t)$-set. Define
\begin{equation*}
    \T_\Delta^\pi (Q) = \{ \bT \in \T_\Delta (Q) : \sigma(\bT) \in \Sigma(Q) \}, \qquad Q \in \cQ,
\end{equation*}
where $\Sigma(Q)$ is the set of good directions of cardinality $|\Sigma(Q)| \ge \frac{1}{4} |\sigma(Q)| \sim |\T_\Delta (Q)|$ (since for a given direction, there are $\sim 1$ many $\Delta$-tubes in that direction that intersect $Q$). Then $\T_\Delta^\pi (Q), Q \in \cQ$ remain $(\Delta, s)$-sets of cardinality $\app \Delta^{-s}$, and so the properties \ref{g1}-\ref{g5} remain valid upon replacing $\T_\Delta (Q)$ with $\T_\Delta (Q)$. (We leave $\T_\Delta$ unchanged, so only \ref{g3} and \ref{g5} are affected.) Thus, $\cP \cap Q$ for $Q \in \cQ$ and their large subsets have nice projections in the sense of Proposition \ref{prop:projections} in every direction orthogonal to the tubes $\bT \in \T_\Delta^\pi (Q)$. We keep the symbol ``$\pi$'' as a reminder of this fact.

The next goal is to find a tube $\bT_0$ with the following properties:
\begin{enumerate}[(P1)]
    \item\label{P1} The set $\{ Q \in \cQ : \bT_0 \in \T_\Delta^\pi (Q) \}$ contains a $(\Delta, t-s)$-subset, which we denote $\bT_0 (\cQ)$.

    \item\label{P2} $|\T \cap \bT_0| \leapp \Delta^{-2s}$.
    
    \item\label{P3} For each $Q \in \bT_0 (\cQ)$, there exists a subset $\cP_Q \subset \cP \cap Q$ such that
    \begin{equation*}
        |\cP_Q| \app \Delta^{-t} \text{ and } |\T(p) \cap \bT_0| \app \Delta^{-s} \text{ for all } p \in \cP_Q.
    \end{equation*}

    \item\label{P4} Let $\sigma$ be the direction of $\bT$. Then $\pi_\sigma (S_Q (\cP_Q))$ contains a $(\Delta, \kappa', k)$-set with cardinality $\geapp \Delta^{-s}$, where $\kappa' := \frac{1}{k+1} \min(\frac{t-k}{2}, \kappa)$.
\end{enumerate}
To get \ref{P1}- \ref{P3}, we will mostly follow Section A.4 of \cite{orponen2021hausdorff}. (We have used the fact that $\T$ is a $(\Delta, \kappa, \Delta^{-\bA \eps}, k)$-set, by converting it into ball concentration near $(k+1)$-planes in Proposition \ref{prop:projections}; the rest of the argument will only use the fact that $\T$ is a $(\Delta, s, \Delta^{-\bA \eps}, 0)$-set.) First, we refine the sets $\cQ$ and $\T_\Delta^\pi (Q)$ further to ensure that the family $\{ Q \in \cQ : \bT \in \T_\Delta^\pi (Q) \}$ will be $(\Delta, t-s)$-sets for $\bT \in \T_\Delta$. Indeed, we have
\begin{align*}
    \sum_{\bT \in \T_\Delta} \sum_{\substack{Q, Q' \in \cQ\\Q \neq Q'}} \frac{\one_{\T_\Delta^\pi (Q) \cap \T_\Delta^\pi (Q')} (\bT)}{d(Q, Q')^{t-s}} &= \sum_{Q, Q' \in \cQ, Q \neq Q'} \frac{|\T_\Delta^\pi (Q) \cap \T_\Delta^\pi (Q')|}{d(Q, Q')^{t-s}} \\
        &\leapp \sum_{Q, Q' \in \cQ, Q \neq Q'} \frac{1}{d(Q, Q')^t} \leapp \Delta^{-2t}.
\end{align*}
The first $\leapp$ inequality uses the fact that $\T_\Delta^\pi (Q)$ is a $(\Delta, s)$-set of tubes with $|\T_\Delta^\pi (Q)| \app \Delta^{-s}$, and the second $\leapp$ inequality uses the fact that $\cQ$ is a $(\Delta, t)$-set with $|\cQ| \app \Delta^{-t}$.

Thus, by Markov's inequality, for a fixed absolute large constant $C \ge 1$, we have
\begin{equation}\label{eqn:exception}
    \sum_{\substack{Q, Q' \in \cQ \\Q \neq Q'} }\frac{\one_{\T_\Delta^\pi (Q) \cap \T_\Delta^\pi (Q')} (\bT)}{d(Q, Q')^{t-s}} \ge \Delta^{-C\eps + 2(s-t)}
\end{equation}
can only hold for $\leapp \Delta^{C\eps-2s}$ many tubes $\bT \in \T_\Delta$.

\textbf{Claim 2.} If $C \ge 1$ is sufficiently large, then there exists a subset $\bcQ \subset \cQ$ with $|\bcQ| \ge \frac{1}{2} |\cQ|$ such that for all $Q_0 \in \bcQ$, at most half of the tubes $\bT \in \T_\Delta^\pi (Q_0)$ satisfy \eqref{eqn:exception}.

\textit{Proof.} Suppose this is not true: there exists a set $\cQ_{\text{bad}}$ such that for $Q_0 \in \cQ$, at least $\frac{1}{2} |\T_\Delta^\pi (Q_0)|$ many tubes $\bT \in \T_\Delta^\pi (Q_0)$ satisfy \eqref{eqn:exception}. Then apply Corollary \ref{cor:easy_est} to $\cQ_{\text{bad}}$ and the bad parts of $\T_\Delta^\pi (Q_0)$, which are still $(\Delta, s$)-sets. By Corollary \ref{cor:easy_est}, we have $\geapp \Delta^{-2s}$ many $\Delta$-tubes in $\T_\Delta$ that satisfy \eqref{eqn:exception}. But we observed before that \eqref{eqn:exception} only holds for $\leapp \Delta^{C\eps-2s}$ many tubes $\bT \in \T_\Delta$. By choosing $C$ large enough (and $\delta$ small enough), we obtain a contradiction. \qed

In what follows, the $C$ in Claim 2 will be absorbed into the $\leapp$ notation. Replace $\cQ$ by $\bcQ$ and $\T_\Delta^\pi (Q)$ by their good subsets without changing notation. All of the properties \ref{g1}-\ref{g5} remain valid, and
\begin{equation}\label{eqn:good energy}
    \sum_{\substack{Q, Q' \in \cQ\\Q \neq Q'}} \frac{\one_{\T_\Delta^\pi (Q) \cap \T_\Delta^\pi (Q')} (\bT)}{d(Q, Q')^{t-s}} \leapp \Delta^{2(s-t)}, \qquad \bT \in \T_\Delta^\pi (Q_0), Q_0 \in \cQ.
\end{equation}
Now, we will find $\bT_0 \in \T_\Delta$ satisfying
\begin{equation}\label{eqn:card s-t}
    |\bT_0 (\cQ)| := |\{ Q \in \cQ : \bT_0 \in \T_\Delta^\pi (Q) \}| \geapp \Delta^{s-t}.
\end{equation}
Indeed, the average tube works, because of the following: since $|\T_\Delta| \app \Delta^{-2s}, |\cQ| \app \Delta^{-t}$, and $|\T_\Delta^\pi (Q)| \app \Delta^{-s}$ (by \ref{g4}, \ref{g1}, \ref{g3} respectively), we have
\begin{equation*}
    \frac{1}{|\T_\Delta|} \sum_{\bT \in \T_\Delta} |\{ Q \in \cQ : \bT_0 \in \T_\Delta^\pi (Q) \}| = \frac{1}{|\T_\Delta|} \sum_{Q \in \cQ} |\T_\Delta^\pi (Q)| \app \frac{|\cQ| \cdot \Delta^{-s}}{\Delta^{-2s}} \app \Delta^{s-t}.
\end{equation*}
Now, we show that using \eqref{eqn:good energy} and \eqref{eqn:card s-t}, the family $\bT_0 (\cQ) \subset \{ Q \in \cQ : Q \cap \bT_0 \neq \emptyset \}$ contains a $(\Delta, t-s)$-set, which proves item \ref{P1}. Indeed, rewrite \eqref{eqn:good energy} as
\begin{equation}\label{eqn:good energy'}
    \sum_{\substack{Q, Q' \in \bT_0 (\cQ)\\Q \neq Q'}} \frac{1}{d(Q, Q')^{t-s}} \leapp \Delta^{2(s-t)}.
\end{equation}
Let
\begin{equation}\label{eqn:bt0'}
    \bT_0' (\cQ) := \{ Q \in \bT_0 (\cQ) : \sum_{Q' \in \bT_0 (\cQ) \setminus \{ Q \}} d(Q, Q')^{s-t} \le \Delta^{s-t-C\eps} \}.
\end{equation}
By Markov's inequality on \eqref{eqn:good energy'}, we have $|\bT_0 (\cQ) \setminus \bT_0' (\cQ)| \leapp \Delta^{s-t+C\eps}$. Hence, if $C$ is chosen large enough, we have by \eqref{eqn:card s-t}, $|\bT_0' (\cQ)| \ge \frac{1}{2} |\bT_0 (\cQ)| \geapp \Delta^{s-t}$. By Markov's inequality on \eqref{eqn:bt0'}, we have that for all $Q \in \bT_0' (\cQ)$ and $r \in (\delta, 1)$,
\begin{equation*}
    |\{ Q' \in \bT_0 (\cQ) : d(Q, Q') \le r \}| \le \Delta^{s-t-C\eps} r^{t-s}.
\end{equation*}
Thus, $\bT_0' (\cQ)$ is a $(\Delta, t-s)$-set, which proves \ref{P1}.

To get \ref{P2}, we use \ref{g2}.
\begin{equation*}
    |\T \cap \bT_0| \leapp \delta^{-s} = \Delta^{-2s}.
\end{equation*}
By \ref{g5}, we have
\begin{equation}\label{eqn:incidences}
    |\{ (p, T) \in (\cP \cap Q) \times \T : T \in \T(p) \cap \bT_0 \}| \geapp \Delta^{-s-t}.
\end{equation}
Fix $Q \in \bT_0 (\cQ)$. Since $|\cP \cap Q| \app \Delta^{-t}$ by \ref{g1} and $|\T(p) \cap \bT_0| \leapp \Delta^{-s}$ since $\T(p)$ is a $(\delta, s)$-set, we use \eqref{eqn:incidences} to find a subset $\cP_Q \subset \cP \cap Q$ with
\begin{equation*}
    |\cP_Q| \app |\cP \cap Q| \app \Delta^{-t} \text{ and } |\T(p) \cap \bT_0| \app \Delta^{-s} \text{ for all } p \in \cP_Q.
\end{equation*}
This verifies \ref{P3}. Finally, we get \ref{P4} by $|\cP_Q| \ge \Delta^{\bB \eps} |\cP \cap Q|$ for some constant $\bB \ge 1$ and Proposition \ref{prop:projections}.


\subsection{Product-like structure}
\textit{This subsection is based on Section A.6 of \cite{orponen2021hausdorff}.}

Our goal is to find a product-type structure and apply Proposition \ref{prop:improved_incidence_weaker}. Choose coordinates such that the $y$-axis is in the direction of $\bT_0$, and let $\pi(\bx, y) := \bx \in \R^{d-1}$ denote the orthogonal projection to the orthogonal complement of the $y$-axis. Define the function $\Delta^{-1} (\bx, y) = (\Delta^{-1} \bx, y)$. If $T \in \bT_0$, then $\Delta^{-1} T$ is roughly a $\Delta$-tube: it is contained in some $C\Delta$-tube and contains a $c\Delta$-tube for some universal constants $c, C > 0$. This technicality will not cause issues in what follows.

For each $Q \in \bT_0 (\cQ)$, let $\by_Q \in \Delta \cdot \Z \cap [0, 1)$ be a point such that the plane $y = \by_Q$ intersects $Q$. By \ref{P1}, we know that $\bY = \{ \by_Q : Q \in \bT_0 (\cQ) \}$ is a $(\Delta, t-s)$-set. By \ref{P4}, we know that for each $\by \in \bY$ that $\pi(\Delta^{-1} (\cP \cap Q))$ contains a $(\Delta, \kappa', k)$-set $\bX_\by'$ with cardinality $\geapp \Delta^{-s}$. Let $\bX_\by \subset (\Delta \cdot \Z)^d \cap [0, 1]^d$ that is $\bX_\by'$ rounded to the nearest multiple of $\Delta$.

Now, let $L = (\Delta \cdot \Z) \cap B(0, \Delta (\sqrt{d}+1))$ and $\T(\bZ) = \{ \sigma(T) + x : T \in \T \cap \bT_0, x \in L \}$. Clearly, $|\T(\bZ)| \lesim_d |\T \cap \bT_0| \leapp \Delta^{-2s}$ by \ref{P2}. On the other hand, we show that $|\T(\bz)| := |\{ \bT \in \T(\bZ) : \bz \in \bT \}| \geapp \Delta^{-s}$ for any $(\bx, \by) \in \bZ$. This follows since $\bz = (\bx, \by_Q)$ for some $Q$ and $\bx \in \bX_\by$. Let $p \in \cP_Q$ such that $d(\pi(\Delta^{-1} p), \bx) \le \Delta$. We know $d((\pi(\Delta^{-1} p), \by_Q), \Delta^{-1} p) \le \Delta$ since $Q$ has diameter $\Delta$, so by triangle inequality, we have $d(\Delta^{-1} p, \bz) \le (\sqrt{d}+1)\Delta$. Thus, $\T(\bz)$ contains $\{ \sigma(T) + x : T \in \T(p) \cap \bT_0 \}$ for some suitable $x \in L$. By \ref{P3}, we get the desired cardinality estimate $|\T(\bz)| \app \Delta^{-s}$.

Finally, we apply Proposition \ref{prop:improved_incidence_weaker} to the sets $\bZ$ and $\T(\bZ)$ to obtain a contradiction if $\eps > 0$ is sufficiently small. This proves Theorem \ref{thm:improved_incidence}.

\section{Improved incidence estimates for general sets}\label{sec:main thm general}
In this section, we will prove the following refinement of Theorem \ref{thm:main}, following Sections 7-9 of \cite{orponen2021hausdorff}.



\begin{theorem}\label{thm:main'}
    For any $0 \le k < d-1$, $0 \le s < k+1$, $s < t \le d$, $\kappa > 0$, there exist $\eps(s, t, \kappa, k, d) > 0$ and $\eta(s, t, \kappa, k, d) > 0$ such that the following holds for all small enough $\delta \in 2^{-\N}$, depending only on $s, t, \kappa, k, d$. Let $\cP \subset \cD_\delta$ be a $(\delta, t, \delta^{-\eps})$-set with $\cup \cP \subset [0, 1)^d$, and let $\cT \subset \cT^\delta$ be a family of $\delta$-tubes. Assume that for every $p \in \cP$, there exists a $(\delta, s, \delta^{-\lambda}, 0)$ and $(\delta, \kappa, \delta^{-\lambda}, k)$-set $\cT(p) \subset \cT$ with $|\T(p)| = M$ such that $T \cap p \neq \emptyset$ for all $T \in \cT(p)$. Then $|\cT| \ge M\delta^{-s-\eps}$.
\end{theorem}

The original theorem follows from taking $\eps = \eta$ and pigeonholing, since $M \in (\delta^{-s+\eps}, \delta^{-d})$.

\begin{proof}
    Before anything else, we state the dependencies of the parameters: $\eps_0 (s, t, \kappa, k, d)$, $ \eps(\eps_0, s, t, \kappa, k, d), T(\eps), \tau(s, t, \eps), \eta(\eps_0, \tau)$.

    First, choose $T=T(\eps)$ such that $\frac{2\log T}{T}\leq \eps$. By Lemma \ref{lem:uniform} we may find a subset $\cP' \subset \cP$ with $|\cP'| \ge \delta^\eps |\cP|$ that is $\{2^{-jT}\}_{j=1}^m$-uniform for $2^{-mT}=\delta$ with associated sequence $\{N_j\}_{j=1}^m$. Thus, $\cP'$ is a $(\delta, t, \delta^{-2\eps})$-set. Replacing $\cP$ with $\cP'$ and $\eps$ with $\frac{\eps}{2}$, we may assume from the start that $\cP$ is $\{2^{-jT}\}_{j=1}^m$-uniform.
    
    Let $f$ be the corresponding branching function. Since $\cP$ is a $(\delta, t, \delta^{-\eps})$-set, we have $f(x) \ge tx - \eps m$ for all $x \in [0, m]$.
    
    Let $\{ [c_j, d_j] \}_{j=1}^n$ be the intervals from Proposition \ref{cor:multiscale} applied with parameters $s, t, \eps$, corresponding to a sequence $0 < \delta = \Delta_n < \Delta_{n-1} < \cdots < \Delta_1 < \Delta_0 = 1$. We can partition $\{ 0, 1, \cdots, n-1 \} = \cS \cup \cB$, ``structured'' and ``bad'' scales such that:
    \begin{itemize}
        \item $\frac{\Delta_{j}}{\Delta_{j+1}} \ge \delta^{-\tau}$ for all $j \in \cS$, and $\prod_{j \in \cB} (\Delta_{j}/\Delta_{j+1}) \le \delta^{-\eps}$;

        \item For each $j \in \cS$ and $\bp \in \cD_{\Delta_j} (\cP)$, the set $\cP_j := S_{\bp} (\cP \cap \bp)$ is either
        \begin{enumerate} 
        \item[(i)] an $(t_j , \Delta_{j+1}/\Delta_j, (\Delta_j/ \Delta_{j+1})^{\eps}, (\Delta_j/ \Delta_{j+1})^{\eps})$-regular set, where $t_j \in (s, 2)$;
        \item[(ii)] a $(s , \Delta_{j+1}/\Delta_j, (\Delta_j/ \Delta_{j+1})^{\eps})$-set.
        \end{enumerate}

        \item $\prod_{j \in S} (\Delta_j/\Delta_{j+1})^{t_j} \ge  |\cP|  \cdot \prod_{j \in \cB} (\Delta_{j+1}/\Delta_{j})^d \ge |\cP| \delta^{O_{s,t,d} (\eps)}$.
    \end{itemize}
    Apply Proposition \ref{cor:multiscale} and $\frac{\Delta_{j}}{\Delta_{j+1}} \ge \delta^{-\tau}$ to get a family of tubes $\cT_{\textbf{p}} \subset \T^{\Delta_{j+1}/\Delta_j}$ with the property that $(S_{\textbf{p}} (\cP \cap \textbf{p}), \cT_{\textbf{p}})$ is a $(\Delta_{j+1}/\Delta_j, s, C_j^1, \kappa, C_j^2 M_{\textbf{p}})$-nice configuration for some $C_j^1, C_j^2 \lessapprox_\delta (\Delta_{j+1}/\Delta_j)^{-\tau^{-1} \eta}$ and
    \begin{equation*}
        \frac{|\cT_0|}{M} \geapp_\delta \prod_{j=0}^{N-1} \frac{|\cT_{\textbf{p}_j}|}{M_{\textbf{p}_j}}.
    \end{equation*}
    Let $\cS_1 = \{ j \in S : t_j \ge \frac{s+t}{2} \}$ and $\cS_2 = \cS \setminus \cS_1$. Then
    \begin{equation*}
        \prod_{j \in S_1} (\Delta_j/\Delta_{j+1})^{t_j} \ge  |\cP| \delta^{O_{s,t,d} (\eps)} \prod_{j \in S_2} (\Delta_j/\Delta_{j+1})^{-\frac{s+t}{2}} \ge \delta^{\frac{t-s}{2} + O_{s,t,d} (\eps)}.
    \end{equation*}
    
    For $j \in \cS_1$ we apply Theorem \ref{thm:improved_incidence} with parameters $s, \frac{s+t}{2}$, and for $j \in \cS_2$ we apply Corollary \ref{cor:easy_est}. If $\eps_0 (s, t, \kappa, k, d)$ is the $\eta$ from Theorem \ref{thm:improved_incidence}, then for $\tau^{-1} \eta < \eps_0$, we get
    \begin{equation*}
        \frac{|\cT_0|}{M} \geapp_\delta \prod_{j \in \cS_1} \left( \frac{\Delta_j}{\Delta_{j+1}} \right)^{-s-\eps_0} \cdot \prod_{j \in \cS_2} \left( \frac{\Delta_j}{\Delta_{j+1}} \right)^{-s+O(\eps)} \ge \delta^{-s(1-\eps) - (\frac{t-s}{2} + O_{s,t,d} (\eps)) \eps_0 + O(\eps)} \ge \delta^{-s-\eps}
    \end{equation*}
    as long as $\eps$ is taken small enough in terms of $\eps_0, s, t, d$.
\end{proof}

\section{Sets contained in an $(r_0, k)$-plate}\label{sec:refined from main}
We restate Theorem \ref{thm:main_refined}.
\begin{theorem}\label{thm:main_refined'}
    For any $0 \le k < d-1$, $0 \le s < k+1$, $\max(s, k) < t \le d$, $\kappa > 0$, $r_0 \le 1$, there exists $\eps(s, t, \kappa, k, d) > 0$ such that the following holds for all small enough $\delta/r_0 \in 2^{-\N} \cap (0, \delta_0)$, with $\delta_0$ depending only on $s, t, \kappa, k, d$. Let $H$ be a $(r_0, k+1)$-plate, $\cP \subset \cD_\delta \cap H$ be a $(\delta, t, (\delta/r_0)^{-\eps})$-set with $\cup \cP \subset [0, 1)^d$, and let $\cT \subset \cT^\delta \cap H$ be a family of $\delta$-tubes. Assume that for every $p \in \cP$, there exists a set $\T(p) \subset \T$ such that:
    \begin{itemize}
        \item $T \cap p \neq \emptyset$ for all $T \in \cT(p)$;

        \item $\cT(p)$ is a $(\delta, s, (\delta/r_0)^{-\eps} r_0^{k-s}, 0)$-set down from scale $r$;

        \item $\cT(p)$ is a $(\delta, \kappa, (\delta/r_0)^{-\eps} r_0^{-\kappa}, k)$-set.
    \end{itemize}
    Then $|\cT| \ge (\frac{\delta}{r_0})^{-\eps} \delta^{-2s} r_0^{2(s-k)}$.
\end{theorem}
\subsection{Multiscale analysis}
We will use Theorem \ref{thm:main} to prove Theorem \ref{thm:main_refined}. Let $S_H$ be the dilation sending $H$ to $[0,1]^d$. Then $\cP, \T(p)$, and $\T$ become deformed under $S_H$, but they satisfy the following statistics assumptions for $r \in [\frac{\delta}{r_0}, 1]$:
\begin{align}
    |\cP \cap S_H^{-1} (Q)| \le \left( \frac{\delta}{r_0} \right)^{-\eps} \cdot |\cP| \cdot r^t, & \qquad Q \in \cD_r (\R^d), \label{eqn:non_conc1} \\
    |\T(p) \cap S_H^{-1} (\bT)| \le \left( \frac{\delta}{r_0} \right)^{-\eps} \cdot |\T(p)| \cdot r^s, & \qquad \bT \quad r \text{-tube}, \label{eqn:non_conc2} \\
    |\T(p) \cap S_H^{-1} (W)| \le \left( \frac{\delta}{r_0} \right)^{-\eps} \cdot |\T(p)| \cdot r^\kappa, & \qquad W \quad (r, k+1) \text{-plate}. \label{eqn:non_conc3}
\end{align}
To prove \eqref{eqn:non_conc1}, observe that $S_H^{-1} (Q)$ is contained in an $r$-ball, and then we use that $\cP$ is a $(\delta, t, (\delta/r_0)^{-\eps}, 0)$-set.

To prove \eqref{eqn:non_conc2}, observe that $S_H^{-1} (\bT)$ is contained in a box with $k$ sides of length $r$ and $d-k$ sides of length $rr_0$. This box can be covered by $\sim r_0^{-k}$ many $rr_0$-balls. Finally, use that $\cT(p)$ is a $(\delta, s, (\delta/r_0)^{-\eps} r_0^{k-s}, 0)$-set.

To prove \eqref{eqn:non_conc3}, observe that $S_H^{-1} (W)$ is contained in a $(rr_0, k)$-plate.

Using these observations, we obtain the following refinement of Proposition \ref{prop:nice_tubes}. We use $(\delta, s, C_1 r_0^{k-s}, \kappa, C_2, M)$-nice configuration down from scale $r_0$ to indicate that $\T(p)$ is a $(\delta, s, C_1 r_0^{k-s}, 0)$-set down from scale $r_0$.

\begin{prop}\label{prop:nice_tubes_2}
    Fix dyadic numbers $0 < \delta' = \frac{\delta}{r_0} < \Delta \le 1$. Let $(\cP_0, \T_0)$ be a $(\delta, s, C_1 r_0^{k-s}, \kappa, C_2, M)$-nice configuration down from scale $r_0$, and assume $\cP_0 \subset H$ for some $(r_0, k)$-plate $H$. Then there exist refinements $\cP \subset \cP_0$, $\T(p) \subset \T_0 (p), p \in \cP$, and $\T_\Delta (Q) \subset \T^\Delta$ such that denoting $\T_\Delta = \cup_{Q \in \cD_{\Delta} (S_H (\cP))} \T_\Delta (Q)$ and $\T = \cup_{p \in \cP} \T(p)$ the following hold:
    \begin{enumerate}[(i)]
        \item $|\cD_\Delta (S_H(\cP))| \approx_{\Delta} |\cD_\Delta (S_H(\cP_0))|$ and $|S_H(\cP) \cap Q| \approx_{\Delta} |S_H(\cP_0) \cap Q|$ for all $Q \in \cD_\Delta (\cP)$.

        \item We have $|\T \cap \bT| \leapp \frac{|\T_0|}{|\T_\Delta|}$ for all $\bT \in \T_\Delta$.

        \item $(\cD_\Delta (S_H(\cP)), \T_\Delta)$ is $(\Delta, s, C^1_\Delta, \kappa, C^2_\Delta, M_\Delta)$-nice for some $C^1_\Delta \approx_{\Delta} C_1$, $C^2_\Delta \approx_{\Delta} C_2$, and $M_\Delta \ge 1$.

        \item For all $\bT \in \T_\Delta (Q)$, we have
        \begin{equation*}
            |\{ (p, T) \in \cP \times \T : T \in \T(p) \text{ and } T \subset S_H^{-1} (\bT) \} | \geapp_\Delta \frac{M \cdot |S_H (\cP) \cap Q|}{|\T_\Delta (Q)|}.
        \end{equation*}


        \item For each $Q \in \cD_\Delta (S_H(\cP_2))$, there exist $C^1_Q \approx_{\Delta} C_1$, $C^2_Q \approx_{\Delta} C_2$, $M_Q \ge 1$, a subset $\cP_Q \subset \cP \cap Q$ with $|\cP_Q| \geapp_\Delta |\cP \cap Q|$ and a family of tubes $\T_Q \subset \T^{\delta/\Delta}$ such that $(S_H^{-1} \circ S_Q (S_H(\cP_2 ) \cap Q), \T_Q)$ is $(\delta/\Delta, s, C^1_Q r_0^{k-s}, \kappa, C^2_Q, M_Q)$-nice down from scale $r_0$.
    \end{enumerate}

    Furthermore, the families $\T_Q$ can be chosen so that
    \begin{equation}\label{eqn:item6'}
        \frac{|\T_0|}{M} \geapp_{\Delta} \frac{|\T^\Delta (\T)|}{M_\Delta} \cdot \left( \max_{Q \in \cD_\Delta (\cP_2)} \frac{|\T_Q|}{M_Q} \right).
    \end{equation}
\end{prop}


\begin{proof}
    The proof will involve many dyadic pigeonholing steps.

    \textbf{Step 1: construct $\T_\Delta (Q)$.} For a given $Q \in \cD_\Delta (S_H (\cP_0)) := \cQ_0$, we claim that we can find a subset $\cP_Q \subset \cP_0 \cap S_H^{-1} (Q)$ with $|\cP_Q| \app_\Delta |\cP_0 \cap S_H^{-1} (Q)|$ and a family of dyadic $\Delta$-tubes $\oT_\Delta (Q)$ intersecting $Q$ such that the following holds:
    \begin{enumerate}[(T1)]
        \item\label{T1} $\oT_\Delta (Q)$ is a $(\Delta, s, C_\Delta^1, 0)$-set and $(\Delta, \kappa, C_\Delta^2, k)$-set for some $C_\Delta^1, C_\Delta^2 \app_\Delta C_1$.
        \item\label{T2} there exists a constant $H_Q \app_\Delta M \cdot |\cP_Q| / |\oT_\Delta (Q)|$ such that
        \begin{equation*}
            |\{ (p, T) \in \cP_Q \times \T_0 : T \in \T_0 (p) \text{ and } T \subset S_H^{-1} (\bT) \} | \gesim H_Q, \qquad \bT \in \overline{\T}_\Delta (Q).
        \end{equation*}
    \end{enumerate}
    This claim generalizes \cite[Proposition 4.1]{orponen2021hausdorff} and relies on the same dyadic pigeonholing steps; for brevity, we only state these steps and refer the reader to \cite{orponen2021hausdorff} for the detailed proof. (We essentially follow the same proof for \ref{T2}, and we introduce a nice shortcut to derive \ref{T1} from \ref{T2}.) Let $\cT_\Delta (Q) \subset \cT^\Delta$ be a minimal finitely overlapping cover of $S_H (\cT_Q) := \cup_{p \in \cP_0 \cap Q} S_H (\T_0(p))$ by $\Delta$-tubes. For $p \in \cP_0 \cap Q$, define
    \begin{equation*}
        \T_{\Delta, j} (p) = \{ \bT \in \T_\Delta (Q) : 2^{j-1} < |\{ T \in \T(p) : T \subset S_H^{-1} (\bT) \}| \le 2^j \}.
    \end{equation*}
    Since $|\T_\Delta (Q)| \lesim 100\Delta^{-2(d-1)}$ and $M \lesim \sum_j 2^j \cdot |\T_{\Delta,j} (p)|$, we in fact have
    \begin{equation*}
        M \lesim \sum_{M\Delta^{2(d-1)}/200 \le 2^j \le M} 2^j \cdot |\T_{\Delta,j} (p)|
    \end{equation*}
    Thus, by dyadic pigeonholing, there exists $j = j(p)$ such that $2^j \cdot |\cT_{\Delta,j} (p)| \app_\Delta M$. Another dyadic pigeonholing allows us to find $\cP_Q \subset \cP_0 \cap Q$ such that $j(p)$ is constant for $p \in \cP_Q$. This is the desired refinement $\cP_Q$ of $\cP_0 \cap Q$. Finally, let
    \begin{equation*}
        \T_{\Delta,i} (Q) := \{ \bT \in \T_\Delta (Q) : 2^{i-1} < |\{ p \in \cP_Q : \bT \in \T_\Delta (p) \}| \le 2^i \}.
    \end{equation*}
    Then by a similar dyadic pigeonholing (for calculations, see \cite[Proposition 4.1]{orponen2021hausdorff}), there is $i$ such that
    \begin{equation}\label{eqn:TDeltaj}
        \frac{1}{200} |\cP_Q| \Delta^{d-1} \le 2^i \le |\cP_Q| \text{ and } 2^{i+j} \cdot |\T_{\Delta,i} (Q)| \app_\Delta M \cdot |\cP_Q|.
    \end{equation}
    Finally, we define $\oT_\Delta (Q) := \T_{\Delta, i} (Q)$, which is the desired refinement of $\T_\Delta (Q)$.

    We check \ref{T2} holds with $H_Q = 2^{i+j}$, which satisfies $H_Q \app_\Delta M \cdot |\cP \cap Q|/|\oT_\Delta|$ by \eqref{eqn:TDeltaj} and $|\cP_Q| \app_\Delta |\cP \cap Q|$. With this choice of $H_Q$, fix $\bT \in \oT_\Delta$ and note that
    \begin{multline*}
        |\{ (p, T) \in \cP_Q \times \T_0 : T \in \T_0 (p), T \subset S_H^{-1} (\bT) \}| = \sum_{p \in \cP_Q} |\{ T \in \T(p) : T \subset S_H^{-1} (\bT) \}| \\
        \ge 2^j |\{ p \in \cP_Q : \bT \in \T_\Delta (p) \}| \ge 2^{i+j} = H.
    \end{multline*}
    To check \ref{T1}, we first pick a $r$-tube $\bT_r$ with $r \ge \Delta$. Then by \ref{T2} and \eqref{eqn:non_conc2},
    \begin{multline*}
        |\{ \bT \in \oT_\Delta : \bT \subset \bT_r \}| \lesim \frac{1}{H} |\{ (p, T) \in \cP_Q \times \T_0 : T \in \T_0 (p), T \subset S_H^{-1} (\bT_r) \}| \\
        \lesim \frac{1}{H} |\cP_Q| \cdot C_1 M r^s \leapp C_1 |\oT_\Delta| r^s.
    \end{multline*}
    Thus, $\oT_\Delta (Q)$ is a $(\Delta, s, C_\Delta^1, 0)$-set with $C_\Delta^1 \app_\Delta C_1$. Doing the same calculation with an $(r, k+1)$-plank instead of an $r$-tube, we get that $\oT_\Delta (Q)$ is a $(\Delta, \kappa, C_\Delta^2, k)$-set with $C_\Delta^2 \app_\Delta C_1$. This proves \ref{T1} and thus the claim.

    \textbf{Step 2: uniformity of $|\T_0 \cap \bT|$.} By the pigeonhole principle, we can find $\oM_\Delta \ge 1$ and a subset $\cQ \subset \cD_\Delta (\cP)$ with $|\cQ| \app_\Delta |\cQ_0|$ such that $|\oT_\Delta (Q)| \sim \oM_\Delta$ for all $Q \in \cQ$. Write
    \begin{equation*}
        \oT_\Delta = \bigcup_{Q \in \cQ} \oT_\Delta (Q).
    \end{equation*}
    Next, by another dyadic pigeonholing, we can find a subset $\oT_\Delta' \subset \oT_\Delta$ such that $I(\cQ, \oT_\Delta') \geapp I(\cQ, \oT_\Delta)$ and $|\T_0 \cap \bT| \sim N_\Delta$ for all $\bT \in \oT_\Delta'$. Also, $|\oT_\Delta (Q)| \lesim \oM_\Delta$ for all $Q \in \cQ$. Thus, we can find $\cQ' \subset \cQ$ with $|\cQ'| \app_\Delta |\cQ|$, and for each $Q \in \cQ'$ a subset $\T_\Delta (Q)$ of cardinality $\app \oM_\Delta$, such that $\T_\Delta (Q) \subset \oT_\Delta'$. In other words,
    \begin{equation*}
        |\T_0 \cap \bT| \sim N_\Delta \text{ for } \bT \in \T_\Delta (Q).
    \end{equation*}
    Thus, we obtain item \ref{item21}.
    \begin{equation}\label{eqn: step 2}
        |\T_0| \ge |\T_\Delta| \cdot \min_{\bT \in \T_\Delta} |\T_0 \cap \bT| \sim |\T_\Delta| \cdot N_\Delta.
    \end{equation}
    Reduce the families $\T_\Delta (Q)$ such that their cardinality is $M_\Delta := \min(|\T_\Delta (Q)| : Q \in \cQ \}) \app_\delta \oM_\Delta$. By \ref{T1}, $\T_\Delta (Q)$ remains a $(\Delta, s, C_\Delta^1, 0)$ and $(\Delta, \kappa, C_\Delta^2, k)$-set with $C_\Delta^1, C_\Delta^2 \app_\delta C_1$.

    Finally, define
    \begin{equation*}
        \cP = \bigcup_{Q \in \cQ} \cP_Q,
    \end{equation*}
    where $\cQ$ is the latest refinement of $\cQ_0$. Since $|\cP_Q| \app_\Delta |\cP_0 \cap S_H^{-1} (Q)|$, we get that item \ref{item1} holds.

    For $p \in \cP_Q = \cP \cap Q$, $Q \in \cQ$, define
    \begin{equation*}
        \T(p) = \bigcup_{\bT \in \T_\Delta (Q)} (\T_0 (p) \cap \bT), \cT = \bigcup_{p \in \cP} \T(p), \cT_\Delta = \bigcup_{Q \in \cQ} \T_\Delta (Q).
    \end{equation*}
    Thus, $(\cD_\Delta (\cP), \T_\Delta) = (\cQ, \T_\Delta)$ is a $(\Delta, s, C_\Delta^1, \kappa, C_\Delta^2, M_\Delta)$-nice configuration, establishing item \ref{item3}. To summarize, in this step, we refined $\cQ$ and $\cT_\Delta (Q)$ for $Q \in \cQ$, so \ref{T2}/\ref{item4} still holds (with same $H_Q$ and a weaker implied constant).

    \textbf{Step 3: uniformity of $\T(p)$ and construct $\T_Q$.} This step will be devoted to verifying \ref{item5} and \eqref{eqn:item6'}. We will not change $\cP, \T$, or $\T_\Delta$.
    
    Fix $Q \in \cQ$, and let $\cP_Q = \cP \cap S_H^{-1} (Q)$. Define
    \begin{equation*}
        \T(Q) = \bigcup_{p \in \cP_Q} \T(p).
    \end{equation*}
    By dyadic pigeonholing and \ref{T2}, we can find a $\app_\Delta$-comparable subset of $\cP_Q$ (which we keep denoting $\cP_Q$) such that
    \begin{equation*}
        |\T(p)| \app_\Delta M, \qquad p \in \cP_Q.
    \end{equation*}
    Next,
    \begin{equation}\label{eqn: step 3}
        |\T(Q)| \le \sum_{\bT \in \T_\Delta (Q)} |\T \cap \bT| \leapp_\Delta M_\Delta \cdot N_\Delta.
    \end{equation}

    For a given $p \in Q$, we consider the tube packet $\U(p) := \T(p) \cap S_H^{-1} (Q)$ (discarding duplicate tubelets). Each tubelet $u \in \U(p)$ lies in at most $\Delta^{-2(d-1)}$ many tubes of $\T(p)$, so by dyadic pigeonholing, we can refine $\T(p)$ by a $\log \Delta^{-1}$ factor to ensure that each tubelet $u \in \U(p)$ lies in $\sim m(p)$ many tubes of $\T(p)$, and there are $M(p) \app_\Delta \frac{M}{m(p)}$ many distinct tubelets through $p$. By refining $\cP_Q$ by a $(\log \Delta^{-1})$-factor, we may assume $m(p) \app m_Q$ for each $p \in \cP_Q$. Now, define
    \begin{equation*}
        \cP^Q := S_H^{-1} \circ S_Q \circ S_H (\cP_Q) \text{ and } \T_Q := \bigcup_{p \in \cP_Q} S_H^{-1} \circ S_Q \circ S_H (\U(p)).
    \end{equation*}
    Since tubelets are essentially distinct and each tubelet in any $\U(p)$ corresponds to $\app m_Q$ many tubes in $\T(Q)$, we obtain:
    \begin{equation}\label{eqn:step4}
        |\T(Q)| \geapp_\Delta \left|\bigcup_{p \in \cP_Q} \U(p) \right| \cdot m_Q \gesim |\T_Q| \cdot \frac{M}{M_Q}.
    \end{equation}
    Then \eqref{eqn:item6'} will follow by combining \eqref{eqn: step 2}, \eqref{eqn: step 3}, and \eqref{eqn:step4}.

    We finally check $(\cP^Q, \T_Q)$ is a $(\delta/\Delta, s, C_Q^1 r_0^{k-s}, \kappa, C_Q^2, M_Q)$-nice configuration down from scale $r_0$. First, for any $\odelta < r < r_0$, we have for any $(r, 0)$-plank $H$ in $S^{d-1}$,
    \begin{equation*}
        |\sigma(\T_Q) \cap H| \sim_\Delta \frac{1}{m_Q} |\sigma(\T(p)) \cap H| \leapp_\Delta \frac{1}{m_Q} \cdot C \cdot M \cdot r^s = C \cdot M_Q \cdot r^s.
    \end{equation*}
    Thus, $\sigma(\T_Q)$ is a $(\odelta, s, C_Q^1, 0)$-set down from scale $r_0$ with $C_Q^1 \app_\Delta C_1$. Similarly, $\sigma(\T_Q)$ is a $(\odelta, \kappa, C_Q^2, k)$-set with $C_Q^2 \app_\Delta C_2$. This shows item \ref{item5} and thus the proof of the Proposition.


\end{proof}

\subsection{Good multiscale decomposition}
The idea is to apply Proposition \ref{prop:nice_tubes_2}, then apply Theorem \ref{thm:main} to bound $|\T_\Delta|$ and Corollary \ref{cor:easy_est_2} to bound $|\T_Q|$. Unfortunately, while we use pigeonholing to ensure that $\cD_\Delta (S_H(\cP))$ is a $(\Delta, t)$-set, we don't know that $S_H^{-1} \circ S_Q (S_H(\cP) \cap Q)$ is a $(\frac{\delta}{\Delta}, t)$-set. In fact, we won't show this statement, but rather a slightly weaker statement that is good enough. For this, a good choice of $\Delta$ based on the branching structure of $\cP$ is needed.

First, we explain the pigeonholing preliminaries.

\begin{lemma}\label{lem:uniform1}
    Given $P \subset H_r$, a $(r_0, k)$-plane, there is a subset $P' \subset P$ with $|P'|_\delta \ge (\log (\frac{r_0}{\delta}))^{-1} |P|_\delta$ such that $|Q \cap S_H(P)|$ is constant for all $Q \in \cD_{\delta/r_0} (S_H(P))$.
\end{lemma}


\begin{proof}
    Let $f(N) = \sum \{ |P \cap S_H^{-1} (Q)|_\delta : Q \in \cD_{\delta/r_0} ([0, 1]^d), |P \cap S_H^{-1} (Q)|_\delta \in [N, 2N] \}$. Then $\sum_{N \text{ dyadic}} f(N) = |P|_\delta$. For each $N$, either $f(N) = 0$ or $N \le f(N) \le (r_0/\delta)^d \cdot N$. Hence, if $N_0$ is the largest $N$ for which $f(N) > 0$, we get $f(N_0) \ge N_0 > \sum_{M < N_0 (\delta/r_0)^d/100 \text{ dyadic}} f(M)$. Thus, we have
    \begin{equation*}
        \sum_{N_0 (\delta/r_0)^d/100 < M < N_0 \text{ dyadic}} f(M) > \frac{1}{2} |P|_\delta.
    \end{equation*}
    Thus, by dyadic pigeonholing, there exists $M \in (N_0 (\delta/r_0)^d/100, N_0)$ such that $f(M) \ge \frac{1}{20d} (\log (\frac{r_0}{\delta}))^{-1} |P|_\delta$.
\end{proof}

The next step is to make $S_H^{-1} \circ S_Q (S_H(\cP) \cap Q)$ satisfy a $t'$-dimensional spacing condition with $t'$ just slightly less than $t$, for all $Q \in \cD_\Delta (S_H (\cP))$ at a certain scale $\Delta$. To do so, we need the following lemma.

\begin{lemma}\label{lem:extract_t'_set}
    Fix $C, \eps > 0$, and let $\frac{\delta}{r_0} = \Delta^m$ and $P \subset H_r$ be a $(\delta, t, C, 0)$-set in $H_r$, a $(r_0, k)$-plane. Let $L = \log(\frac{r_0}{\delta}) \cdot (\log(1/\Delta))^m$. If $t' < \frac{t - d\eps}{1-\eps}$, then there exists $m\eps \le k \le m$ and a subset $P' \subset P$ with $|P'| \ge L^{-1} |P|$ such that for any $k \le j \le m$, $Q \in \cD_{\Delta^k} (P')$, and $R \in \cD_{\Delta^j} (P') \cap Q$, we have
    \begin{equation}\label{eqn:extract1}
        |P' \cap S_H^{-1} (R)| \le |P' \cap S_H^{-1} (Q)| \cdot \Delta^{(j-k)t'},
    \end{equation}
    and for $\delta \le r \le \frac{\delta}{r_0}$ and a ball $B_r$, we have
    \begin{equation}\label{eqn:extract2}
        |P' \cap B_r| \le C \cdot L \cdot |P' \cap S_H^{-1} (Q)| \left( \frac{r}{\Delta^k} \right)^{t'}.
    \end{equation}
\end{lemma}

\begin{proof}
    Throughout this proof we will not distinguish between $m\eps$ and $\lceil m\eps \rceil$.

    First, we will make $S_H(P)$ uniform at scales $1, \Delta, \Delta^2, \cdots, \Delta^m = \frac{\delta}{r_0}$. By Lemmas \ref{lem:uniform1} and \ref{lem:uniform}, we can find $|P'| \ge L^{-1} |P|$ such that there is a sequence $(N_j)_{j=1}^n$ with $|S_H(P') \cap Q|_{\Delta^k} = N_k$ for all $1 \le k \le n$ and $Q \in \cD_{\Delta^k} (P')$.

    Let $m\eps \le k \le m$ be the largest index such that $N_k \ge |P'| \Delta^{m\eps \cdot d+(k-m\eps)t'}$ for one (equivalently all) $Q \in \cD_{\Delta^k} (P')$. Certainly $k = m\eps$ is a valid index since $|\cD_{\Delta^{m\eps}}| = \Delta^{-dm\eps}$.

    Now, we will check the given conditions. By maximality of $k$, we have for $k \le j \le m$,
    \begin{equation*}
        N_j \le |P'| \Delta^{dm\eps+(j-m\eps)t'} \le N_k \Delta^{(j-k)t'}.
    \end{equation*}
    Noticing that $|P' \cap S_H^{-1} (Q)| = |S_H(P') \cap Q|$ and likewise for $R \in \cD_{\Delta^j} (P') \cap Q$, this proves \eqref{eqn:extract1}.

    To check \eqref{eqn:extract2}, we recall that $L N_k \ge L \cdot |P'| \Delta^{dm\eps+(k-m\eps)t'} \ge |P| \Delta^{dm\eps+(k-m\eps)t'}$. Using $r \le \frac{\delta}{r_0} = \Delta^m$, $t' \le \frac{t - d\eps}{1-\eps}$, and that $P$ is a $(\delta, t, C)$-set, we have
    \begin{equation*}
        |P' \cap B_r| \le |P \cap B_r| \le C |P| r^t \le C \cdot N_k L \left( \frac{r}{\Delta^k} \right)^{t'}.
    \end{equation*}
\end{proof}

Finally, we will need the following variant of Corollary \ref{cor:easy_est}.

\begin{cor}\label{cor:easy_est_2}
    Let $0 \le \max(s, k) < t \le d-1$, $\delta \le r \le 1$, and let $C_P \ge 1, C_T \ge 0$. Let $\cP \subset \cD_\delta$ be a set contained in an $(r_0, k+1)$-plate $H$ satisfying the following conditions:
    \begin{itemize}
        \item For all $\frac{\delta}{r_0} \le r \le 1$ and balls $B_r$, we have
        \begin{equation}\label{eqn:p_condition1}
            |\cP \cap S_H^{-1} (B_r)| \le C_P \cdot |\cP| \cdot r^t.
        \end{equation}

        \item For all $\delta \le r \le \frac{\delta}{r_0}$ and balls $B_r$, we have
        \begin{equation}\label{eqn:p_condition2}
            |\cP \cap B_r| \le C_P \cdot |\cP| \cdot r^t.
        \end{equation}
    \end{itemize}
    Assume that for every $p \in \cP$ there exists a family $\T(p) \subset \T^\delta$ of dyadic $\delta$-tubes satisfying the following conditions:
    \begin{itemize}
        \item $T \cap p \neq \emptyset$ for all $T \in \T(p)$;

        \item $|\T(p) \cap \bT| \le C_T \cdot |\T(p)| \cdot r_0^{k-s} x^s$ for all $x$-tubes $\bT$ with $\delta \le x \le r_0$.
    \end{itemize}
    Further assume that $|\T(p)| = M$ for some $M \ge 1$. If $\T = \cup_{p \in \cP} \T(p)$, then
    \begin{equation*}
        |\T| \gesim (C_P C_T)^{-1} \cdot Mr_0^{s-k} \delta^{-s}.
    \end{equation*}
\end{cor}

\begin{proof}
Let
\begin{equation*}
    j_P (\cP, \T) = \{ (q, t) \in \cP \times \T(p) : t \in \T(q) \}
\end{equation*}

We have the following:

\begin{lemma}\label{lem:jp_est}
    For all $p \in \cP$, we have $j_p (\cP, \T) \lesim_{s,t,k} C_P C_T |\cP| \cdot Mr_0^{k-s} \delta^s$.
\end{lemma}

\begin{proof}
    We count $j_p (\cP, \T)$ by first choosing a dyadic $\delta < r < 1$, then counting the number of $q \in \cP$ with $|p - q| \sim r$, then finally counting the number of $t \in \T$ that pass through $p, q$.

    If $r > \frac{\delta}{r_0}$, we claim that if $|x - y| \in [r, 2r]$ and some tube through $x, y$ lies in $H$, then $|S_H(x) - S_H(y)| \le 100r$, so $y \in S_H^{-1} (B_{100r} (S_H(x)))$.
    
    To prove this, we may assume $r_0 \le \frac{1}{50}$, as otherwise we can use the simple fact $|S_H(x) - S_H(y)| \le r_0^{-1} |x-y| \le 100r$. Now choose a coordinate system such that the first $k+1$ axes correspond to the long sides of $H$, and the remaining axes correspond to the short sides of $H$. Let $x - y = (\va, \vb) \in \R^{k+1} \times \R^{d-k-1}$. Then $|\va| \le |x-y| \le r$. Furthermore, we have $|\vb| \le 50r_0 |\va|$, otherwise any tube through $x, y$ would be roughly orthogonal to $H$ and intersect $H$ in a subtube with length $2r_0 \le 1$, contradiction. Thus, we have $|S_H(x) - S_H(y)| \le |\va| + r_0^{-1} |\vb| \le 100r$.
    
    Using the claim and condition \eqref{eqn:p_condition1}, we see that there are $\lesim C_P |\cP| \cdot r^t$ many choices for $q$. For each $q$, the set of tubes $t \in \T(p)$ passing through $q$ lies in a $\frac{\delta}{r}$-tube, so by the tube non-concentration condition (and noting that $\frac{\delta}{r} < r_0$), we have $C_T \cdot Mr_0^{k-s} \left( \frac{\delta}{r} \right)^s$ choices for $t$.

    Thus, the contribution to $j_p (\cP, \T)$ for a given dyadic $r > \frac{\delta}{r_0}$ is $C_P C_T |\cP| \cdot Mr_0^{k-s} \delta^s \cdot r^{s-t}$, and summing over dyadic $r$ gives $C_P C_T |\cP| \cdot Mr_0^{k-s} \delta^s \cdot O_{s-t} (1)$.

    If $r < \frac{\delta}{r_0}$, then by condition \ref{eqn:p_condition2} we see that there are $\lesim C_P |\cP| \cdot r^t$ many choices for $q$. For each $q$, the set of tubes $t \in \T(p)$ passing through $q$ lies in a $\frac{\delta}{r}$-tube $\bT_{\delta/r}$. We note that $\frac{\delta}{r} > r_0$, so the tube non-concentration doesn't apply directly, but luckily we note that $\bT_{\delta/r} \cap H$ can be covered by $(\frac{\delta}{rr_0})^k$ many $r_0$-tubes. Thus, by using tube non-concentration at scale $r_0$, we have $C_T \cdot Mr_0^k \cdot (\frac{\delta}{rr_0})^k$ choices for $t$.

    Thus, the contribution to $j_p (\cP, \T)$ for a given dyadic $r < \frac{\delta}{r_0}$ is (after some manipulation)
    \begin{equation*}
        C_P C_T |P| Mr_0^k \left( \frac{\delta}{r_0} \right)^t \cdot \left( \frac{rr_0}{\delta} \right)^{t-k}.
    \end{equation*}
    Since $t > \max(k, s)$, the sum is $\lesim C_P C_T |P| Mr_0^k \left( \frac{\delta}{r_0} \right)^s$.

    Adding up both $r > \frac{\delta}{r_0}$ and $r < \frac{\delta}{r_0}$ contributions, we prove the Lemma.
\end{proof}

For $t \in \T$, let $\cP(t) = \{ p \in \cP : t \in \T(p) \}$. By Cauchy-Schwarz, we have
\begin{equation*}
    (M|\cP|)^2 = \left( \sum_{t \in \T} |\cP(t)| \right)^2 \le |\T| \sum_{t \in \T} |\cP(t)|^2 = |\T| \sum_{p \in \cP} j_p (\cP, \T).
\end{equation*}
By Lemma \ref{lem:jp_est}, we get
\begin{equation*}
    |\T| \ge \frac{M^2 |\cP|^2}{C_P C_T |\cP|^2 Mr_0^{k-s} \delta^s} = (C_P C_T)^{-1} M r_0^{s-k} \delta^{-s}.
\end{equation*}
\end{proof}

\begin{proof}[Proof of Theorem \ref{thm:main_refined'}]
    A small reduction: we would like to assume $|\T(p)| \sim M$ for all $p \in \cP$. To assume this, we first observe that $|\T(p)| \ge M_0 = (\delta/r_0)^{-\eps} r_0^{k-s} \delta^{-s}$ for all $p \in \cP$. On the other hand, if for at least half of the $p \in \cP$ (call them $\cP'$) we have $|\T(p)| \ge M_0 (\delta/r_0)^{-1}$, then we are immediately done by Corollary \ref{cor:easy_est_2} applied to $\cP'$ and $\T(p)$. Thus, by reducing $\cP$ if necessary, we may assume $|\T(p)| \in (M_0, M_0 (\delta/r_0)^{-1})$. Then by reducing $\cP$ further by a $\lesim \log (\delta/r_0)^{-1}$ factor, we may assume $|\T(p)| \in (M, 2M)$ for some $M \in (M_0, M_0 (\delta/r_0)^{-1})$. Finally, we may remove some tubes from each $\T(p)$ to make $|\T(p)| = M$. Then $(\cP_0, \T_0)$ is a $(\delta, s, C_1 r_0^{k-s}, \kappa, C_2, M)$-nice configuration.

    Pick $\beta(s, t, k) > 0$ such that $\frac{t-d\beta}{1-\beta} > \max(s, k)$, and let $t' = \frac{1}{2} (\frac{t-d\beta}{1-\beta} + \max(s, k))$. Pick $\Delta > 0$ such that $\log(1/\Delta) < \Delta^{-\eps}$. Find $\Delta' = \Delta^k \in (\delta/r_0, (\delta/r_0)^\beta)$ such that the conclusion of Lemma \ref{lem:extract_t'_set} holds. Now by Proposition \ref{prop:nice_tubes_2}, we have
    \begin{equation*}
        \frac{|\T|}{M} \ge \frac{|\T_Q|}{M_Q} \cdot \frac{|\T^{\Delta'} (\T)|}{M_{\Delta'}}.
    \end{equation*}
    If $\eps < \beta \eta^2$, where $\eta(s, t, \kappa, k, d)$ is the parameter in Theorem \ref{thm:main'}, we have $\frac{|\T^{\Delta'} (\T)|}{M_{\Delta'}} \ge (\Delta')^{-s-\sqrt{\eps}}$.
    
    Pick $Q$. Then $S_H^{-1} \circ S_Q (S_H(\cP) \cap Q)$ satisfies the conditions of Corollary \ref{cor:easy_est_2} with $C_P = \left( \frac{\delta}{r_0} \right)^{-\eps} \cdot L \cdot \Delta^{-d}$. Thus, we have $\frac{|\T_Q|}{M_Q} \ge \left( \frac{\delta}{r_0} \right)^{-\eps} \Delta^d L^{-1} \cdot r_0^{s-k} \left( \frac{\delta}{\Delta'} \right)^{-s}$. Using these two bounds and $M \ge (\delta/r_0)^\eps r_0^{s-k} \delta^{-s}$, we get
    \begin{equation*}
        |\T| \ge \left( \frac{\delta}{r_0} \right)^{2\eps-\sqrt{\eps}\beta} \Delta^{-d} L^{-1} r_0^{2(s-k)} \cdot \delta^{-2s}.
    \end{equation*}
    It remains to choose $\eps < \beta^2/100$ and also for $\frac{\delta}{r_0}$ small enough, we have $\Delta^{-d} < \left( \frac{\delta}{r_0} \right)^{-\eps}$ and $L \le \left( \frac{\delta}{r_0} \right)^{-\eps} \Delta^{-\eps m} \le \left( \frac{\delta}{r_0} \right)^{-2\eps}$. Thus, $|\T| \ge \left( \frac{\delta}{r_0} \right)^{-\eps} r_0^{2(s-k)} \cdot \delta^{-2s}$ and we are done.
\end{proof}

\section{Power decay around $k$-planes}\label{sec:power decay}
In this section, we will roughly deal with the following situation:
\begin{itemize}
    \item $\mu, \nu$ are $s$-Frostman measures with $k-1 < s \le k$;

    \item $\nu$ gives mass $\le \eps$ to any $(r_0, k)$-plate.
\end{itemize}
In other words, $\nu$ does not concentrate around $(r_0, k)$-plates. We would like to understand the $\nu$-mass of $(r, k)$-plates for $r$ much smaller than $r_0$. A result of Shmerkin \cite[Proposition B.1]{shmerkin2022non} says that there exist $r_1 (r_0, s, k), \kappa(s, k) > 0$, a subset $X \subset \spt \mu$ with $\mu(X) > 1 - O(\eps)$, and for each $x \in X$, a subset $Y_x \subset \spt \nu$ with $\mu(Y_x) > 1 - O(\eps)$ such that $\nu(H \cap Y_x) \le r^\eta$ for all $r \le r_1$ and $(r, k)$-plates $H$ through $x$. Thus, we do obtain a power decay for sufficiently small $r$. But what is the optimal starting point of the power decay? Can we hope for a power decay $\nu(H \cap Y_x) \lesim K (\frac{r}{r_0})^\eta$ for all $(r, k)$-plates through $x$? The answer is yes, and indeed we shall prove it by making small but meaningful tweaks to Shmerkin's argument. But before stating our result, we shall introduce some convenient notation. We define thin $k$-plates, a generalization of thin tubes, as follows.

\begin{defn}\label{def:thin tubes}
    Let $K, t \ge 0$, $1 \le k \le d-1$, and $c \in (0, 1]$. Let $\mu, \nu \in \PP(\R^d)$ supported on $X, Y$. Fix $G \subset X \times Y$. We say $(\mu, \nu)$ has $(t, K, c)$-thin $k$-plates on $G$ down from scale $r_0$ if
    \begin{equation}\label{eqn:thin tubes}
        \nu(H \cap G|_x) \le K \cdot r^t \quad \text{ for all } r \in (0, r_0) \text{ and all } (r, k)\text{-plates } H \text{ containing } x. 
    \end{equation}
\end{defn}

\begin{remark}
    In this paper, we will choose $G = (A \cup B)^c$ where $\mu \times \nu(B)$ is small. (The complement is taken with respect to $\R^d \times \R^d$.) In this case, the equation \eqref{eqn:thin tubes} becomes
    \begin{equation*}
        \nu(H \setminus (A|_x \cup B|_x)) \le K \cdot r^t \quad \text{ for all } r \in (0, r_0) \text{ and all } (r, k)\text{-plates } H \text{ containing } x. 
    \end{equation*}
\end{remark}

Now, we can state the main proposition, which generalizes and extends Proposition B.1 of \cite{shmerkin2022non}. It may be of independent interest.

\begin{prop}\label{prop:b1}
    Let $1 \le k \le d-1$ and $k-1 < s \le k$. There exist $\eta(\kappa, k, d) > 0$ and $K_0 (\kappa, k, d) > 0$ with the following property. Fix $r_0 \le 1$ and $K \ge K_0$. Suppose that $\mu, \nu$ are positive measures with $|\mu|, |\nu| \ge 1$ and for any $(r, k-1)$-plate $H$, we have
    \begin{gather*}
        \mu(H) \le C_\mu r^\kappa, \\
        \nu(H) \le C_\nu r^\kappa.
    \end{gather*}
    Let $A \subset X \times Y$ be the pairs of points that lie in some $K^{-1}$-concentrated $(r_0, k)$-plate. Then there exists $B$ with $\mu \times \nu(B) \le K_0 K^{-1}$ such that $(\mu, \nu)$ have $(\eta, K r_0^{-\eta})$-thin $k$-plates on $(A \cup B)^c$. (The complement is taken with respect to $\R^d \times \R^d$.)
\end{prop}

\begin{remark}
    (a) We can apply Proposition \ref{prop:b1} in case $\mu, \nu$ are $s$-dimensional with $s > k-1$.
    
    (b) In Proposition B.1 of \cite{shmerkin2022non}, the exponents for $\mu, \nu$ are allowed to differ. The proof of Proposition \ref{prop:b1} is easily modified to include this detail.
\end{remark}

To prove Proposition \ref{prop:b1}, we need the following two lemmas. Fix $r \le r_0$. The first says that there are few dense $(r, k)$-plates, and the second says that for most $x \in X$, the dense $(r, k)$-plates through $x$ lie in some $(r_0, k)$-plate.
\begin{lemma}\label{lem:few_large_plates}
    There is $N=N(\kappa, k, d)$ such that the following holds: let $\nu$ be a measure with mass $\le 1$ such that $\nu(W) \le C_\nu \rho^\kappa$ for all $(\rho, k-1)$-plates $W$, $1 > \rho > r$. Let $\cE_{r,k}$ be a set of $(r, k)$-plates such that every $(s, k)$-plate contains $\lesim \left( \frac{s}{r} \right)^{(k+1)(d-k)}$ many $r$-plates of $\cE_{r,k}$ (as in Section \ref{subsec:r-net}). Let $\cH = \{ H \in \cE_{r,k} : \nu(H) \ge a \}$. Then $|\cH| \lesim (\frac{C_\nu}{a})^N$.
\end{lemma}

Lemma \ref{lem:few_large_plates} follows from the condition on $\cE_{r,k}$ and the following generalization of \cite[Lemma B.3]{shmerkin2022non}. In the case $a = \delta^\eta$, the resulting bound is stronger but the assumption is also stronger.

\begin{lemma}
    Suppose $\nu(W) \le C_\nu \rho^\kappa$ for all $(\rho, k-1)$-plates $W$, $1 > \rho > r$. Then there exists a family of $\lesim a^{-1}$ many $(r(C_\nu/a^2)^{1/\kappa}, k)$-plates $\{ T_j \}$ such that every $(r, k)$-plate $H$ with $\nu(H) \ge a$ is contained in some plate $T_j$.
\end{lemma}

\begin{proof}
    Choose a maximal set of $(r, k)$-plates $\{ Y_j \}_{j=1}^m$ such that
    \begin{enumerate}
        \item $\nu(Y_i) \ge a$,

        \item $\nu(Y_i \cap Y_j) \le a^2/2$ for $1 \le i < j \le m$.
    \end{enumerate}
    We claim $m \le 2a^{-1}$. Indeed, if $S = \sum_{i=1}^m \nu(Y_i)$ and $f = \sum_{i=1}^m \one_{Y_i}$, then, then
    \begin{equation}\label{eqn:L2}
        S^2 = \left( \int f \, d\nu \right)^2 \le \int f^2 \, d\nu = S + \sum_{1 \le i < j \le m} \nu(Y_i \cap Y_j) \le S + m^2 a^2/2.
    \end{equation}
    Now, $S \ge ma > 2$, so $S^2 - S > \frac{S^2}{2}$. Combining with \eqref{eqn:L2} gives $S^2 < m^2 a^2$, a contradiction.

    Let $\{ T_j \}_{j=1}^m$ be the $(r(C_\nu/a^2)^{1/\kappa}, k)$-plates with same central $k$-plane as $Y_j$. We show the problem condition. Given an $(r, k)$-plate $H$ with $\nu(H) \ge a$, by maximality there exists $Y_j$ such that $\nu(H \cap Y_j) \ge a^2/2$. Thus, if $\angle(H, Y_j)$ is the largest principal angle between the central planes of $H$ and $Y_j$, then $H \cap Y_j$ is contained in a box of dimensions
    \begin{equation*}
        \underbrace{1 \times \cdots \times 1}_{(k-1) \text{ times}} \times r/\angle(H, Y_j) \times \underbrace{r \times \cdots \times r}_{d-k \text{ times}}.
    \end{equation*}
    Thus, $H \cap Y_j$ is contained in a $(r/\angle(H, Y_j), k-1)$-plate, so $\nu(H \cap Y_j) \le C_\nu (r/\angle(H, Y_j))^\kappa$. Thus, $\angle(H, Y_j) \lesim r(C_\nu/a^2)^{1/\kappa}$, so $H$ is contained in $T_j$.
\end{proof}

\begin{remark}
    We would like to present an alternative proof of Lemma \ref{lem:few_large_plates}, which was the original one found by the author. It gives slightly worse bounds but we believe it is slightly more motivated.
    
    If $a > 1$ then $\cH = \emptyset$, so assume $a \le 1$. Let $\xi = \left(\frac{a}{2C_\nu}\right)^{1/\kappa} \le 1$.
    By induction, for each $0 \le i \le k$, there exist $x_0, \cdots, x_i$ such that $|x_0 \wedge x_1 \wedge \cdots \wedge x_i| \ge \xi^i$ and that lie in at least $|\cH| (\frac{a}{2})^{i+1}$ many elements of $\cH$.

    The base case $i = 1$ is trivial. For the inductive step, suppose $x_0, \cdots, x_i$ are found. Let $\Omega$ be the $\tilde{r}$-neighborhood of the span of $x_1, \cdots, x_i$. Then since $\mu(\Omega) \le C_\nu \xi^\kappa < \frac{1}{2} a$ for every $H \in \cH$, we have $\mu(H \setminus \Omega) \ge \frac{1}{2} a$. Thus, there is $x_{i+1} \in \R^d \setminus \Omega$ such that $x_0, \cdots, x_{i+1}$ lie in at least $|\cH| (\frac{a}{2})^{(i+2)}$ many elements of $\cH$, and by construction, $|x_0 \wedge x_1 \wedge \cdots \wedge x_{i+1}| \ge \xi^{i+1}$. This completes the inductive step and thus the proof of the claim.

    Finally, the set of $(r,k)$-plates through $x_1, \cdots, x_k$ must lie in a $(r \xi^{-k}, k)$-plate, so at most $\xi^{-k(k+1)(d-k)}$ many $(r,k)$-plates of $\cE_{r,k}$ can lie in it. Thus, $|\cH| \le (\frac{a}{2})^{-(k+1)} \xi^{-k(k+1)(d-k)} \lesim (\frac{C_\nu}{a})^N$.
\end{remark}

The following lemma is in the same spirit as \cite[Proposition B.2]{shmerkin2022non}.
\begin{lemma}\label{lem:concentrate_in_r0}
    Let $\cH$ be a collection of $(r,k)$-plates, and suppose $\mu(W) \le C_\mu \rho^\kappa$ for all $(\rho, k-1)$-plates $W$, $1 > \rho > r$. Then for all $x \in X$ except a set of $\mu$-measure $\le C_\mu \left( \frac{r}{r_0} \right)^\kappa |\cH|^2$, there exists an $(r_0,k)$-plate that contains every $(r,k)$-plate in $\cH$ that passes through $x$.
\end{lemma}

\begin{proof}
    The exceptional set is contained in the set of $x \in X$ that lies in two plates of $\cH$ with ``angle'' $\ge \frac{1}{r_0}$. The intersection of two such plates is contained in a box with dimensions $\underbrace{r \times \cdots \times r}_{d-k \text{ times}} \times \frac{r}{r_0} \times \underbrace{1 \times \cdots \times 1}_{k-1 \text{ times}}$, which in turn is contained in a $(\frac{r}{r_0}, k-1)$-plate (since $r_0 \le 1$). Thus, by assumption on $\mu$, this box has mass $\lesim C_\mu \left( \frac{r}{r_0} \right)^{s-(k-1)}$. Finally, there are $|\cH|^2$ pairs of plates in $\cH$.
\end{proof}


\begin{proof}[Proof of Proposition \ref{prop:b1}]
    Fix $r \le r_0$, and let $\eta = \frac{\kappa}{4N}$, where $N$ is the constant in Lemma \ref{lem:few_large_plates}. We may assume $N \ge 2$. By Lemmas \ref{lem:few_large_plates} and \ref{lem:concentrate_in_r0}, we can find a set $E_r$ with $\mu(E_r) \le K^{-2} \left( \frac{r}{r_0} \right)^{\eta}$ and, for each $x \notin E_r$, a set $P_r (x) \subset Y$ that is either empty or a $(r_0^{1/2} r^{1/2}, k)$-plate through $x$ such that $\nu(W) \le K \left( \frac{r}{r_0} \right)^\eta$ for every $W$ intersecting $Y \setminus P_r (x)$.

    Now, let $E = \cup_{n \ge 0} E_{r_0 K^{-2^n}}$ and $P(x) = \cup_{n \ge 0} P_{r_0 K^{-2^n}} (x)$. We claim that $\mu(E) \le K^{-1}$ and if $x \notin E$, then $\nu(P(x) \setminus A|_x) \lesim K^{-1}$. Then if $r \ge r_0 K^{-1}$, then $\nu(W) \le 1 \le K \left( \frac{r}{r_0} \right)^\eta$ for $\eta < 1$; for any $r_0 K^{-2^n} \le r < r_0 K^{-2^{n-1}}$, we have for any $(r, k)$-plate $W$,
    \begin{equation*}
        \nu(W \setminus P_{r_0 K^{-2^n}} (x)) \le K \left( \frac{r_0 K^{-2^{n-1}}}{r_0} \right)^\eta \le K \left( \frac{r}{r_0} \right)^{\eta/2}
    \end{equation*}
    Then $(\mu, \nu)$ have $(\eta/2, K^2 r_0^{-\eta/2}, 1-K^{-1})$-thin $k$-plates relative to $A$.

    To prove the first claim, we observe that $\mu(E) \le K^{-2} \sum_{n=0}^\infty K^{-\eta 2^n} \le K^{-1}$ if $K_0$ is sufficiently large in terms of $\eta$.

    Next, by definition of $P_{r_0 K^{-2^{n-1}}}$, we have $\nu(P_{r_0 K^{-2^n}} (x) \setminus P_{r_0 K^{-2^{n-1}}} (x)) \le K \left(\frac{r_0 K^{-2^{n-1}}}{r_0} \right)^{\eta/2} \le K^{1-2^{n-2} \eta}$. We also have the bound $\nu(P_{r_0 2^{-n}} (x) \setminus A|_x) \le K^{-1}$ from the given condition (note that $P_{r_0 2^{-n}} (x)$ is a $(r_0 2^{-(n-1)}, k)$-plate). Thus,
    \begin{align*}
        \nu(P(x) \setminus A|_x) &\le \sum_{n=0}^{\log \eta^{-1}} \nu(P_{r_0 K^{-2^n}} (x) \setminus A|_x) + \sum_{n=\log \eta^{-1}}^\infty \nu(P_{r_0 K^{-2^n}} (x) \setminus P_{r_0 K^{-2^{n-1}}} (x)) \\
            &\le \log \eta^{-1} \cdot K^{-1} + \sum_{n=\log \eta^{-1}}^\infty K^{1-2^{n-2} \eta} \\
            &\lesim K^{-1},
    \end{align*}
    if $K_0$ is chosen large enough.
\end{proof}

\section{Radial projection estimates}\label{sec:threshold}
In this section, we will first prove a key special case, and then the general case of Theorem \ref{thm:intermediate}.

\subsection{Maximal plate concentration case}
\textit{This subsection is based on ideas from \cite{orponen2022kaufman}.}


\begin{theorem}\label{thm:threshold_refined}
Let $k \in \{ 1, 2, \cdots, d-1 \}$, $k-1 < \sigma < s \le k$, and fix $K \ge 1$. There exists $N \in \N$ and $K_0$ depending on $\sigma, s, k$ such that the following holds. Fix $r_0 \le 1$ and $K_1, K_2 \ge K_0$. Let $\mu, \nu$ be $\sim 1$-separated $s$-dimensional measures with constant $C_\mu, C_\nu$ supported on $E_1, E_2$, which lie in an $(r_0, k)$-plate $H_r$. Assume that $|\mu|, |\nu| \le 1$. Let $A$ be the pairs of $(x, y) \in E_1 \times E_2$ that lie in some $K_1^{-1}$-concentrated $(\frac{r_0}{K_2}, k)$-plate. Then there exists a set $B \subset E_1 \times E_2$ with $\mu \times \nu (B) \lesim K_1^{-1}$ such that for every $x \in E_1$ and $r$-tube $T$ through $x$, we have
\begin{equation*}
    \mu(T \setminus (A|_x \cup B|_x)) \lesim \frac{r^\sigma}{r_0^{\sigma-(k-1)}} (K_1 K_2)^N.
\end{equation*}
The implicit constant may depend on $s, k$.
\end{theorem}

Theorem \ref{thm:threshold_refined} is the special case of Theorem \ref{thm:intermediate} where $(\mu, \nu)$ are concentrated in a $(r_0 K, k)$-plate for some small $K \ll r_0$ (we call this the maximal plate concentration case). For this, we closely follow the bootstrapping approach of \cite{orponen2022kaufman}. There are three ingredients.
\begin{itemize}
    \item The next Proposition \ref{prop:k-thin} will be the base case for the bootstrapping argument ($\sigma = 0$).

    \item Proposition \ref{prop:b1} will ensure power decay for $\mu, \nu$ around $k$-planes.
    
    \item Theorem \ref{thm:main_refined} will be used in the bootstrapping step to upgrade $\sigma$ to $\sigma + \eta$.
\end{itemize} 

\begin{prop}\label{prop:k-thin}
    Let $1 \le k \le d-1$ and $k-1 < s \le k$, then there exists $N = N(s, k)$ such that the following holds. Fix $K \ge K_0$. Then for any $s$-dimensional measures $\mu, \nu$ with constant $\sim 1$ contained in the $r_0$-neighborhood of a $k$-plane and $d(\mu, \nu) \gesim 1$, there exists $B \subset X \times Y$ with $\mu \times \nu(B) \le K^{-1}$ such that $(\mu, \nu)$ has $(0, K^N r_0^{k-1})$-thin tubes on $B^c$ down from scale $r_0$.
\end{prop}

\begin{proof}
    Let $\tmu, \tnu$ be the projected measures on the $k$-plane. Then $\tmu, \tnu$ satisfy $s$-dimensional Frostman conditions for $r_0 \le r \le 1$. Let
    \begin{equation*}
        B = \{ (x, y) : x, y \in T \text{ for some } r_0\text{-tube } T \text{ with } \nu(T) \ge K^N r_0^{k-1}. \}
    \end{equation*}
    The rest is a standard argument following \cite[Proof of Lemma 3.6]{guth2019incidence}. Define the radial projection $P_y (x) = \frac{x-y}{|x-y|}$. Orponen's radial projection theorem \cite[Equation (3.5)]{orponen2018radial} can be written in the form (where $p = p(s, k) > 1$):
    \begin{equation}\label{eqn:orponen radial}
        \int \norm{P_x \tmu}_{L^p}^p \, d\tmu (x) \lesim 1.
    \end{equation}
    To effectively use \eqref{eqn:orponen radial}, we will show that $|P_x (B|_x)|$ is small for $x \in X$. Indeed, let $\T_x$ be a minimal set of finitely overlapping $2r_0$-tubes through $x$ such that any $r_0$-tube through $x$ with $\nu(T) \ge K^N r_0^{k-1}$ lies in a $2r_0$-tube in $\T_x$. Then each $2r_0$-tube in $\T_x$ has $\nu$-measure $\ge K^N r_0^{k-1}$. Since $d(x, \nu) \gesim 1$, we conclude that $|\T_x| \lesim K^{-N} r_0^{1-k}$. Therefore, since the Lebesgue measure $|P_x (T)| \lesim r_0^{k-1}$ for a $2r_0$-tube $T$ through $x$, we obtain $|P_x (B|_x)| \lesim K^{-N}$. 
    Finally, we can use Holder's inequality and \eqref{eqn:orponen radial} to upper bound $\mu \times \nu(B)$:
    \begin{align*}
        \mu \times \nu(B) &= \int \nu(B|_x) d\mu(x) \\
                &= 
                \int \left( \int_{P_x (B|_x)} P_x (\nu) \right) d\mu(x) \\
                &\le \sup_x |P_x (B|_x)|^{1-1/p} \int \norm{P_x \nu}_{L^p} d\mu(x) \\
                &\lesim K^{-N(1-1/p)}.
    \end{align*}
    Choose $N = 1 + (1 - 1/p)^{-1}$ to finish (the implicit constant is dominated by $K \ge K_0$ if $K_0$ is large enough).
\end{proof}

The bootstrapping step is as follows:

\begin{prop}\label{prop:bootstrap}
Let $k \in \{ 1, \cdots, d-1 \}$, $0 \le \sigma \le k$, $\max(\sigma, k-1) < s \le k$, $\kappa > 0$. There exist $\eta(\sigma, s, \kappa, k, d)$ and $K_0(\eta, k) > 0$ such that the following holds. Fix $r_0 \le 1$ and $K \ge K_0$. Let $\mu, \nu$ be $\sim K^{-1}$-separated $s$-dimensional measures with constant $K$ supported on $X, Y$, which lie in an $(r_0, k)$-plate $H$. Let $G \subset X \times Y$. Suppose that $(\mu, \nu)$ and $(\nu, \mu)$ have $(\sigma, K r_0^{-(\sigma-(k-1))})$-thin tubes and $(\kappa, K r_0^{-\kappa})$-thin $k$-plates on $G$ down from scale $r_0$. Then there exists a set $B \subset X \times Y$ with $\mu \times \nu(B) \le K^{-1}$ such that
$(\mu, \nu)$ and $(\nu, \mu)$ have $(\sigma+\eta,  K^{d+1} r_0^{-(\sigma+\eta-(k-1))})$-thin tubes on $G \setminus B$ down from scale $r_0$. Furthermore, $\eta(\sigma, s, \kappa, k, d)$ is bounded away from zero on any compact subset of $\{ (\sigma, s, \kappa, k) : \max(\sigma, k-1) < s \le k \le d-1 \}$.
\end{prop}


\begin{remark}
    The reader is advised to set $r_0 = 1$ in the following argument, in which case it is a straightforward modification of \cite[Lemma 2.8]{orponen2022kaufman}, with one small technical exception in the proof of the concentrated case, where we improve upon the dyadic pigeonholing step. Also if $r_0 = 1$, then the simpler Theorem \ref{thm:main} can be used instead of Theorem \ref{thm:main_refined} in the proof.
\end{remark}

\begin{proof}
    
    We are given that for all $r \in (0, r_0]$,
    \begin{gather}
        \nu(T \cap G|_x) \le K \cdot \frac{r^{\sigma}}{r_0^{\sigma-(k-1)}} \text { for all } r\text{-tubes } T \text{ containing } x \in X, \label{gather1}\\
        \nu(W \cap G|_x) \le K \cdot \frac{r^{\sigma}}{r_0^{\sigma-(k-1)}} \text { for all } (r, k)\text{-plates } W \text{ containing } x \in X, \label{gather2}\\
        \mu(T \cap G|^y) \le K \cdot \frac{r^{\sigma}}{r_0^{\sigma-(k-1)}} \text { for all } r\text{-tubes } T \text{ containing } y \in Y, \label{gather3}\\
        \mu(W \cap G|^y) \le K \cdot \frac{r^{\sigma}}{r_0^{\sigma-(k-1)}} \text { for all } (r, k)\text{-plates } W \text{ containing } y \in Y. \label{gather4}
    \end{gather}
    For $x \in X$ and $r \le r_0$, let $\T''_{x,r}$ denote the $r$-tubes through $x$ such that
    \begin{equation}\label{eqn:T''}
        \nu(T \cap G|_x) \ge K^{d+1} \cdot \frac{r^{\sigma+\eta}}{r_0^{\sigma+\eta-(k-1)}}.
    \end{equation}
    Now, let $\T'_{x,r}$ denote a covering of $\T''_{x,r}$ by essentially distinct $2r$-tubes. Then for $x \in X$, since $d(x, Y) \ge K^{-1}$, we have that the tubes in $\T'_{x,r}$ have $\lesim K^{d-1}$-overlap on $\nu$, so $|\T'_{x,r}| \lesim \frac{r^{-(\sigma+\eta)}}{r_0^{-(\sigma+\eta-(k-1))}}$. For a dyadic $r \in (0, r_0]$, let $H_r = \{ (x, y) \in G : y \in \cup \T'_{x,r} \}$, where $\cup \T'_{x,r}$ denotes the union of the tubes in $\T'_{x,r}$.

    \textbf{Claim.} There are $\eta(\sigma, s, \kappa, k, d) > 0$ and $K_0 (\eta) > 0$ such that the following holds for $K \ge K_0$. If $\frac{r}{r_0} < K^{-1/\eta}$, then $\mu \times \nu(H_r) \le 2\left(\frac{r}{r_0}\right)^\eta$. Furthermore, $\eta(\sigma, s, \kappa, k, d)$ is bounded away from zero on any compact subset of $\{ (\sigma, s, \kappa, k, d) : \max(\sigma, k-1) < s \le k \le d-1 \}$.

    We will be done if we show the claim. Indeed, let $B_1 = \cup_{r \le r_0 \text{ dyadic }} H_r$; then for any dyadic $r \le r_0$ and any $r$-tube $T$ through some $x \in X$, we either have $T \in \T'_{x,r}$, which means $T \cap G|_x \setminus B_1|_x = \emptyset$, or the negation of \eqref{eqn:T''} holds. In either case, we get
    \begin{equation}\label{eqn:T'' neg}
        \nu(T \cap G|_x \setminus B_1|_x) \le K^{d+1} \cdot \frac{r^{\sigma+\eta}}{r_0^{\sigma+\eta-(k-1)}}.
    \end{equation}
    We have \eqref{eqn:T'' neg} for dyadic $r \le r_0$, but it also holds for all $r \le r_0$ at the cost of introducing a multiplicative factor of $2^{\sigma+\eta} \le 2^{k+1}$ on the RHS of \eqref{eqn:T'' neg}.
    Thus, $(\mu, \nu)$ have $(\sigma+\eta, 2^{k+1} \cdot K^d r_0^{-(\sigma+\eta-(k-1))})$-thin tubes on $G \setminus B_1$ down from scale $r_0$. Now we move to upper-bounding $\mu \times \nu(B_1)$. By \eqref{gather1} and \eqref{eqn:T''}, we have $H_r \neq \emptyset$ for all $r > r_0 K^{-d/\eta}$, and so if $K \ge K_0$ from Claim, then
    \begin{equation*}
        \mu \times \nu(B_1) \le \sum_{r \le r_0 K^{-d/\eta} \text{ dyadic }} \mu \times \nu(H_r) \le\sum_{r \le r_0 K^{-d/\eta} \text{ dyadic }} 2\left( \frac{r}{r_0} \right)^{\eta} \le C_\eta K^{-d}.
    \end{equation*}
    Let $K_0$ be the maximum of the value of $K_0$ from Claim, $2C_\eta$, and $2^{k+1}$. Since $d \ge 2$, we get $\mu \times \nu(B_1) \le \frac{1}{2} K^{-1}$ and $(\mu, \nu)$ have $(\sigma+\eta, K^{d+1} r_0^{-(\sigma+\eta-(k-1))})$-thin tubes on $G \setminus B_1$ down from scale $r_0$. We can analogously find $B_2 \subset X \times Y$ with $\mu \times \nu(B_2) \le \frac{1}{2} K^{-1}$ such that $(\nu, \mu)$ have $(\sigma+\eta, K^{d+1} r_0^{-(\sigma+\eta-(k-1))})$-thin tubes on $G \setminus B_2$ down from scale $r_0$, and so $B = B_1 \cup B_2$ would be a good choice. Now we turn to proving the Claim.

    \textit{Proof of Claim.} We will choose $\eta = \min\{ \frac{1}{2} (6 + \frac{15(d-1)}{s - \max(\sigma, k-1)} )^2, \frac{1}{5} \eps^2\}$, where $\eps$ is obtained from Theorem \ref{thm:main_refined}. From Remark \ref{rmk:uniform eps} and the continuity of the function $(s, \sigma, k) \mapsto (s - \max(\sigma, k-1))^{-1}$, we see that $\eta(\sigma, s, \kappa, k, d)$ is bounded away from zero on any compact subset of $\{ (\sigma, s, \kappa, k, d) : \max(\sigma, k-1) < s \le k \le d-1 \}$.
    
    Suppose that Claim is false. Let $\bX = \{ x \in X : \nu(H_r) \ge \left( \frac{r}{r_0} \right)^\eta \}$. Then $\mu(\bX) \ge \left( \frac{r}{r_0} \right)^\eta$.

    Recall that for $x \in X$, the fiber $H_r|_x$ is covered by $\T'_{x,r}$, which is a set of cardinality $\lesim \frac{r^{-(\sigma+\eta)}}{r_0^{-(\sigma+\eta-(k-1))}}$. Let
    \begin{equation*}
        \T_x = \{ T \in \T'_{x,r} : \nu(T \cap H_r|_x) \ge \frac{r^{\sigma+3\eta}}{r_0^{\sigma+3\eta-(k-1)}} \}, \qquad Y_x = (H_r|_x) \cap \bigcup \T_x.
    \end{equation*}
    Then $\nu(Y_x) \ge \left( \frac{r}{r_0} \right)^\eta - \left( \frac{r}{r_0} \right)^{2\eta} \ge \left( \frac{r}{r_0} \right)^{2\eta}$ for all $x \in \bX$. Furthermore, for every $T \in \T_x$, we have
    \begin{equation}\label{eqn:T cap Hr}
        \frac{r^{\sigma+3\eta}}{r_0^{\sigma+3\eta-(k-1)}} \le \nu(T \cap Y_x) \le \frac{r^{\sigma-\eta}}{r_0^{\sigma-\eta-(k-1)}}.
    \end{equation}
    The upper bound follows from $Y_x \subset H_r|_x \subset G|_x$, \eqref{gather1}, and $K \le \left( \frac{r}{r_0} \right)^{-\eta}$. In fact, we have in general,
    \begin{equation*}
        \nu(T^{(\rho)} \cap Y_x) \le \left( \frac{r}{r_0} \right)^{-\eta} \rho^\eta, \qquad \rho \in [r, 1], T \in \T_x.
    \end{equation*}
    We also take the time to state the thin plates assumption:
    \begin{equation*}
        \nu(W^{(\rho)} \cap Y_x) \le \left( \frac{r}{r_0} \right)^{\kappa-\eta} \qquad \rho \in [r, 1], W \text{ is } (\rho, k)\text{-plate}.
    \end{equation*}
    Since $\cup \T_x$ covers $Y_x$, we get by the upper bound in \eqref{eqn:T cap Hr}, $|\T_x| \gesim \frac{r^{-\sigma+\eta}}{r_0^{-\sigma+\eta+(k-1)}} \nu(Y_x) \ge \frac{r^{-\sigma+3\eta}}{r_0^{-\sigma+3\eta+(k-1)}}$. Hence, $\T_x$ is a $(r, \sigma, r_0^{-(\sigma-(k-1))} \left(\frac{r}{r_0} \right)^{-5\eta})$-set and $(r, \kappa, r_0^{-\kappa} \left(\frac{r}{r_0} \right)^{-5\eta}, k-1)$-set for each $x \in \bX$.

    Let $\gamma = \frac{15\eta}{s-\max(\sigma,k-1)}$. Call a tube $T \in \T_x$ concentrated if there is a ball $B_T$ with radius $\left(\frac{r}{r_0}\right)^\gamma$ such that
    \begin{equation}\label{eqn:non-conc}
        \nu(T \cap B_T \cap Y_x) \ge \frac{1}{3} \cdot \nu(T \cap Y_x).
    \end{equation}
    Suppose that there is $\bX' \subset \bX$ with $\mu(\bX') \ge \mu(\bX)/2$ such that for each $x \in \bX'$, at least half the tubes of $\T_x$ are non-concentrated. Since $\mu(\bX') \ge \frac{1}{2} \mu(\bX)/2 \ge \frac{1}{2} \left( \frac{r}{r_0} \right)^{2\eta}$ and $\mu$ is Frostman with constant $K \le \left( \frac{r}{r_0} \right)^{-\eta}$, we can find a $(r, \sigma, \left( \frac{r}{r_0} \right)^{-3\eta})$-set $P \subset \bX'$. For each $x \in \bX'$, the set of non-concentrated tubes $\T'_x \subset \T_x$ is a $(r, \sigma, 2r_0^{-(\sigma-(k-1))} \left(\frac{r}{r_0} \right)^{-5\eta})$-set and $(r, \kappa, 2r_0^{-\kappa} \left(\frac{r}{r_0} \right)^{-5\eta}, k-1)$-set. Let $\T = \cup_{x \in P} \T'_x$. By Lemma \ref{lem:concentrated_points}, since $d(X, Y) \ge K^{-1}$, we have that $\T$ is contained in the $O(K) \cdot r_0$-neighborhood of $H$. Now, we apply Theorem \ref{thm:main_refined} with $\overline{r}_0 := \min(O(K) \cdot r_0, 1)$. Since $K \le \left( \frac{r}{r_0} \right)^{-\eta}$ and $\sigma \le k$, we still have that for each $x \in \bX'$, the set of non-concentrated tubes $\T_x'$ is a $(r, \sigma, 2\overline{r}_0^{-(\sigma-(k-1))} \left(\frac{r}{\overline{r}_0} \right)^{-7\eta})$-set and $(r, \kappa, 2\overline{r}_0^{-\kappa} \left(\frac{r}{\overline{r}_0} \right)^{-7\eta}, k-1)$-set. At this point, let us remark that implicit constants are dominated by $\left( \frac{r}{\overline{r}_0} \right)^{-\eta} \ge K^\eta$ if $K \ge K_0 (\eta)$ is chosen large enough. 
    
    Then if $\eta \le \eps^2/4$, where $\eps$ is obtained from Theorem \ref{thm:main_refined}, then
    \begin{equation*}
        |\T| \ge \frac{r^{-2\sigma - 2\sqrt{\eta}}}{\overline{r}_0^{-2(\sigma-(k-1))- 2\sqrt{\eta}}} \ge \frac{r^{-2\sigma - \sqrt{\eta}}}{r_0^{-2(\sigma-(k-1))-\sqrt{\eta}}}.
    \end{equation*}
    In other words, we get a gain of $\left( \frac{r}{r_0} \right)^{-\sqrt{\eta}}$, which means a two-ends argument gives an immediate contradiction. Specifically, by \eqref{eqn:T cap Hr} and \eqref{eqn:non-conc}, we have for each non-concentrated $T \in \T$, $\nu \times \nu(\{(x, y) : x, y \in T, d(x, y) \ge \left( \frac{r}{r_0} \right)^\gamma \}) \ge \frac{2}{3} \nu(T \cap Y_x)^2 \ge \frac{r^{2\sigma+6\eta}}{r_0^{2\sigma+6\eta-2(k-1)}}$. Thus, by Fubini, there exists a pair $(x, y)$ with $d(x, y) \ge \left( \frac{r}{r_0} \right)^\gamma$ such that $x, y \in T$ for $\gesim \frac{r^{2\sigma+6\eta}}{r_0^{2\sigma+6\eta-2(k-1)}} |\T| \ge \left( \frac{r}{r_0} \right)^{-\sqrt{\eta}+6\eta}$ many tubes $T \in \T$. However, since $d(x, y) \ge \left( \frac{r}{r_0} \right)^\gamma$, we have that $x, y$ can only lie in $\lesim \left( \frac{r}{r_0} \right)^{-(d-1)\gamma}$ many essentially distinct $2r$-tubes. Since $\sqrt{\eta} - 6\eta \ge (d-1)\gamma$, we get a contradiction.

    Now we focus on the concentrated case: assume there is a subset $\bX' \subset \bX$ with $\mu(\bX') \ge \mu(\bX)/2$ such that at least half of the tubes in $\T_x$ are concentrated for all $x \in \bX'$. This case is where we use the fact that $\nu$ is a $s$-dimensional measure. Let $\T_x'$ denote the concentrated tubes and $\{ B_T : T \in \T_x' \}$ denote the corresponding heavy $\left( \frac{r}{r_0} \right)^\gamma$-balls. Because the family $\T_x$ has $K$-overlap on $\spt(\nu)$, the set
    \begin{equation*}
        H' = \{ (x, y) : x \in \bX', y \in T \cap B_T \cap Y_x \text{ for some } T \in \T_x' \}
    \end{equation*}
    has measure
    \begin{multline*}
        (\mu \times \nu)(H') \gesim K^{-1} \cdot \mu(\bX') \cdot \inf_{x \in \bX'} |\T_x'| \cdot \inf_{x \in \bX', T \in \T_x'} \nu(T \cap B_T \cap Y_x) \\
            \gesim \left( \frac{r}{r_0} \right)^{2\eta} \cdot \frac{r^{-\sigma+3\eta}}{r_0^{-(\sigma-3\eta-(k-1))}} \cdot \frac{r^{\sigma + 3\eta}}{r_0^{\sigma + 3\eta-(k-1)}} = \left( \frac{r}{r_0} \right)^{8\eta}.
    \end{multline*}
    Notice that if $(x, y) \in H'$, then there is a tube $T(x, y) \in \T^r$ containing $x, y$ such that 
    \begin{equation*}
        \nu(B(y, 2(r/r_0)^\gamma) \cap T(x, y)) \gesim \frac{r^{\sigma + 3\eta}}{r_0^{\sigma + 3\eta-(k-1)}}.
    \end{equation*}
    Thus, $\nu$ can't be too concentrated near $y$:
    \begin{equation*}
        \nu(B(y, r)) \le K \cdot r^s \le \frac{1}{2} \nu(B(y, 2\left( \frac{r}{r_0} \right)^\gamma) \cap T(x, y)),
    \end{equation*}
    assuming $4\eta < s - \sigma$ and $k-1 < s$. (The relevant inequalities are $K \le \left( \frac{r}{r_0} \right)^{-\eta}$ and $r^{s-\sigma-3\eta} \le r_0^{s-\sigma-3\eta} \le r_0^{k-1-\sigma-3\eta}$.)

    Therefore, for each $(x, y) \in H'$, we can choose a dyadic number $r \le \xi(x, y) \le (r/r_0)^\gamma$ such that
    \begin{equation*}
        \nu(A(y, \xi(x, y), 2\xi(x, y)) \cap T(x,y)) \ge \left( \frac{r}{r_0} \right)^{\sigma+4\eta} \left( \frac{\xi(x,y)}{(r/r_0)^\gamma} \right)^\eta r_0^k,
    \end{equation*}
    where the annulus $A(y, \xi, 2\xi) := B(y, 2\xi) \setminus B(y, \xi)$.
    (One remark: \cite{orponen2022kaufman} used dyadic pigeonholing at this step, but we can't do this because then we would introduce a $\log r_0^{-1}$ factor. Fortunately, we are allowed to introduce the decaying tail $\left( \frac{\xi(x,y)}{(r/r_0)^\gamma} \right)^\eta$, which is summable in $\xi(x,y)$.)

    Then, recalling that $(\mu \times \nu)(H') \gesim \left( \frac{r}{r_0} \right)^{7\eta}$, we can further find $r \le \xi \le \left( \frac{r}{r_0} \right)^\gamma$ such that
    \begin{equation*}
        (\mu \times \nu)(H'') \ge \left( \frac{r}{r_0} \right)^{8\eta} \left( \frac{\xi(x,y)}{(r/r_0)^\gamma} \right)^\eta, \text{ where } H'' = \{ (x, y) \in H' : \xi(x,y) = \xi \} \subset G.
    \end{equation*}
    By Fubini, we can find $y \in Y$ such that $\mu(H''|^y) \ge \left( \frac{r}{r_0} \right)^{8\eta} \left( \frac{\xi(x,y)}{(r/r_0)^\gamma} \right)^\eta$. Then by construction, $H''|^y$ can be covered by a collection of tubes $\T_y \subset \T^r$ containing $y$ that satisfy
    \begin{equation*}
        \nu(A(y, \xi, 2\xi) \cap T) \ge \nu(A(y, \xi(x, y), 2\xi(x, y)) \cap T(x,y)) \ge \left( \frac{r}{r_0} \right)^{\sigma+4\eta} \left( \frac{\xi(x,y)}{(r/r_0)^\gamma} \right)^\eta r_0^k.
    \end{equation*}
    Finally, we claim that $\T_y$ contains a subset $\T_y'$ whose directions are separated by $\ge (r/\xi)$, such that $|\T_y'| \gesim \mu(H''|^y) \cdot r^\eta \cdot \left( \frac{\xi r_0}{r} \right)^\sigma r_0^{-k}$ if $\xi > \frac{r}{r_0}$ and $|\T_y'| \gesim \mu(H''|^y) \cdot r^\eta \cdot \left( \frac{\xi r_0}{r} \right)^\sigma r_0^{-k}$ if $\xi > \frac{r}{r_0}$ if $r < \xi < \frac{r}{r_0}$. Indeed, if $\xi > \frac{r}{r_0}$, then any $r/\xi$-tube $\bT$ containing $y$ has
    \begin{equation*}
        \mu(\bT \cap H''|^y) \le \mu(\bT \cap A|^y) \le K \cdot \left( \frac{r}{\xi r_0} \right)^\sigma r_0^k \le \left( \frac{r}{r_0} \right)^\eta \cdot \left( \frac{r}{\xi r_0} \right)^\sigma r_0^k.
    \end{equation*}
    If $\xi < \frac{r}{r_0}$, then any $r/\xi$-tube $\bT$ containing $y$ lies in the union of $(\frac{r}{\xi r_0})^k$ many $r_0$-tubes, and so
    \begin{equation*}
        \mu(\bT \cap H''|^y) \le \mu(\bT \cap A|^y) \le K \cdot \left( \frac{r}{\xi r_0} \right)^{-k} r_0^k \le \left( \frac{r}{r_0} \right)^\eta \cdot \left( \frac{r}{\xi} \right)^k.
    \end{equation*}
    Thus, if $\xi > \frac{r}{r_0}$, then it takes $\gesim \mu(H''|^y) \cdot r^\eta \cdot \left( \frac{\xi r_0}{r} \right)^\sigma r_0^{-k}$ many $(r/\xi)$-tubes to cover $H''|^y$, and perhaps even more to cover $\cup \T_y$. We may now choose $\T_y' \subset \T_y$ to be a maximal subset with $(r/\xi)$-separated directions to prove the claim for $\xi > \frac{r}{r_0}$. A similar argument holds for $\xi < \frac{r}{r_0}$.

    Finally, let's first assume $\xi > \frac{r}{r_0}$. Since $\T^y$ has bounded overlap in $\R^d \setminus B(y, \xi)$, we obtain
    \begin{multline*}
        \left( \frac{r}{r_0} \right)^{\sigma+13\eta} \left( \frac{\xi}{(r/r_0)^\gamma} \right)^{2\eta} \cdot \left( \frac{\xi r_0}{r} \right)^\sigma r_0^{-k} \\
        \lesim \inf_{T \in \T_y'} \nu(A(y, \xi, 2\xi) \cap T) \cdot |\T_y'| \lesim \nu(B(y, 2\xi)) \le C \cdot (2\xi)^s.
    \end{multline*}
    We will obtain a contradiction if we show the opposite inequality holds, for $\gamma = \frac{15\eta}{s - \max(k, \sigma)}$. Since $2\eta + \sigma < s$ and $\xi \le (\frac{r}{r_0})^\gamma$, it suffices to check $\xi = (\frac{r}{r_0})^\gamma$.

    If $\xi < \frac{r}{r_0}$, then we obtain
    \begin{multline*}
        \left( \frac{r}{r_0} \right)^{\sigma+13\eta} \left( \frac{\xi}{(r/r_0)^\gamma} \right)^{2\eta} \cdot \left( \frac{\xi r_0}{r} \right)^k r_0^{-k} \\
        \lesim \inf_{T \in \T_y'} \nu(A(y, \xi, 2\xi) \cap T) \cdot |\T_y'| \lesim \nu(B(y, 2\xi)) \le C \cdot (2\xi)^s.
    \end{multline*}
    Again, since $2\eta + k < s$, it suffices to check $\xi = \frac{r}{r_0}$.

    This proves the result.
\end{proof}

\begin{proof}[Proof of Theorem \ref{thm:threshold_refined}]
    By Propositions \ref{prop:b1} (with $\frac{r_0}{K_2}$ for $r_0$) and \ref{prop:k-thin}, there exists a set $B_0 \subset X \times Y$ with $\mu \times \nu(B_0) \lesim K_1^{-1}$ such that $(\mu, \nu)$ and $(\nu, \mu)$ have $(0, K_1^N r_0^{k-1})$-thin tubes on $B_0^c$ down from scale $r_0$, and $(\mu, \nu)$ and $(\nu, \mu)$ have $(\kappa, K_1 \left( \frac{r_0}{K_2} \right)^{-\kappa})$-thin $k$-plates on $(A \cup B_0)^c$. Then iterate Proposition \ref{prop:bootstrap} applied to a uniform $\eta(\sigma, s, \kappa, k, d)$. So initially we have $K = \max(K_1^N K_2^\kappa, K_0 (\eta, k))$, and after each iteration, $K$ becomes $K^{d+1}$. After iterating $\lesim \eta^{-1}$ many times and letting $B_1 \subset X \times Y$ be the union of the $B$'s outputted from the Proposition (so $\mu \times \nu(B_1) \lesim K^{-1} \le K_1^{-1}$), we find that $(\mu, \nu)$ and $(\nu, \mu)$ have $(\sigma, K^{(d+1)\eta^{-1}} r_0^{k-1})$-thin tubes on $(A \cup B_0 \cup B_1)^c$. Then we can take $B := B_0 \cup B_1$ to be our desired set.
\end{proof}

\subsection{Proof of Theorem \ref{thm:intermediate}, general case}
We will prove Theorem \ref{thm:intermediate}, which we restate here.

\begin{theorem}\label{thm:intermediate_refined}
Let $k \in \{ 1, 2, \cdots, d-1 \}$, $k-1 < \sigma < s \le k$, and $\eps > 0$. There exist $N, K_0$ depending on $\sigma, s, k$, and $\eta(\eps) > 0$ (with $\eta(1) = 1$) such that the following holds. Fix $r_0 \le 1$, and $K \ge K_0$. Let $\mu, \nu$ be $\sim 1$-separated $s$-dimensional measures with constant $C_\mu, C_\nu$ supported on $E_1, E_2$, which lie in $B(0, 1)$. Assume that $|\mu|, |\nu| \le 1$. Let $A$ be the pairs of $(x, y) \in E_1 \times E_2$ that lie in some $K^{-1}$-concentrated $(r_0, k)$-plate. Then there exists a set $B \subset E_1 \times E_2$ with $\mu \times \nu (B) \lesim K^{-\eta}$ such that for every $x \in E_1$ and $r$-tube $T$ through $x$, we have
\begin{equation*}
    \nu(T \setminus (A|_x \cup B|_x)) \lesim \frac{r^\sigma}{r_0^{\sigma-(k-1)+N\eps}} K^N.
\end{equation*}
The implicit constant may depend on $C_\mu, C_\nu, \sigma, s, k$.
\end{theorem}

\begin{remark}
    Note that in Theorem \ref{thm:threshold_refined}, we demand the stronger conclusion $\mu \times \nu(B) \lesim K^{-1}$.
\end{remark}

The idea is to apply Theorem \ref{thm:threshold_refined} at different scales. As a start, if $\eps = 1$, then we can directly apply Theorem \ref{thm:threshold_refined} with $K_1 = K_2 = K$ (and thus we may take $\eta(1) = 1$).

We may assume $\eps = \frac{1}{M}$ for some $M$.
Let $N$ be the large constant in Lemma \ref{lem:few_large_plates}, and let $\eta_n = (N+2)^{n-M}$. For $1 \le n \le M$, let $A_n$ be the pairs of $(x, y) \in E_1 \times E_2$ that lie in some $K^{-\eta_n}$-concentrated $(r_0^{n\eps}, k)$-plate. We remark that $A_M = A$.



\begin{lemma}\label{lem:intermediate_r}
    Fix $n \ge 1$. There exists a set $B_n \subset A_n$ with $\mu \times \nu(B_n) \lesim K^{-\eta_n}$ such that for every $x \in E_1$ and $r$-tube through $x$ that intersects $A_n|_x$, we have
    \begin{equation*}
        \nu(T \setminus (A_{n+1}|_x \cup B|_x)) \lesim \frac{r^\sigma}{r_0^{n\eps (\sigma-(k-1)) + N\eps}} K^N.
    \end{equation*}
\end{lemma}

\begin{proof}
    By Lemma \ref{lem:concentrated_points}, there exists an absolute constant $C$ such that every $r$-tube through some $(x, y) \in A_n$ lies in some $K^{-\eta_n}$-concentrated $(CK^{-n}, k)$-plate. We can find a collection $\cH$ of essentially distinct $K^{-\eta_n}$-concentrated $(2CK^{-n}, k)$-plates such that each $K^{-\eta_n}$-concentrated $(CK^{-n}, k)$-plate is contained within some element of $\cH$. By Lemma \ref{lem:few_large_plates}, $|\cH| \lesim K^{-N\eta_n}$. By construction, every $r$-tube through some $(x, y) \in A_n$ is contained in some member of $\cH$. Apply Theorem \ref{thm:threshold_refined} to each $H \in \cH$ with measures $\mu|_H, \nu|_H$ and $K_1 \rarrow K^{-\eta_{n+1}}$, $K_2 \rarrow 2Cr_0^{-\eps}$, and $r_0 \rarrow r_0^{n\eps}$ to obtain a set $B_H$ with $\mu \times \nu(B_H) \lesim K^{-\eta_{n+1}}$. Let $B_n = \cup_{H \in \cH} B_H$, and then $\mu \times \nu(B_n) \le K^{N\eta_n} \cdot K^{-\eta_{n+1}} < K^{-\eta_n}$ since $(N+1)\eta_n < \eta_{n+1}$.
\end{proof}

\begin{proof}[Proof of Theorem \ref{thm:intermediate_refined}]
    Let $B = \cup_{n=1}^M B_n$; then $\mu \times \nu(B) \le K^{-\eta_0}$. Fix an $r$-tube $T$ and $x \in E_1$. Let $n \le M-1$ be the largest number such that $T$ passes through points in $A_n|_x$. Then by Lemma \ref{lem:intermediate_r}, we have $\nu(T \setminus (A_{n+1}|_x \cup B|_x)) \lesim \frac{r^\sigma}{K^{-n (\sigma-(k-1)}} K^N$. If $m < M-1$, then $T \cap A_{n+1}|_x = \emptyset$. In any case, we have $\nu(T \setminus (A|_x \cup B|_x)) \lesim \frac{r^\sigma}{r_0^{\sigma-(k-1)+N\eps}} K^N$, completing the proof of Theorem \ref{thm:intermediate_refined}.

\end{proof}

\section{Corollaries of Radial Projection Estimates}\label{sec:corollaries}
We prove a variant of Corollary \ref{cor:shm_conj}. 
\begin{prop}\label{cor:shm_conj'}
    Fix $s \in (k-1, k]$ and $\eta > 0$. Let $\mu, \nu \in \cP(\R^d)$ be measures with $\cE_s (\mu), \cE_s (\nu) < \infty$ and $\sim 1$-separated supports. Suppose that $\mu(H) = \nu(H) = 0$ for each $k$-plane $H \in \A(\R^d, k)$. Then for $\mu$-almost all $x$, for all sets $Y$ of positive $\nu$-measure,
    \begin{equation*}
        \dim_H (\pi_x Y) \ge s - \eta.
    \end{equation*}
\end{prop}

\begin{proof}
    The proof is standard and follows \cite[Proof of Proposition 6.9]{shmerkin2022dimensions}.
    By Lemma \ref{lem:energy}, by passing to subsets of nearly full measure and replacing $s$ by an arbitrary $s' < s$, we may assume that $\mu(B_r), \nu(B_r) \lesim r^s$ for all $r \in (0, 1]$.

    Fix $\eps > 0$. By a compactness argument, there exists $r_0 > 0$ such that $\mu(H), \nu(H) < \eps$ for all $(r_0, k)$-plates $H$. In Theorem \ref{thm:threshold_refined}, we know that for $\eps > 0$ sufficiently small, the set $A = \emptyset$. Thus, there exists $B \subset X \times Y$ with $\mu \times \nu(B) \lesim \eps$ such that for every $x \in X$ and $r$-tube through $x$, we have
    \begin{equation*}
        \nu(T \setminus B|_x) \lesim_{\eta,\eps,s} r^{s - \eta}.
    \end{equation*}
    Thus, there is a set $X$ with $\mu(X) > 1-O(\eps)$ such that if $x \in X$, then
    \begin{equation*}
        \dim_H (\pi_x Y) \ge s - \eta \text{ for all } Y \text{ with } \nu(Y) \ge O(\eps).
    \end{equation*}
    Taking $\eps \to 0$ completes the proof.
\end{proof}

Using this, we prove Corollary \ref{cor:shm_conj_2}.
\begin{cor}\label{cor:shm_conj_2'}
    Let $s \in (d-2, d]$, then there exists $\eps(s, d) > 0$ such that the following holds. Let $\mu, \nu$ be Borel probability measures on $\R^d$ with disjoint supports that satisfy $\cE_s (\mu), \cE_s (\nu) < \infty$ and $\dim_H (\spt(\nu)) < s + \eps(s, d)$. Further, assume that $\mu, \nu$ don't simultaneously give full measure to any affine $(d-1)$-plane $H \subset \R^d$. Then there exist restrictions of $\mu, \nu$ to subsets of positive measure (which we keep denoting $\mu, \nu$) such that the following holds. For almost every affine 2-plane $W \subset \R^d$ (with respect to the natural measure on the affine Grassmanian), if the sliced measures $\mu_W$, $\nu_W$ on $W$ is non-trivial, then they don't simultaneously give full measure to any line.
    In other words,
    \begin{equation*}
        (\gamma_{d,2} \times \mu) \{ (V, x) : \mu_{V,x} (\ell) \nu_{V,x} (\ell) = |\mu_{V,x}| |\nu_{V,x}| > 0 \text{ for some } \ell \in \A(V + x, 1) \} = 0
    \end{equation*}
    where we parametrize affine 2-planes as $V + x$, for $x \in \R^d$ and $V$ in the Grassmannian $\Gr(d, 2)$ with the rotationally invariant Haar measure $\gamma_{d,2}$.
\end{cor}

\begin{proof}
    First, if $\mu(H) > 0$ for some affine $(d-1)$-plane $H$, then $\nu(H^c) > 0$ where $H^c$ denotes the complement of $H$ in $\R^d$. By restricting $\mu$ to $H$ and $\nu$ to $H^c$ (and calling the results $\mu, \nu$), we see that the sliced measures $\mu_W$ and $\nu_W$ can't give full mass to any line $\ell$ for any affine $(d-1)$-plane $W$, for the simple reason that $\mu_W (\ell) > 0$ forces $\ell \subset H$, and $\nu_W (\ell) > 0$ forces $\ell \subset H^c$. Likewise, we are done if $\nu(H) > 0$ for some affine $(d-1)$-plane $H \subset \R^d$. Thus, assume $\mu(H) = \nu(H) = 0$ for all affine $(d-1)$-planes $H$.

    With this assumption, the remainder of the proof is nearly identical to the proof of Proposition 6.8 in \cite{shmerkin2022dimensions}, except using Proposition \ref{cor:shm_conj'} instead of \cite[Proposition 6.9]{shmerkin2022dimensions}. One can take $\eps(s, d)$ to be arbitrarily close to $s - (d-2)$.
\end{proof}

Finally, we can deduce Theorem \ref{cor:shm_conj} from either Proposition \ref{cor:shm_conj'} or Proposition \ref{cor:shm_conj_2'}, see \cite[Section 4]{orponen2022kaufman} for details. The only case not yet considered in this paper is when either $\mu, \nu$ gives positive mass to a $k$-plane. But this special case was considered in \cite[Section 4]{orponen2022kaufman} (briefly, if $X$ gives positive mass to some $k$-plane, then radial projections become orthogonal projections and then we apply Kaufman's projection theorem; if $Y$ gives positive mass to some $k$-plane $H$, then for $x \notin H$, we have $\dim_H (\pi_x (Y)) = \dim_H (Y)$.) 

\appendix

\section{Proof of Balog-Szemer\'{e}di-Gowers}\label{appendix:proof of bsg}
By a standard covering argument (e.g. see Section 3 of \cite{katz2001some}), Theorem \ref{thm:bsg} follows from the case $\delta = 0$, which we prove below.

\begin{theorem}[refined Theorem 4.1 of \cite{sudakov2005question}]\label{thm:bsg'}
    Let $K \ge 1$ be a parameter. Let $A, B$ be finite subsets of $\R^n$, and let $P \subset A \times B$ satisfy $|P| \ge K^{-1} |A| |B|$. Suppose that $|A \plusP B| \le K (|A| |B|)^{1/2}$, where $A \plusP B = \{ a + b : (a, b) \in P \}$. Then one can find subsets $A' \subset A, B' \subset B$ with $|A'| \ge \frac{1}{16K^2} |A|, |B'| \ge \frac{1}{16K^2} |B|$ such that $|A' + B'| \le 2^{12} K^8 (|A| |B|)^{1/2}$ and $|P \cap (A' \times B')| \ge \frac{|A| |B|}{16K^2}$.
\end{theorem}

\begin{proof}
    We follow the exposition in \cite{sheffer2016balog}.
    
    \textbf{Claim.} There exist subsets $A' \subset A, B' \subset B$ with $|P \cap (A' \times B')| \ge \frac{|A| |B|}{16K^2}$, such that for each $a \in A', b \in B'$, there are $\ge \frac{|A||B|}{2^{12} K^5}$ many pairs $(a', b') \in A \times B$ such that $(a, b')$, $(a', b')$, and $(a', b) \in P$.

    Assuming the claim, we will see how the theorem follows. First, we get $|A'| |B'| \ge |P \cap (A' \times B')| \ge \frac{|A| |B|}{16K^2}$. Since $|A'| \le |A|$ and $|B'| \le |B|$, we get $|A| \ge \frac{|A|}{16K^2}$ and $|B| \ge \frac{|B|}{16K^2}$.

    Next, for $a \in A', b \in B'$, we have
    \begin{equation*}
        a + b = (a + b') - (a' + b') + (a' + b).
    \end{equation*}
    Thus, there are $\ge |A| |B| 2^{-12} K^{-5}$ many solutions to $a + b = x - y + z$ with $x, y, z \in A \plusP B$. Since $|A \plusP B| \le K(|A| |B|)^{1/2}$, we get $|A' + B'| \lesim \frac{K^3 (|A| |B|)^{3/2}}{|A| |B| 2^{-12} K^{-5}} = 2^{12} K^8 |A|^{1/2} |B|^{1/2}$.

    Now, we prove the claim. For convenience, we can prune $P$ to satisfy $|P| = K^{-1} |A| |B|$ (this is not necessary but will make the proof look nicer). Treat $(A \cup B, P)$ as a bipartite graph with an edge between $a \in A$ and $b \in B$ if $(a, b) \in P$. Then we want to find $A', B'$ such that there are many paths of length $3$ between any $a \in A', b \in B'$.

    The average degree of a vertex in $A$ is $K^{-1} |B|$. Thus, if we delete the vertices in $A$ with degree $\le \frac{1}{2} K^{-1} |B|$, then at least $\frac{1}{2K} |A| |B|$ many edges remain. Let $E$ be the set of edges. For $v \in A \cup B$, let $N(v)$ be the set of neighbors of $v$.

    Now pick a vertex $b \in B$. On average, it has $\frac{|E|}{|B|} \ge \frac{1}{2K} |A|$ many neighbors. 

    Now, we say $(a, a') \in A^2$ is bad if $|N(a) \cap N(a')| < \frac{1}{128K^3} |B|$. For $v \in B$, let $\Bad_v$ be the set of bad pairs in $N(v)^2$. There are $\binom{|A|}{2}$ many pairs in $A$, so (expectation is taken over uniformly chosen $v \in B$)
    \begin{equation*}
        \E[|\Bad_v|] < \binom{|A|}{2} \cdot \frac{1}{128K^3} < \frac{|A|^2}{256K^3}.
    \end{equation*}
    If $A_{bad,v}$ is the set of vertices of $A$ that lie in at least $\frac{|A|}{32K^2}$ many pairs of $B_v$, then
    \begin{equation*}
        \E[|A_{bad,v}|] \le \frac{2\E[|B_v|]}{|A|/(32K^2)} < \frac{|A|}{4K}.
    \end{equation*}
    Finally, let $A_v = N(v) \setminus A_{bad,v}$. Then by linearity of expectation,
    \begin{equation*}
        \E[|A_v|] = \E[|N(v)|] - \E[|A_{bad,v}|] > \frac{|A|}{2K} - \frac{|A|}{4K} = \frac{|A|}{4K}.
    \end{equation*}
    Thus, there exists $v \in B$ such that $|A_v| > \frac{|A|}{4K}$. Then, let $A' = A_v$ and
    \begin{equation*}
        B' = \{ w \in B : |N(w) \cap A'| \ge \frac{|A|}{16K^2}.
    \end{equation*}
    Let $E(X, Y)$ be the number of edges between $X$ and $Y$.
    We first check that $E(A', B') \ge \frac{|A||B|}{16K^2}$. Indeed, since every vertex of $A$ has degree $\ge \frac{|B|}{2K}$, we have
    \begin{equation*}
        |E(A', B)| \ge \frac{|A'| |B|}{2K} \ge \frac{|A||B|^2}{8K^2}.
    \end{equation*}
    On the other hand, every vertex in $B \setminus B'$ corresponds to fewer than $\frac{|A|}{16K^2}$ many edges of $A'$, so $|E(A', B \setminus B')| \le \frac{|A||B|^2}{16K^2}$. Hence, $|E(A', B')| \ge \frac{|A||B|^2}{16K^2}$.

    Finally, for any $v \in A'$, $w \in B'$, we know that $w$ has at least $\frac{|A|}{16K^2}$ many neighbors in $A'$, and fewer than $\frac{|A|}{32K^2}$ of those form a bad pair with $w$. For the remaining $\ge \frac{|A|}{32K^2}$ vertices $v'$ that do not form a bad pair with $w$, there are $\ge \frac{|B|}{128K^3}$ many vertices $w' \in B$ that are common neighbors of $v, v'$. Thus, we get at least $\frac{|A|}{32K} \cdot \frac{|B|}{128K^3} = \frac{|A||B|}{2^{12} K^5}$ many paths $(v, w', v', w)$ between $v$ and $w$.
\end{proof}

\bibliographystyle{plain}
\bibliography{main}

\begin{thebibliography}{10}

\bibitem{bourgain2010discretized}
Jean Bourgain.
\newblock The discretized sum-product and projection theorems.
\newblock {\em Journal d'Analyse Math{\'e}matique}, 112(1):193--236, 2010.

\bibitem{carbery2013endpoint}
A.~Carbery and S.I. Valdimarsson.
\newblock The endpoint multilinear {K}akeya theorem via the {B}orsuk--{U}lam
  theorem.
\newblock {\em Journal of Functional Analysis}, 264(7):1643--1663, 2013.

\bibitem{du2021improved}
X.~Du, A.~Iosevich, Y.~Ou, H.~Wang, and R.~Zhang.
\newblock An improved result for {F}alconer’s distance set problem in even
  dimensions.
\newblock {\em Mathematische Annalen}, 380(3):1215--1231, 2021.

\bibitem{DOKZFalconer}
X.~Du, Y.~Ou, K.~Ren, and R.~Zhang.
\newblock New improvement to {F}alconer distance set problem in higher
  dimensions.
\newblock {\em arXiv preprint}, 2023.

\bibitem{DOKZFalconerDec}
X.~Du, Y.~Ou, K.~Ren, and R.~Zhang.
\newblock Weighted refined decoupling estimates and application to {F}alconer
  distance set problem.
\newblock {\em arXiv preprint}, 2023.

\bibitem{du2019sharp}
X.~Du and R.~Zhang.
\newblock Sharp $l^2$ estimates of the {S}chr{\"o}dinger maximal function in
  higher dimensions.
\newblock {\em Annals of Mathematics}, 189(3):837--861, 2019.

\bibitem{guth2020falconer}
L.~Guth, A.~Iosevich, Y.~Ou, and H.~Wang.
\newblock On {F}alconer’s distance set problem in the plane.
\newblock {\em Inventiones Mathematicae}, 219(3):779--830, 2020.

\bibitem{guth2019incidence}
L.~Guth, N.~Solomon, and H.~Wang.
\newblock Incidence estimates for well spaced tubes.
\newblock {\em Geometric and Functional Analysis}, 29(6):1844--1863, 2019.

\bibitem{harris2021low}
T.L.J. Harris.
\newblock Low-dimensional pinned distance sets via spherical averages.
\newblock {\em The Journal of Geometric Analysis}, 31(11):11410--11416, 2021.

\bibitem{he2020orthogonal}
W.~He.
\newblock Orthogonal projections of discretized sets.
\newblock {\em Journal of Fractal Geometry}, 7(3):271--317, 2020.

\bibitem{he2016discretized}
Weikun He.
\newblock Discretized sum-product estimates in matrix algebras.
\newblock {\em Journal d'Analyse Math{\'e}matique}, 139(2):637--676, 2019.

\bibitem{katz2001some}
Nets~Hawk Katz and Terence Tao.
\newblock Some connections between falconer’s distance set conjecture and
  sets of furstenburg type.
\newblock {\em New York J. Math}, 7:149--187, 2001.

\bibitem{kaufman1968hausdorff}
R.~Kaufman.
\newblock On {H}ausdorff dimension of projections.
\newblock {\em Mathematika}, 15(2):153--155, 1968.

\bibitem{keleti2019new}
Tam{\'a}s Keleti and Pablo Shmerkin.
\newblock New bounds on the dimensions of planar distance sets.
\newblock {\em Geometric and Functional Analysis}, 29(6):1886--1948, 2019.

\bibitem{liu2020hausdorff}
B.~Liu.
\newblock Hausdorff dimension of pinned distance sets and the ${L}^2$-method.
\newblock {\em Proceedings of the American Mathematical Society},
  148(1):333--341, 2020.

\bibitem{mattila1999geometry}
P.~Mattila.
\newblock {\em Geometry of sets and measures in {E}uclidean spaces: fractals
  and rectifiability}.
\newblock Number~44. Cambridge university press, 1999.

\bibitem{orponen2018radial}
T.~Orponen.
\newblock On the dimension and smoothness of radial projections.
\newblock {\em Analysis \& PDE}, 12(5):1273--1294, 2018.

\bibitem{orponen2021hausdorff}
T.~Orponen and P.~Shmerkin.
\newblock On the {H}ausdorff dimension of {F}urstenberg sets and orthogonal
  projections in the plane.
\newblock {\em arXiv preprint arXiv:2106.03338}, 2021.

\bibitem{orponen2022kaufman}
T.~Orponen, P.~Shmerkin, and H.~Wang.
\newblock Kaufman and {F}alconer estimates for radial projections and a
  continuum version of {B}eck's {T}heorem.
\newblock {\em arXiv preprint arXiv:2209.00348}, 2022.

\bibitem{orponen2020improved}
Tuomas Orponen.
\newblock An improved bound on the packing dimension of furstenberg sets in the
  plane.
\newblock {\em J. Eur. Math. Soc.(JEMS)}, 22(3):797--831, 2020.

\bibitem{orponen2023projections}
Tuomas Orponen and Pablo Shmerkin.
\newblock Projections, {F}urstenberg sets, and the $abc$ sum-product problem,
  2023.

\bibitem{ren2023furstenberg}
Kevin Ren and Hong Wang.
\newblock Furstenberg sets estimate in the plane, 2023.

\bibitem{ruzsa1978cardinality}
I.Z. Ruzsa.
\newblock On the cardinality of ${A}+ {A}$ and ${A}- {A}$.
\newblock In {\em Combinatorics (Proc. Fifth Hungarian Colloq., Keszthely,
  1976)}, volume~2, pages 933--938, 1978.

\bibitem{sheffer2016balog}
Adam Sheffer.
\newblock The {B}alog-{S}zemer{\'e}di-{G}owers theorem, 2016.

\bibitem{shmerkin2022non}
P.~Shmerkin.
\newblock A non-linear version of {B}ourgain’s projection theorem.
\newblock {\em Journal of the European Mathematical Society}, 2022.

\bibitem{shmerkin2021distance}
P.~Shmerkin and H.~Wang.
\newblock On the distance sets spanned by sets of dimension $ d/2$ in
  $\mathbb{R}^d$.
\newblock {\em arXiv preprint arXiv:2112.09044}, 2021.

\bibitem{shmerkin2022dimensions}
P.~Shmerkin and H.~Wang.
\newblock Dimensions of {F}urstenberg sets and an extension of {B}ourgain's
  projection theorem.
\newblock {\em arXiv preprint arXiv:2211.13363}, 2022.

\bibitem{shmerkin2021improved}
Pablo Shmerkin.
\newblock Improved bounds for the dimensions of planar distance sets.
\newblock {\em J. Fractal Geom.}, 8(1):27--51, 2021.

\bibitem{stull2022pinned}
D.M. Stull.
\newblock Pinned distance sets using effective dimension.
\newblock {\em arXiv preprint arXiv:2207.12501}, 2022.

\bibitem{sudakov2005question}
B.~Sudakov, E.~Szemer{\'e}di, and V.H. Vu.
\newblock On a question of {E}rd{\H{o}}s and {M}oser.
\newblock 2005.

\bibitem{tao2006additive}
T.~Tao and V.H. Vu.
\newblock {\em Additive combinatorics}, volume 105.
\newblock Cambridge University Press, 2006.

\bibitem{wang2022improvement}
Z.~Wang and J.~Zheng.
\newblock An improvement of the pinned distance set problem in even dimensions.
\newblock In {\em Colloquium Mathematicum}, volume 170, pages 171--191.
  Instytut Matematyczny Polskiej Akademii Nauk, 2022.

\bibitem{wolff1999decay}
T.~Wolff.
\newblock Decay of circular means of {F}ourier transforms of measures.
\newblock {\em Int. Math. Res. Not.}, 10:547--567, 1999.

\bibitem{zahl2022unions}
J.~Zahl.
\newblock Unions of lines in $\mathbb{R}^n$.
\newblock {\em arXiv preprint arXiv:2208.02913}, 2022.

\end{thebibliography}

\end{document}